\documentclass{amsart} 
\usepackage[utf8]{inputenc}
\usepackage[T1]{fontenc}
\usepackage[top=2.54cm,bottom=2cm,right=3.5cm,left=3.5cm]{geometry}
\usepackage{color}
\usepackage{indentfirst}
\usepackage{lmodern} \normalfont
\DeclareFontShape{T1}{lmr}{bx}{sc} { <-> ssub * cmr/bx/sc }{}
\usepackage{amsmath}
\usepackage{amsfonts}
\usepackage{derivative}
\usepackage{amssymb}
\usepackage[alphabetic]{amsrefs}

\usepackage{lipsum}                     
\usepackage{xargs}                      
\usepackage[pdftex,dvipsnames]{xcolor}

\usepackage{listings}
\usepackage{amsmath}
\usepackage{amsthm,amsfonts,mathrsfs,amsfonts}
\usepackage{slashed}
\usepackage{breqn}
\usepackage{pb-diagram}
\usepackage{bbm}
\usepackage[all]{xy}
\usepackage{times}
\usepackage{amsaddr}
\usepackage{microtype}
\usepackage{enumerate}
\usepackage[super]{nth}
\usepackage{multicol}
\usepackage{tikz-cd}
\usepackage{array}
\usepackage[colorinlistoftodos]{todonotes}

\usepackage{url}
\usepackage{verbatim}

\usepackage[misc]{ifsym}

\newcommand{\rG}{{\rm G}}

\newcommand{\rI}{{\rm I}}

\newcommand{\rN}{{\rm N}}

\newcommand{\rS}{{\rm S}}

\newcommand{\br}{{\bf r}}

\newcommand{\cB}{\mathcal{B}}

\newcommand{\sX}{\mathscr{X}}

\newcommand{\fg}{{\mathfrak g}}

\newcommand{\fh}{{\mathfrak h}}

\newcommand{\fp}{{\mathfrak p}}

\newcommand{\Z}{\mathbb{Z}}

\newcommand{\R}{\mathbb{R}}

\renewcommand{\H}{\mathbb{H}}

\newcommand{\bbS}{\mathbb{S}}
\newcommand{\bbRP}{\mathbb{RP}}

\newcommand{\so}{\mathfrak{so}}

\newcommand{\SO}{{\rm SO}}
\newcommand{\Sp}{{\rm Sp}}

\newcommand{\SU}{{\rm SU}}

\newcommand{\Ad}{\mathrm{Ad}}
\newcommand{\Aut}{\mathrm{Aut}}
\newcommand{\Inn}{\mathrm{Inn}}
\newcommand{\Out}{\mathrm{Out}}
\newcommand{\Crit}{\mathrm{Crit}}
\renewcommand{\det}{\mathop\mathrm{det}\nolimits}

\newcommand{\End}{{\mathrm{End}}}

\renewcommand{\epsilon}{\varepsilon}

\newcommand{\Lie}{\mathrm{Lie}}

\newcommand{\ad}{\mathrm{ad}}

\newcommand{\red}{{\rm red}}

\newcommand{\diag}{\mathrm{diag}}

\newcommand{\ind}{\mathop{\mathrm{index}^{\mathrm{red}}}}
\newcommand{\nulli}{\mathop{\mathrm{null}^{\mathrm{red}}}}

\newcommand{\tr}{\mathop{\mathrm{tr}}\nolimits}

\newcommand{\vol}{\mathrm{vol}}

\newcommand{\Sym}{\mathrm{Sym}}

\newcommand{\qandq}{\quad\text{and}\quad}
\newcommand{\qwithq}{\quad\text{with}\quad}

\def\<{\mathopen{}\left<}
\def\>{\right>\mathclose{}}
\def\({\mathopen{}\left(}
\def\){\right)\mathclose{}}

\usepackage{multicol, color}

\definecolor{gold}{rgb}{0.85,.66,0}
\definecolor{cherry}{rgb}{0.9,.1,.2}
\definecolor{burgundy}{rgb}{0.8,.2,.2}
\definecolor{orangered}{rgb}{0.85,.3,0}
\definecolor{orange}{rgb}{0.85,.4,0}
\definecolor{olive}{rgb}{.45,.4,0}
\definecolor{lime}{rgb}{.6,.9,0}
\definecolor{green}{rgb}{.2,.7,0}
\definecolor{grey}{rgb}{.4,.4,.2}
\definecolor{brown}{rgb}{.4,.3,.1}

\makeatletter
\newtheorem*{rep@theorem}{\rep@title}
\newcommand{\newreptheorem}[2]{%
\newenvironment{rep#1}[1]{%
 \def\rep@title{#2 \ref{##1}}%
 \begin{rep@theorem}}%
 {\end{rep@theorem}}}
\makeatother

\newtheorem{theorem}{Theorem}
\newreptheorem{theorem}{Theorem}
\newtheorem{lemma}{Lemma}[section]
\newtheorem{corollary}[lemma]{Corollary}
\newtheorem{definition}[lemma]{Definition}
\newtheorem{example}[lemma]{Example}
\newtheorem{remark}[lemma]{Remark}

\newtheorem*{remark*}{Remark}
\newtheorem*{theorem*}{Theorem}
\newtheorem{proposition}[lemma]{Proposition}

\newcommand{\gt}{\mathrm{G}_2}

\DeclareMathOperator{\Gl}{Gl}
\DeclareMathOperator{\Sl}{Sl}
\DeclareMathOperator{\sptc}{Sp}
\DeclareMathOperator{\real}{Re}

\DeclareMathOperator{\Or}{O}

\DeclareMathOperator{\imag}{Im}

\DeclareMathOperator{\Span}{span}
\DeclareMathOperator{\fraksp}{\mathfrak{sp}}
\DeclareMathOperator{\frakgl}{\mathfrak{gl}}
\DeclareMathOperator{\fraksl}{\mathfrak{sl}}
\DeclareMathOperator{\frakp}{\mathfrak{p}}

\DeclareMathOperator{\Der}{Der}
\DeclareMathOperator{\Div}{div}

\DeclareMathOperator{\curl}{curl}

\DeclareMathOperator{\diver}{div}
\DeclareMathOperator{\ricci}{Ric}

\DeclareMathOperator{\hess}{Hess}

\newcommand{\qforq}{\quad \text{for} \quad}

\newcommand{\qwhereq}{\quad \text{where} \quad}
\newcommand{\qorq}{\quad \text{or} \quad}

\newcommand\blfootnote[1]{%
  \begingroup
  \renewcommand\thefootnote{}\footnote{#1}%
  \addtocounter{footnote}{-1}%
  \endgroup
}

\newcommand{\Addresses}{{
  \bigskip
  \footnotesize

  \textsc{Univ. Brest, CNRS UMR 6205, LMBA, F-29238 Brest, France}\par\nopagebreak
  \quad\quad E.~Loubeau: \texttt{loubeau@univ-brest.fr}

  \bigskip

  \textsc{IMECC, University of Campinas, Campinas - São Paulo, Brazil}\par\nopagebreak
  \Letter \quad A.~Moreno: \texttt{andres.moreno@ime.unicamp.br} 
  
  \quad\quad H.~Sá Earp: \texttt{henrique.saearp@ime.unicamp.br}
   
   \quad\quad J.~Saavedra: \texttt{julieth.p.saavedra@gmail.com}

}}

\title{harmonic $\sptc(2)$-invariant $\gt$-structures on the $7$-sphere}
\author{\small Eric Loubeau, Andrés J. Moreno, Henrique N. Sá Earp \& Julieth Saavedra}

\begin{document}
\begin{abstract}
    We describe the $10$-dimensional space of $\rm{Sp}(2)$-invariant $\rm{G}_2$-structures on the homogeneous $7$-sphere $\mathbb{S}^7=\rm{Sp}(2)/\rm{Sp}(1)$ as $\Omega_+^3(\mathbb{S}^7)^{\rm{Sp}(2)}\simeq \mathbb{R}^+\times\mathrm{Gl}^+(3,\mathbb{R})$. 
    In those terms, we formulate a general Ansatz for $\rm{G}_2$-structures, which realises representatives in each of the $7$ possible isometric classes of homogeneous $\rm{G}_2$-structures.
    Moreover, the well-known  nearly parallel \emph{round} and \emph{squashed} metrics occur naturally as opposite poles in an $\mathbb{S}^3$-family, the equator of which is a new $\mathbb{S}^2$-family of coclosed  $\rm{G}_2$-structures satisfying the harmonicity condition $\mathrm{div}\; T=0$. 
    We show general existence of harmonic representatives of $\rm{G}_2$-structures in each isometric class through explicit solutions of the associated flow and describe the qualitative behaviour of the flow. We study the stability of the Dirichlet gradient flow near these critical points, showing explicit examples of degenerate and nondegenerate local maxima and minima, at various regimes of the general Ansatz. Finally, for metrics outside of the Ansatz, we identify families of harmonic $\rm{G}_2$-structures, prove long-time existence of the flow and study the stability properties of some well-chosen examples.\\
    
\end{abstract}

\maketitle


\tableofcontents

\blfootnote{\Addresses}

\newpage
\section*{Introduction}

This paper presents a detailed study of the space $\Omega_+^3(\bbS^7)^{\sptc(2)}$ of $\sptc(2)$-invariant $\gt$-structures on the homogeneous $7$-sphere $\sptc(2)/\sptc(1)$. $\rG_2$-Geometry on the sphere has attracted significant interest over the past decade, perhaps most notably in the classification of homogeneous structures by Reidegeld \cite{reidegeld2010spaces}, the description of $7$-dimensional homogeneous spaces with isotropy representation in $\gt$ by Lê \& Munir \cite{van2012}
and the study of the $7$-sphere's calibrated geometry by Lotay \cite{lotay2012associative} and Kawai \cite{kawai2015}.

While Reidegeld did establish, by infinitesimal arguments  \cite{reidegeld2010spaces}*{Theorem 1}, that $\Omega_+^3(\bbS^7)^{\sptc(2)}$ has dimension $10$, no explicit parametrisation has so far been proposed. In Lemma \ref{Lemma_invariant_G2_structures}, we identify, up to orientation, the $10$-manifold $\Omega_+^3(\bbS^7)^{\sptc(2)}$ with the product $\mathbb{R}^+\times\Gl^+(3,\mathbb{R})$. 
Since $\bbS^7$ is spinnable, the map \cite{karigiannis2009flows}*{eq. (11)} 
$$
 B: \varphi\in \Omega_+^3(\bbS^7)\mapsto g_\varphi\in \Sym_+^2(T^\ast \bbS^7),
$$
associating a Riemmanian metric to  each $\gt$-structure, is surjective. By reduction of the principal $\SO(7)$-bundle to a principal $\gt$-bundle \cite{bryant2003some}*{Remark 3}, the restriction $B|_{\Omega_+^3(\bbS^7)^{\sptc(2)}}$ maps onto the set $\Sym_+^2(T^\ast \bbS^7)^{\sptc(2)}$ of $\sptc(2)$-invariant Riemannian metrics. Each preimage subset $\mathcal{B}_\varphi:=B^{-1}(g_{\varphi})^{\sptc(2)}$, consisting of the homogeneous  $\gt$-structures defining the same Riemannian metric,  is called the \emph{($\sptc(2)$-invariant) isometric class of $\varphi$}. 

According to the \emph{reductive decomposition} $\fraksp(2)=\fraksp(1)\oplus\frakp$,  where $\Ad(\sptc(1))\frakp \subset \frakp$, the tangent space of $\bbS^7$ at the identity class $o=1_{\sptc(2)}\sptc(1)$ is identified with the $\Ad(\sptc(1))$-invariant complement vector space $\frakp$. Moreover, each $\sptc(2)$-invariant $\gt$-structure on $\bbS^7$ is determined by a $\gt$-structure on $\frakp\simeq T_o\bbS^7$ which is invariant by the $\Ad(\sptc(1))$-action, i.e. $\Omega_+^3(\bbS^7)^{\sptc(2)}\simeq \Lambda_+^3(\frakp^\ast)^{\Ad(\sptc(1))}$. 
From the identification $\Omega_+^3(\bbS^7)^{\sptc(2)}\simeq \mathbb{R}^+\times\Gl^+(3,\mathbb{R})$, each such isometric class is described by a $\SO(3)$-orbit (Proposition \ref{prop: isometric_equivalence}): 
\begin{equation*}
    \Omega_+^3(\bbS^7)^{\sptc(2)}\ni \varphi \longleftrightarrow
    (a,D)\in  \mathbb{R}^+\times \Gl^+(3,\mathbb{R}),
    \qwithq
    \mathcal{B}_\varphi\simeq (a, D\cdot \SO(3)), 
\end{equation*}
for each $g_\varphi\in \Sym_+^2(T^\ast \bbS^7)^{\sptc(2)}$. Here, the $\SO(3)$-orbit represents the $\bbRP^3$-subbundle of isometric $\sptc(2)$-invariant $\gt$-structures within the $\bbRP^7$-family of isometric $\gt$-structures \cite{bryant2003some}*{Remark 4}. In \S\ref{homogeneous_space}, we prove:
\begin{theorem}
\label{Th: Phi_isomorphism_theorem}
    The space of $\sptc(2)$-invariant $\gt$-structures on $\bbS^7\simeq\sptc(2)/\sptc(1)$ is described by the homogeneous manifold
$$
    \Omega_+^3(\bbS^7)^{\sptc(2)}\simeq \mathbb{R}^+\times \Gl^+(3,\mathbb{R}),
$$
    via the isomorphism  
    $\Omega_+^3(\bbS^7)^{\sptc(2)}\simeq \Lambda_+^3(\frakp^\ast)^{\Ad(\sptc(1))}$ and the map
\begin{align*}
    \Phi: \quad 
    \mathbb{R}^+\times \Gl^+(3,\mathbb{R})
    &\to 
    \Lambda_+^3(\frakp^\ast)^{\Ad(\sptc(1))}\\
    (a,D)
    &\mapsto \Phi(a,D)=\varphi_{a,D} .
\end{align*}
    The corresponding space of $\sptc(2)$-invariant $\gt$-metrics is
\begin{equation}
\label{metric_homogeneous_space}
    \Sym_+^2(T^\ast \bbS^7)^{\sptc(2)}\simeq \mathbb{R}^+\times \Big(\Gl^+(3,\mathbb{R})/\SO(3)\Big).
\end{equation}
    In particular, the moduli space of $\sptc(2)$-invariant Riemannian metrics described in \cite{ziller1982} corresponds to the $4$-manifold   
\begin{equation}\label{eq: classes of homogeneous metrics}
   \Aut(\bbS^7)\backslash\Sym_+^2(T^\ast \bbS^7)^{\sptc(2)}\simeq\mathbb{R}^+\times \SO(3)\backslash\Big(\Gl^+(3,\mathbb{R})/\SO(3)\Big),
\end{equation}
where $\Aut(\bbS^7)$ denotes the group of $\Sp(2)$-equivariant diffeomorphisms $f$ of $\bbS^7$ (i.e. automorphisms of $\sptc(2)$, preserving the subgroup $\sptc(1)$) such that, if $g\in \Sym_+^2(T^\ast \bbS^7)^{\sptc(2)}$ then $(f^{-1})^\ast g$ is an $\sptc(2)$-invariant metric.
\end{theorem}
Notice that the $\SO(3)$-classes in \eqref{metric_homogeneous_space} represent the $\gt$-metrics induced by the $\bbRP^3$-isometric class of $\sptc(2)$-invariant $\gt$-structures, and the left $\SO(3)$-action in \eqref{eq: classes of homogeneous metrics} sits inside the $\sptc(2)$-action, cf. \S\ref{Subsec: Ziller_remark}.
Indeed, we realise the dimensions of the spaces of $\sptc(2)$-invariant $\gt$-structures and $\sptc(2)$-invariant metrics,  first computed by F. Reidegeld \cite{reidegeld2010spaces}, respectively as 
$$
n_{\gt}=10=\dim\Big( \mathbb{R}^+\times \Gl^+(3,\mathbb{R})\Big) 
\qandq
n_{\Or(7)}=7=\dim\Big( \mathbb{R}^+\times\Gl^+(3,\mathbb{R})/\SO(3)\Big).
$$
As a curious application of Theorem \ref{Th: Phi_isomorphism_theorem}, in \S\ref{Subsec: G2_structure_of_the_space_of_metric}  we endow the $7$-manifold of homogeneous $\gt$-metrics  $\Sym^2(\frakp)^{\Ad(\sptc(2))}$ itself with a natural left-invariant $\gt$-structure, using the identification \eqref{metric_homogeneous_space} and the global decomposition of noncompact semi-simple Lie groups.

We know from \cite{ziller1982} that, up to isometry, any $\sptc(2)$-invariant metric is a multiple of the inner product expressed in terms of the basis $e^1,\dots,e^7\in \frakp^\ast$ by
\begin{equation}
\label{Eq: isometric:G2-metric_introduction}
    g_\br=\frac{1}{r_1^6}(e^1)^2+\frac{1}{r_2^6}(e^2)^2+\frac{1}{r_3^6}(e^3)^2+r_1r_2r_3\Big((e^4)^2+(e^5)^2+(e^6)^2+(e^7)^2\Big),
\end{equation}
with $r_1,\dots,r_3>0$ and $\br:=(r_1,r_2,r_3)$. The corresponding isometric class of $\gt$-structures is parametrised by
\begin{align}
\label{eq: isometric:G2-str_intro}
\begin{split}
    \varphi_{(\br,h)}
    =&\;\frac{1}{(r_1r_2r_3)^3}e^{123}\\
    &+\Big(\frac{r_2r_3}{r_1^2}(h_0^2+h_1^2-h_2^2-h_3^2)e^1 +2\frac{r_1r_3}{r_2^2}(h_1h_2-h_0h_3)e^2 +2\frac{r_1r_2}{r_3^2}(h_1h_3+h_0h_2)e^3\Big) \wedge\omega_1\\
    &+\Big(2\frac{r_2r_3}{r_1^2}(h_1h_2+h_0h_3)e^1 +\frac{r_1r_3}{r_2^2}(h_0^2-h_1^2+h_2^2-h_3^2)e^2 +2\frac{r_1r_2}{r_3^2}(h_2h_3-h_0h_1)e^3\Big) \wedge\omega_2\\
    &+\Big(2\frac{r_2r_3}{r_1^2}(h_1h_3-h_0h_2)e^1 +2\frac{r_1r_3}{r_2^2}(h_2h_3+h_0h_1)e^2 +\frac{r_1r_2}{r_3^2}(h_0^2-h_1^2-h_2^2+h_3^2)e^3\Big) \wedge\omega_3,
\end{split}
\end{align}
with $(h_0,...,h_3)\in\H$ a unit quaternion parametrising a $\SO(3)$-transformation  
and $\omega_1,\dots,\omega_3\in\Omega^2(\bbS^7)^{\sptc(2)}$ as in \eqref{eq: 7-family_g2_structures} below.
 The $\gt$-structures of the form \eqref{eq: isometric:G2-str_intro} can be expressed (see Lemma \ref{lemma: (f,X)_invariant}), according to Bryant's description  of isometric $\gt$-structures as sections of a $\R P^7$-bundle \cite{bryant2003some}*{Remark 4}, in the form
\begin{equation*}
    \varphi_{(\br,h)}:=\varphi_{(h_0,X)}=(h_0^2-|X|^2)\varphi_{(\br,1)}-2h_0(X\lrcorner\psi_{(\br,1)})+2X^\flat\wedge(X\lrcorner\varphi_{(\br,1)}), 
    \qforq h_0^2+|X|^2=1,
\end{equation*}
in terms of a left-invariant vector field 
$$
X=r_1^3h_1e_1+r_2^3h_2e_2+r_3^3h_3e_3\in \sX(\bbS^7)^{\sptc(2)}, 
\qwithq
|X|^2=g_\br(X,X)=1-h_0^2.
$$
In \S\ref{Sec:Isometric_classes}, we adopt a convenient Ansatz in the general expression \eqref{eq: isometric:G2-str_intro}, which suffices to realise, in the same $4$-dimensional half-cylindrical family over $\bbRP^3=\SO(3)$, several well-known torsion classes of  $\gt$-structures:

\begin{theorem}
\label{Th: isometric_families_theorem}
    Let  $r_1=r_2=r_3=r^{-1/3}$ in \eqref{eq: isometric:G2-str_intro}. Then each $\bbRP^3$-family of $\gt$-structures
$$
  \cB_r:=\Phi(\{r\}\times\SO(3))=B^{-1}(g_r)^{\sptc(2)}
  \simeq \bbRP^3\simeq\bbS^3/\Z_2
$$
 determines a distinct isometric class. Moreover, in terms of the equator and poles
$$
  \bbRP^2\simeq \{(h_0,h_1,0,h_3)\in \bbS^3\}/\Z_2 
  \qandq 
  \rN\rS=[(0,0,\pm 1,0)], 
$$
we characterise the following torsion regimes in each isometric class $\cB_r$, up to the diffeomorphism $\Phi$:
\begin{enumerate}[(i)]
      \item The coclosed $\gt$-structures correspond to $\{r\}\times \bbRP^2$ and $\{r\}\times \rN\rS$.
      \item The nearly parallel $\gt$-structures correspond to $\{\sqrt[3]{2}\}\times \bbRP^2$ (round) and $\{\sqrt[3]{2/5}\}\times \rN\rS$ (squashed).
      \item
     The locally conformal coclosed $\gt$-structures correspond to
     $\{1\}\times \bbRP^3$.
  \end{enumerate}
  Furthermore, there is no  purely coclosed structure in $\mathcal{B}_r$.
\end{theorem}

In particular, Theorem \ref{Th: isometric_families_theorem} recovers previously known facts about nearly parallel $\gt$-structures obtained from the canonical variation of a Riemannian submersion \cite{friedrich1997nearly}*{Theorem 5.4}. Theorem \ref{Th: isometric_families_theorem} identifies these cases as distinguished points in a same cylindrical family.

Within a given isometric class $\mathcal{B}_r$,  the $\gt$-structures with divergence-free torsion $T$ were first highlighted by Grigorian \cite{grigorian2017}, in relation to a Coulomb gauge condition for octonionic connections. Their interpretation as stationary points for isometric flows of $\rG_2$-structures was thoroughly explored in  \cites{dwivedi2019,Grigorian2019}, and subsequently their characterisation in terms of harmonicity was systematised in \cite{loubeau2019}. In \S\ref{Sec: Harmonicity_and_stability}, we describe the $\sptc(2)$-invariant $\gt$-structures of $\mathcal{B}_r$ with $\diver T=0$ (Proposition \ref{prop: divergence}), i.e. critical points of the Dirichlet energy functional on this class:
\begin{align*}
    E: \varphi \in\mathcal{B}_r \mapsto \int_M|T_\varphi|^2\vol_\varphi\in \mathbb{R}.
\end{align*}
More generally, we will make the blanket assumption that all the variational problems we consider are restricted to a given isometric class.
In \S\ref{sec: divergence}, we prove:
\begin{theorem}
\label{Th: critical_points_theorem}
    In the context of Theorem \ref{Th: isometric_families_theorem}, for a given $r>0$, the set $\Crit(E|_{\cB_r})$ is described as follows:\\
    \underline{$r\neq 1$}: the harmonic $\gt$-structures in the isometric class $\cB_r$ are precisely those  parametrised by the equator $\{r\}\times\bbRP^2$ or the poles $\{r\}\times\mathrm{NS}$.\\
    \underline{$r=1$}: all the $\gt$-structures $\varphi_{(1,h)}$, with $h\in \bbS^3$, are harmonic.
\end{theorem}

In \S\ref{Sec:2nd_variation}, we examine the energy functional and its critical points with respect to \emph{homogeneous variations}. We introduce the \emph{reduced energy} functional,  restricted to isometric classes of $\gt$-structures:  
$$
E^{\red} := E|_{\Omega_+^3(\bbS^7)^{\sptc(2)}},
$$ 
In order to quantify the stability properties of this reduced energy functional at a given critical point, we count the $\Ad(\sptc(1))$-invariant deformations yielding a local maximum or a stationary point, which naturally leads to the concepts of \emph{reduced index} $\ind$ and \emph{reduced nullity} $\nulli$, cf. Definition~\ref{def: index_nullity}. The stability profile of critical regions changes qualitatively with $r>0$:

\begin{theorem}
\label{Th: reduced_energy_stability_theorem}
    In the context of Theorem \ref{Th: isometric_families_theorem}, given $r>0$, the critical points of $E^{\red}$ have the following reduced index and nullity:  
$$\begin{array}{r|lccl}
    &\text{region} & \ind & \nulli &  \text{stability profile}\\\hline
    r<1 & \{r\}\times \rN\rS & 0 & 0 & \text{stable: local minima}\\
    & \{r\}\times \bbRP^2 & 1 & 2 & \text{unstable}\\\hline
    r>1 & \{r\}\times \rN\rS & 3 & 0 & \text{unstable}\\
    & \{r\}\times \bbRP^2 & 0 & 2 & \text{stable: local minima}
\end{array}$$
Finally, when $r=1$, $E^\red$ is constant on the corresponding isometric class, so all of its points are trivially stable.
    
    
\end{theorem}

As to the stability of $E|_{\cB_r}$,  the reduced energy functional provides lower bounds on the index, over some connected components of the critical set. As a consequence, the stability profile of harmonic $\sptc(2)$-invariant $\gt$-structures changes qualitatively according to the values of $r\in \mathbb{R}^+$: 

\begin{theorem}
\label{Th: index_nullity_estimates_theorem}
    In the context of Theorems \ref{Th: isometric_families_theorem} and \ref{Th: critical_points_theorem}, for a given $r>0$, the index and nullity of points in $\Crit(E|_{\cB_r})$ are subject to the following lower bounds: \\
    \underline{$r <1$}: harmonic $\gt$-structure on $\{r\}\times \bbRP^2$
    have index at least $1$ and hence, they are unstable.\\\underline{$r >1$}: harmonic $\gt$-structures on
 $\{r\}\times \rN\rS$    
    has index at least $3$ and hence, it is unstable.
\end{theorem}

Furthermore, for invariant initial data, the isometric flow reduces to an ODE, which we can solve explicitly (Proposition~\ref{prop: ode_explicit_solution}). From this, we deduce long-time existence of the isometric flow, convergence towards harmonic $\gt$-structures and an asymptotic study of flow lines between the equatorial $\bbRP^2$ and $\rN\rS$ on the three-dimensional real projective space (Theorem~\ref{th: convergent_subseq}). One of the consequences of this description is that the isometric flow forever stays in the hemisphere of the initial data.

In the last section, we drop the Ansatz from \S2 and carry out this programme for arbitrary $\sptc(2)$-invariant metrics. We classify
harmonic $\gt$-structures (Proposition~\ref{thm critical points - gen case}), verify again the long-time existence of the isometric flow and study the stability profile for some families of harmonic $\gt$-structures in this general setting (Example~\ref{Th: seventh_main_theorem}).

\noindent\textbf{Acknowledgements:}
We thank the anonymous referees for the useful comments and remarks in the earlier version of this paper. The present article stems from an ongoing CAPES-COFECUB bilateral collaboration (2018-2021), granted by Brazilian Coordination for the Improvement of Higher Education Personnel (CAPES) – Finance Code 001 [88881.143017/2017-01], and COFECUB [MA 898/18]. HSE has also been funded by the São Paulo Research Foundation (Fapesp)  \mbox{[2017/20007-0]} \& \mbox{[2018/21391-1]} and the Brazilian National Council for Scientific and Technological Development (CNPq)  \mbox{[307217/2017-5]}.

\section{Invariant $\gt$-structures of  $\sptc(2)/\sptc(1)$}
\label{homogeneous_space}

We begin with a reductive decomposition of $\bbS^7=\sptc(2)/\sptc(1)$ as a homogeneous space. The main goal of this section is a description of the space of $\sptc(2)$-invariant $\gt$-structures as a homogeneous manifold itself (Theorem \ref{Th: Phi_isomorphism_theorem}). The main references for this section are \cites{ziller1982,reidegeld2010spaces,van2012}.

\subsection{General form for homogeneous $\rG_2$-structures on $\sptc(2)/\sptc(1)$}

Let $A^\ast=\bar{A}^T$ be the Hermitian transpose, consider the Lie algebra 
$$
  \fraksp(2):=\{A\in\frakgl(2,\mathbb{H}) \mid A+A^\ast=0\},
$$
and fix its basis 
\begin{align}\label{Eq:sp(2)_basis}
\begin{split}
    v_1&=\left(
    \begin{array}{cc}
        i & 0 \\
        0 & 0
    \end{array}
    \right), \quad v_2=\left(
    \begin{array}{cc}
        j & 0 \\
        0 & 0
    \end{array}
    \right), \quad  v_3=\left(
    \begin{array}{cc}
        k & 0 \\
        0 & 0
    \end{array}
    \right), \quad e_4=\frac{1}{\sqrt{2}}\left(
    \begin{array}{cc}
        0 & i \\
        i & 0
    \end{array}
    \right), \quad e_5=\frac{1}{\sqrt{2}}\left(
    \begin{array}{cc}
        0 & j \\
        j & 0
    \end{array}
    \right),  \\
    e_1&=\left(
    \begin{array}{cc}
        0 & 0 \\
        0 & i
    \end{array}
    \right), \quad  e_2=\left(
    \begin{array}{cc}
        0 & 0 \\
        0 & j
    \end{array}
    \right), \quad e_3=\left(
    \begin{array}{cc}
        0 & 0 \\
        0 & k
    \end{array}
    \right), \quad e_6=\frac{1}{\sqrt{2}}\left(
    \begin{array}{cc}
        0 & k \\
        k & 0
    \end{array}
    \right), \quad e_7=\frac{1}{\sqrt{2}}\left(
    \begin{array}{cc}
        0 & 1 \\
        -1 & 0
    \end{array}
    \right).
    \end{split}
\end{align}
Define the embedding $\fraksp(1)\subset \fraksp(2)$ as the subalgebra generated by $\{v_1,v_2,v_3\}$, corresponding to the reductive splitting 
\begin{equation}\label{eq: reductive_decomposition}
\fraksp(2)=\fraksp(1)\oplus\frakp  \qwithq \frakp:=\fraksp(1)^\perp=\Span(e_1,...,e_7),
\end{equation} 
with respect to the canonical inner product $\langle A_1,A_2\rangle=\real(\tr(A_1A_2^\ast))$. A straightforward computation shows that
\begin{align*}
    \ad(\fraksp(1))e_l&=0, \qforq l=1,2,3,\\ \ad(\fraksp(1))\frakp_4&=\frakp_4, \qwithq \frakp_4:=\Span(e_4,e_5,e_6,e_7).
\end{align*}
In terms of the trivial submodules $\frakp_l=\Span(e_l)$, we have the irreducible $\ad(\fraksp(1))$-decomposition $\frakp=\frakp_1\oplus \frakp_2\oplus \frakp_3\oplus \frakp_4$. Consider the $2$-forms on $\frakp_4$  
\begin{equation}\label{Eq: self-dual_2-forms}
  \omega_1=e^{47}+e^{56}, \quad \omega_2=e^{46}-e^{57}, \quad \omega_3=e^{45}+e^{67},
\end{equation}
and notice that $\omega_1^2=\omega_2^2=\omega_3^2=2e^{4567}$ and $\omega_l\wedge\omega_m=0$, for any $l\neq m$.

Recall that we denote by $\Omega_+^3(\bbS^7)^{\sptc(2)}$  the bundle  of $\sptc(2)$-invariant $\gt$-structures on $\bbS^7$, and by $\Sym^2_+(T^\ast \bbS^7)^{\sptc(2)}$  the bundle  of $\sptc(2)$-invariant $\gt$-metrics.
Every $\sptc(2)$-invariant $\gt$-structure on $\sptc(2)/\sptc(1)$ is determined by a $3$-form $\varphi$ on $\fp=T_0\bbS^7$, which is $\Ad(\sptc(1))$-invariant, therefore $\Omega_+^3(\bbS^7)^{\sptc(2)}\simeq \Lambda^3_+(\frakp^\ast)^{\Ad(\sptc(1))}$. On the other hand, by \cite{bryant1987}*{Proposition 2} the $\Gl(\frakp)$-orbit of a $\gt$-structure $\varphi$ is open in $\Lambda_+^3(\frakp^\ast)$, consequently:
$$
 \Lambda^3_+(\frakp^\ast)^{\Ad(\sptc(1))}\simeq \Big(\Gl(\frakp)\cdot\varphi\Big)^{\Ad(\sptc(1))},
$$
where the left $\Gl(\frakp)$-action is defined by $\Phi\cdot\varphi:=(\Phi^{-1})^\ast\varphi$ for the $\Ad(\sptc(1))$-invariant morphism $\Phi\in \Gl(\frakp)$ and the $3$-form $\varphi\in \Lambda_+^3(\frakp^\ast)^{\Ad(\sptc(1))}$.
Looking into $\Gl(\frakp)^{\Ad(\sptc(1))}$, the splitting \eqref{eq: reductive_decomposition} establishes the inclusion
\begin{equation*}
    h\in\sptc(1)\mapsto \left(\begin{array}{c|c}
        h &  \\ \hline
         & 1
    \end{array}\right)\in \sptc(2),
\end{equation*}
in such a way that, for any $(x,y)\in\frakp$ with $x\in \frakp_1\oplus\frakp_2\oplus\frakp_3\simeq\fraksp(1)$ and $y\in\frakp_4\simeq\mathbb{H}$, we have $\Ad(h)(x,y)=(x,hy)$. The equivariance condition $\Ad(h)\circ\Phi=\Phi\circ\Ad(h)$, for each $h\in\sptc(1)$, implies
\begin{equation}\label{eq: general_Phi}
  \Phi=\left(\begin{array}{c|c}
      D_1 &  \\ \hline
        & D_2 
  \end{array}\right) \qwithq D_1\in\Gl(\frakp_1\oplus\frakp_2\oplus\frakp_3) \qandq D_2\in C_{\Gl(\frakp_4)}(\sptc(1)),
\end{equation}
where $C_{\Gl(\frakp_4)}(\sptc(1))$ denotes the centralizer of $\Ad(\sptc(1))|_{\frakp_4}$ inside $\Gl(\frakp_4)$ given by: $$
D_2(y)=y\bar{m} \qforq y\in\frakp_4 \qandq \bar{m}\in \H\setminus\{0\}.
$$
Furthermore, the orbit $\Big(\Gl(\frakp)\cdot\varphi\Big)^{\Ad(\sptc(1))}$ is diffeomorphic to $\Big(\Gl(\frakp)/\rG_\varphi\Big)^{\Ad(\sptc(1))}$, where $\rG_\varphi\simeq \gt$ is the stabiliser of $\varphi$. Thus, the matrix \eqref{eq: general_Phi} can be rewritten as
\begin{equation*}
 \Phi=\left(\begin{array}{c|c}
      D_1\Upsilon(\bar{n}) &  \\ \hline
        & |m|^2I_{4\times 4} 
  \end{array}\right)\Ad(1,n)_\frakp,
\end{equation*}
where $m=|m|n$, with $n\in \sptc(1)$, and 
\begin{equation*}
   \Ad(1,n)_\frakp(x,y):= \Ad\left(\left(\begin{array}{c|c}
      1 &  \\ \hline
        & n 
  \end{array}\right)\right)_{\frakp}(x,y)=(nx\bar{n},y\bar{n})=(\Upsilon(n)x,y\bar{n}).
\end{equation*}
On the other hand, the map $\Ad(1,n)$ lies in $\gt\simeq \rG_\varphi$ \cite{harvey1982calibrated}*{Chapter IV, (1.9)}, therefore $\Phi\cdot\varphi=\widetilde{\Phi}\cdot\varphi$, where
\begin{equation*}
 \widetilde{\Phi}=\left(\begin{array}{c|c}
      D &  \\ \hline
        & aI_{4\times 4} 
  \end{array}\right) \qwithq D=\pm D_1\Upsilon(\bar{n}) \qandq a=\pm |m|.
\end{equation*}
We highlight that $\Lambda^3(\frakp^\ast)$ has only two open $\Gl(\frakp)$-orbits, generated by a $\gt$-structure and a $\widetilde{\rG}_2$-structure, respectively  \cite{bryant1987}*{Definition 2 \& Proposition 2}. Likewise, $\Lambda^3(\frakp^\ast)^{\Ad(\sptc(1))} $ has two open orbits induced by open subsets of $(\Gl(\frakp)\cdot\varphi)^{\Ad(\sptc(1))}\simeq \Gl(\frakp_1\oplus\frakp_2\oplus\frakp_3)\times\mathbb{R}\smallsetminus\{0\}$, corresponding to the orbit of a $\Ad(\sptc(1))$-invariant $\gt$-structure or a $\Ad(\sptc(1))$-invariant $\widetilde{\rG}_2$-structure. 

In accordance with \cite{reidegeld2010spaces}, we denote by $n_{\gt}$  the rank of the bundle $\Omega_+^3(\bbS^7)^{\sptc(2)}$, and by $n_{\Or(7)}$ the rank of $\Sym^2_+(T^\ast \bbS^7)^{\sptc(2)}$. The following Lemma contains an alternative proof of Reidegeld's assertion that $n_{\gt}=10$ and $n_{\Or(7)}=7$. It provides moreover an explicit description of all $\sptc(2)$-invariant $\gt$-structures and their induced $\sptc(2)$-invariant metrics, in which $\frakp$ is identified with the tangent space of $\bbS^7$ at the orbit of the identity $o=1_{\sptc(2)}\sptc(1)$. Furthermore, it distinguishes the open subset in $\Gl(\frakp)^{\Ad(\sptc(1))}$ which parameterises $\Lambda_+^3(\frakp^\ast)^{\Ad(\sptc(1))}$.

\begin{lemma}
\label{Lemma_invariant_G2_structures}
    Consider the homogeneous space $\sptc(2)/\sptc(1)$ with the reductive decomposition \eqref{eq: reductive_decomposition}. The $\Ad(\sptc(1))$-invariant $\gt$-structures on $\frakp$ have the form
\begin{align}
\label{Eq:ad-invariant_G2-structure}
    \varphi=a^3e^{123}+ (\alpha_1e^1+\alpha_2e^2+\alpha_3e^3)\wedge \omega_1+(\beta_1e^1+\beta_2e^2+\beta_3e^3)\wedge \omega_2
    +(\gamma_1e^1+\gamma_2e^2+\gamma_3e^3)\wedge \omega_3,
\end{align}
    the coefficients of which satisfy 
$$
a\cdot\det D^{-1} >0,
\qwithq
D^{-1}:=\left(
   \begin{array}{ccc}
       \alpha_1 & \alpha_2 & \alpha_3 \\
      \beta_1 & \beta_2 & \beta_3 \\
      \gamma_1 & \gamma_2 & \gamma_3
   \end{array}
\right):\frakp_1\oplus \frakp_2\oplus \frakp_3\rightarrow \frakp_1\oplus \frakp_2\oplus \frakp _3.
$$
Moreover, the $\Ad(\sptc(1))$-invariant inner product on $\frakp$, induced by \eqref{Eq:ad-invariant_G2-structure}, is given by
\begin{equation}
\label{Eq:ad-invariant_G2-metric}
    g=\frac{a^2}{\det(D^{-1})^{2/3}}\left(
   \begin{array}{c|c}
     (DD^t)^{-1} & 0 \\ \hline
         & \frac{\det(D^{-1})}{a^3}I_{4\times 4}
   \end{array}
\right),
\end{equation}
and the induced orientation is parametrised by $a\in \R\smallsetminus\{0\}$. In particular, $n_{\gt}=10$ and $n_{\Or(7)}=7$.
\end{lemma}

\begin{proof}
According to the decomposition $\frakp=\frakp_1\oplus\frakp_2\oplus\frakp_3\oplus\frakp_4$, we have the isomorphism
\begin{equation}\label{eq: graded_isomorphism}
    \Lambda^3(\frakp^\ast)\simeq \bigoplus_{m+n=3}\Lambda^m(\frakp_1\oplus\frakp_2\oplus\frakp_3)^\ast\otimes\Lambda^n(\frakp_4^\ast).
\end{equation}
Note that $\Lambda^1(\frakp_4^\ast)\simeq\Lambda^3(\frakp_4^\ast)$ is not $\Ad(\sptc(1))$-invariant since $\frakp_4$ is irreducible. Also, each $\frakp_l$ (for $l=1,2,3$) is trivially $\Ad(\sptc(1))$-invariant, thus  $\Lambda^m(\frakp_1\oplus\frakp_2\oplus\frakp_3)^\ast$ is $\Ad(\sptc(1))$-invariant for $m=1,2,3$.
Taking the $\Ad(\sptc(1))$-invariant component on each side of \eqref{eq: graded_isomorphism}, we obtain  
\begin{equation*}\label{eq: model_of_invariant_G2_structures}
    \Lambda^3(\frakp^\ast)^{\Ad(\sptc(1))}\simeq \Lambda^3(\frakp_1\oplus\frakp_2\oplus\frakp_3)^\ast\bigoplus\Big(\Lambda^1(\frakp_1\oplus\frakp_2\oplus\frakp_3)^\ast\otimes\Lambda^2(\frakp_4^\ast)^{\Ad(\sptc(1))}\Big).
\end{equation*}
Using the decomposition $\Lambda^2(\frakp_4^\ast)=\Lambda^2_+\oplus \Lambda^2_-$ into self- and anti-self-dual $2$-forms on $\frakp_4$, we have that $\Lambda^2(\frakp_4^\ast)^{\Ad(\sptc(1))}\simeq\Lambda^2_+$, since $\Ad(\sptc(1))$ acts on $\frakp_4$ by left-quaternionic multiplication and $\Lambda^2_+\simeq \imag\H$ is in correspondence with the right multiplication of $\imag\H$ on $\frakp_4$. Hence, in terms of the basis \eqref{Eq: self-dual_2-forms}, a $\Ad(\sptc(1))$-invariant $3$-form reads as

$$
  \varphi=a^3e^{123}+ (\alpha_1e^1+\alpha_2e^2+\alpha_3e^3)\wedge \omega_1+(\beta_1e^1+\beta_2e^2+\beta_3e^3)\wedge \omega_2
    +(\gamma_1e^1+\gamma_2e^2+\gamma_3e^3)\wedge \omega_3.
$$
To be a well-defined $\gt$-structure, the $3$-form $\varphi$ must induce a symmetric definite bilinear form $B:\frakp\times \frakp\rightarrow \mathbb{R}$, defined by
\begin{equation}
    \label{eq: induced G2-metric}
    B_{ij}:=B(e_i,e_j)=\left((e_i\lrcorner\varphi)\wedge(e_j\lrcorner\varphi)\wedge\varphi\right)(e_1,e_2,...,e_7).
\end{equation}
By a long and straightforward computation, we get
\begin{align*}
    B_{11}&=6a^3(\alpha_1^2+\beta_1^2+\gamma_1^2) & B_{12}=&B_{21}=6a^3(\alpha_1\alpha_2+\beta_1\beta_2+\gamma_1\gamma_2)\\
    B_{22}&=6a^3(\beta_2^2+\beta_2^2+\gamma_2^2) & B_{23}=&B_{32}=6a^3(\alpha_2\alpha_3+\beta_2\beta_3+\gamma_2\gamma_3)\\
    B_{33}&=6a^3(\alpha_3^2+\beta_3^2+\gamma_3^2) & B_{13}=&B_{31}=6a^3(\alpha_1\alpha_3+\beta_1\beta_3+\gamma_1\gamma_3)
\end{align*}
\begin{align*}
    B_{kk}&=6\alpha_1(\beta_2\gamma_3-\beta_3\gamma_2)-6\alpha_2(\beta_1\gamma_3-\beta_3\gamma_1)+6\alpha_3(\beta_1\gamma_2-\beta_2\gamma_1) \qforq k=4,5,6,7,\\
    B_{ij}&=0 \qquad\text{otherwise.}
\end{align*}

Assembling the coefficients $\alpha_k,\beta_k,\gamma_k$ ($k=1,2,3$), we construct the matrix
$$
D^{-1}=\left(
   \begin{array}{ccc}
      \alpha_1 & \alpha_2 & \alpha_3 \\
      \beta_1 & \beta_2 & \beta_3 \\
      \gamma_1 & \gamma_2 & \gamma_3
   \end{array}
\right).
$$
Expressing the coefficients of \eqref{eq: induced G2-metric} as
\begin{equation}
\label{eq: B_ij_matrix}
\begin{array}{rcll}
    B_{ij}&=&6a^3[(DD^{t})^{-1}]_{ij}, & i,j=1,2,3,\\
    B_{kk}&=&6\det (D^{-1}), & k=4,5,6,7,
\end{array}
\end{equation}
we see that $\det(B)=6^7a^9(\det(D^{-1}))^6$. For the definiteness of \eqref{eq: induced G2-metric}, we have 
\begin{eqnarray}
\label{eq: definiteness_of_B}
    B_{jj}\in \R_{\pm}, \quad \forall j\in\{1,...,7\}
    &\Longleftrightarrow& a,\det(D^{-1}) \in \R_{\pm}.
\end{eqnarray}

For the second part of the Lemma, recall that the inner product on $\frakp$ induced by $\varphi$ is given by (see \cite{karigiannis2009flows})
$$
  g_{ij}=\frac{1}{6^{2/9}}\frac{B_{ij}}{\det(B)^{1/9}}.
$$
Condition \eqref{eq: definiteness_of_B} then guarantees the positive-definiteness of the inner product, which indeed takes the form \eqref{Eq:ad-invariant_G2-metric}, by \eqref{eq: B_ij_matrix}. Finally, the orientation  $\sqrt{\det(g)}=6^{-7/9}\det(B)^{1/9}=a(\det(D^{-1}))^{2/3}$ is determined by the sign of $a\in \mathbb{R}\smallsetminus \{0\}$.
\end{proof}

We fix henceforth an orientation by setting $a>0$. Notice that   $\Ad(\sptc(1))$-invariant $\gt$-structures are parametrised  by $G:=\mathbb {R}^+\times \Gl^+(3,\mathbb{R})$ (with positive orientation). 
Indeed, the invariant $\gt$-structure \eqref{Eq:ad-invariant_G2-structure} can be written as $\varphi=(D(a)^{-1})^\ast\varphi_0$ where $\varphi_0=e^{123}+\sum_{i=1}^3e^i\wedge \omega_i$ is the canonical $\gt$-structure and 
   $$
    D(a)^{-1}=\left( 
    \begin{array}{c|c}
        \frac{a}{(\det(D^{-1}))^{1/3}}D^{-1} & 0 \\ \hline
        0 & \frac{(\det(D^{-1}))^{1/6}}{a^{1/2}}I_{4\times 4}
    \end{array}\right) \qforq D^{-1}=\left(
   \begin{array}{ccc}
      \alpha_1 & \alpha_2 & \alpha_3 \\
      \beta_1 & \beta_2 & \beta_3 \\
      \gamma_1 & \gamma_2 & \gamma_3
   \end{array}
    \right),
   $$
   for the pair $(a,D)\in \mathbb{R}^+\times\Gl^+(3,\mathbb{R})$. This yields a surjective map
   \begin{equation}\label{Eq: Phi_map}
       \Phi: (a,D)\in \mathbb{R}^+\times\Gl^+(3,\mathbb{R})\mapsto \varphi=(D(a)^{-1})^\ast\varphi_0\in  \Lambda_+^3(\frakp^\ast)^{\Ad(\sptc(1))}.
   \end{equation}
\subsection{Distinguishing homogeneous $\rG_2$-metrics}
\label{Subsec: Ziller_remark}

In \cite{ziller1982}, W. Ziller describes how the seven-parameter family of $\sptc(2)$-invariant metrics falls into isometry classes, which depend on four parameters. We will now apply his approach to the  $\rG_2$-metrics of Lemma \ref{Lemma_invariant_G2_structures}, in order to simplify the description of the space of $\sptc(2)$-invariant $\gt$-structures, as a seven-parameter family.
In general, given a homogeneous space $G/K$, the normaliser $H=:N_G(K)$ of $K$ inside $G$ acts on the right of $G/K$ and thus, each element of $H/K$ defines a $G$-equivariant diffeomorphism \cite{bredon1972}*{4.3 Corollary}. If $G$ is a compact, connected, semisimple Lie group, the group of inner automorphism $\Inn(G)$ is a normal subgroup of $\Aut(G)$ of finite index, and the group $\Out(G)=\Aut(G)/\Inn(G)$ of outer automorphism represents the group of components of $\pi_0(\Aut(G))$, \cite{shankar2001}*{Theorem 1.4}. Thus, the description of $G$-invariant geometric structures (e.g. Riemannian metric, $\gt$-structures, etc) on $G/K$, up to $G$-equivariant equivalence, depends on $\Aut(G,K)$, the set of automorphisms of $G$ fixing $K$. 
For the advantageous case of $G=\sptc(2)$ and $K=\sptc(1)$, it is well known that $\Out(\sptc(2))$ is trivial \cite{helgason2001}*{Theorem 2.6} (since the Dynkin diagram of $C_n$ admits no  non-trivial symmetry), hence $\Aut(\sptc(2))\simeq \sptc(2)$ and $\Aut(\sptc(2),\sptc(1))\simeq N_{\sptc(2)}(\sptc(1))/\sptc(1)$ where
$$
  N_{\sptc(2)}(\sptc(1))=\left\{ \left(\begin{array}{cc}
      m & 0 \\
      0 & n
  \end{array}\right);\quad m,n\in \sptc(1)\right\}.
$$
Thus, consider  
$$
  h\in \sptc(1)\mapsto h:=\left(\begin{array}{cc}
      1 & 0 \\
      0 & h
  \end{array}\right)\in  \sptc(2).
$$
As above, $C_h:=L_h\circ R_{h^{-1}}$ acts by diffeomorphism on $\sptc(2)/\sptc(1)$, for each $h\in H$. In particular:
$$
  (dC_h)_{o}(x,y)=(hx\bar{h},y\bar{h})=(\Upsilon(h)x,y\bar{h}), 
  \qforq (x,y)\in \fp,
$$
where
\begin{equation}
\label{Eq: double_cover_SO3}
      \Upsilon(h)=\left(
        \begin{array}{ccc}
             h_0^2+h_1^2-h_2^2-h_3^2 & 2(h_1h_2-h_0h_3) & 2(h_1h_3+h_0h_2)\\
             2(h_1h_2+h_0h_3) & h_0^2-h_1^2+h_2^2-h_3^2 & 2(h_2h_3-h_0h_1)\\
             2(h_1h_3-h_0h_2) & 2(h_2h_3+h_0h_1) & h_0^2-h_1^2-h_2^2+h_3^2
        \end{array}
      \right),
  \end{equation}
for $h=h_0+h_1i+h_2j+h_3k\in \sptc(1)$.
Since an $\sptc(2)$-invariant metric is determined by its value at $T_o\bbS^7$, the map $(dC_h)_o$ identifies the inner-product \eqref{Eq:ad-invariant_G2-metric} with
\begin{align*}
    \frac{a^2}{\det(D^{-1})^{2/3}}\left(
    \begin{array}{c|c}
     A^t(DD^t)^{-1}A & 0 \\ \hline
         & \frac{\det(D^{-1})}{a^3}I_{4\times 4}
   \end{array}
   \right), \qforq A=\Upsilon(h).
\end{align*}
Hence, $\sptc(1)$ acts on $\sptc(2)/\sptc(1)$ by isometries. Setting $r_4^2=\det D$ and $r_k^2$ the eigenvalue of $DD^t$ with eigenvector $f_k=Ae_k$ (for $k=1,2,3$), the inner product \eqref{Eq:ad-invariant_G2-metric} takes the form
\begin{equation}
\label{Eq: Ziller_metric_form}
    g=\frac{1}{r_1^2}(f^1)^2+\frac{1}{r_2^2}(f^2)^2+\frac{1}{r_3^2}(f^3)^2+\sqrt[3]{\frac{r_1r_2r_3}{r_4^2}}\Big((e^4)^2+(e^5)^2+(e^6)^2+(e^7)^2\Big).
\end{equation}
Therefore the $\sptc(2)$-invariant metrics of Lemma~\ref{Lemma_invariant_G2_structures} are parametrised by $(r_1,r_2,r_3,r_4, h)\in \mathbb{R}^4\smallsetminus \{0\}\times \sptc(1)$, provided $r_1r_2r_3>0$. 
Moreover, the parameter $r_4$ can be omitted under special geometric assumptions, for instance:
\begin{itemize}
    \item \textbf{Normalised volume} Fixing  $\sqrt{\det(g)}=1$, then $r_4^4=\frac{1}{r_1r_2r_3}$ and the metric \eqref{Eq: Ziller_metric_form} becomes
    \begin{equation*}
        g=\frac{1}{r_1^2}(f^1)^2+\frac{1}{r_2^2}(f^2)^2+\frac{1}{r_3^2}(f^3)^2+\sqrt{r_1r_2r_3}\Big((e^4)^2+(e^5)^2+(e^6)^2+(e^7)^2\Big).
    \end{equation*}
    \item \textbf{Homothetic change}. Under a reparametrisation $r_k=r_4^{2/9}s_k$, $k=1,2,3$, the metric \eqref{Eq: Ziller_metric_form} is $r_4^{-2/9}$-homothetic to 
    \begin{equation}\label{Eq: homothetic_metric}
        \widetilde{g}=\frac{1}{s_1^2}(f^1)^2+\frac{1}{s_2^2}(f^2)^2+\frac{1}{s_3^2}(f^3)^2+\sqrt[3]{s_1s_2s_3}\Big((e^4)^2+(e^5)^2+(e^6)^2+(e^7)^2\Big).
    \end{equation}   
    In this case, $\sqrt{\det(\widetilde{g})}=(s_1s_2s_3)^{-2/3}>0$. Furthermore, the $\Ad(\sptc(1))$-map $r_4^{1/9}\rI_\frakp$ defines an isometry between the $\sptc(2)$-invariant metrics (induced by) $g$ and $\widetilde{g}$.
\end{itemize}
\begin{definition}
    Two $\gt$-structures $\varphi_1, \varphi_2$ will be called \emph{isometric} if they induce the same metric under \eqref{eq: induced G2-metric}. 
\end{definition}
The following result reformulates a well-known condition of isometry between homogeneous metrics of the underlying  $\bbS^7=\sptc(2)/\sptc(1)$, but from the point of view of $\gt$-structures.

\begin{proposition}
\label{prop: isometric_equivalence}
   Two $\sptc(2)$-invariant $\gt$-structures, described by $(a,D),(\widetilde{a},\widetilde{D})\in \mathbb{R}^+\times\Gl^+(3,\mathbb{R})$, are isometric if, and only if, $\widetilde{a}=a$ and $\widetilde{D}=D\cdot A$ with $A\in \SO(3)$. 
\end{proposition}

\begin{proof} Let $\varphi$ and $\widetilde{\varphi}$ be  $\gt$-structures described by $(a,D)$ and $(\widetilde{a},\widetilde{D})$, respectively. Consider the induced bilinear forms $B,\widetilde{B}$ given by \eqref{eq: induced G2-metric}. 
According to the proof of Lemma \ref{Lemma_invariant_G2_structures}, if $\varphi$ and $\widetilde{\varphi}$ are isometric, then 
$$
  \frac{B_{ij}}{a\det(D^{-1})^{2/3}}
  =\frac{\widetilde{B}_{ij}}{\widetilde{a}\det(\widetilde{D}^{-1})^{2/3}},\quad  i,j=1,\dots,7.
$$
Equivalently,
\begin{align}
\label{eq: conditions}
\left\{\begin{array}{l}
    \displaystyle
    \frac{a^2(DD^t)^{-1}}{\det(D^{-1})^{2/3}} =\frac{\widetilde{a}^2(\widetilde{D}\widetilde{D}^t)^{-1}}{\det(\widetilde{D}^{-1})^{2/3}}\\
    \\
    \widetilde{a}^3\det(D^{-1})=a^3\det(\widetilde{D}^{-1})
\end{array}\right., 
\end{align}
    and therefore
\begin{align*}
\left\{\begin{array}{l}
    \displaystyle\frac{DD^t}{a^2\det(D)^{2/3}} =\frac{(\widetilde{D}\widetilde{D}^t)}{\widetilde{a}^2\det(\widetilde{D})^{2/3}}\\
    \\
    a^3\det(D)=\widetilde{a}^3\det(\widetilde{D})>0
\end{array}\right.
    \qandq
    DD^t=\widetilde{D}\widetilde{D}^t.
\end{align*}
Define $A:=D^{-1}\widetilde{D}$, so that $\widetilde{D}=D\cdot A$. Then $AA^t=I$ and $\det(A)=\frac{\det(\widetilde{D})}{\det(D)}>0$, since $D,\widetilde{D}\in \Gl^+(3,\mathbb{R})$, which actually implies  $\det(A)=1$, hence $A\in \SO(3)$ and $a=\widetilde{a}$. Conversely, if $\widetilde{D}=D\cdot A$ and $\widetilde{a}=a$, it is easy to verify that \eqref{eq: conditions} holds, thus $\varphi$ and $\widetilde{\varphi}$ are isometric.
\end{proof}

In summary, combining Lemma \ref{Lemma_invariant_G2_structures} and Proposition \ref{prop: isometric_equivalence}, we obtain the first main result:

\begin{proof}[Proof of Theorem \ref{Th: Phi_isomorphism_theorem}]
 By definition, the map $\Phi(a,D)=\varphi_{a,D}=\varphi$ is surjective. Consider $(a,D),(\widetilde{a},\widetilde{D})\in \mathbb{R}^+\times\Gl^+(3,\mathbb{R})$ such that $\Phi(a,D)=\Phi(\widetilde{a},\widetilde{D})$. Then $\widetilde{D}(\widetilde{a})^{-1}D(a)\in G(\varphi_0)\simeq \gt$, where
 $$
   \varphi_0=e^{123}+e^1\wedge \omega_1+e^2\wedge \omega_2+e^3\wedge \omega_3.
 $$ 
 In particular, the $\gt$-structures $\varphi_{a,D}$ and $\varphi_{\widetilde{a},\widetilde{D}}$ are isometric, so Proposition \ref{prop: isometric_equivalence} implies $\widetilde{a}=a$ and $\widetilde{D}=D\cdot A$, with $A\in \SO(3)$. Then, we have
 $$
   \widetilde{D}(\widetilde{a})^{-1}D(a)=\left(\begin{array}{c|c}
    A^t & 0 \\ \hline
    0   & I_{4\times4} 
   \end{array}\right)\in \gt.
 $$
 From the invariance $(\widetilde{D}(\widetilde{a})^{-1}D(a))^\ast\varphi_0=\varphi_0$, we deduce that $A=I_{3\times 3}$, so the map $\Phi$ is injective. 
 
 For the second part of the Theorem, we consider the isomorphism $\Sym_+^2(T^\ast M)^{\sptc(2)}\simeq \Sym_+^2(\frakp)^{\Ad(\sptc(1))}$ and the surjective map $B\circ \Phi$, where $B(\varphi)=g_\varphi$ assigns the corresponding $\Ad(\sptc(1))$-invariant inner product induced by $\varphi\in \Lambda^3_+(\frakp^\ast)^{\Ad(\sptc(1))}$. By Proposition \ref{prop: isometric_equivalence}, $\ker(B\circ \Phi)=\{1\}\times\SO(3)$, and therefore the correspondence \eqref{metric_homogeneous_space} is an  isomorphism. 
 
\end{proof}

We conclude with some observations regarding the linear algebra of the degrees of freedom contained in a pair $(a,D)\in \mathbb{R}^+\times \Gl^+(3,\mathbb{R})$. The matrix $a^{-1} \sqrt[3]{\det(D^{-1})}D$ is a square root of the positive symmetric matrix 
\begin{equation}
\label{Eq: S_matrix}
    S=\Big(\frac{1}{a\sqrt[3]{\det D}}D\Big)\Big(\frac{1}{a\sqrt[3]{\det D}}D\Big)^t\in \Gl(3,\mathbb{R}),
\end{equation}
However, $S$ admits another representation in terms of the eigenvalues  $r_1^2,r_2^2,r_3^2$ and the orthogonal matrix $P=\Upsilon(v)$ ($v\in \sptc(1)$) obtained from  eigenvectors:
\begin{equation}
\label{Eq: matrix C}
    S=CC^t \qwhereq  C=P\sqrt{Q} \qandq \sqrt{Q}=\diag(r_1,r_2,r_3).
\end{equation}
Both representations are related in the following way:
\begin{proposition}\label{prop: block_decomposition}
    If the matrix $S$ can be written in the form \eqref{Eq: matrix C}, then there exists $A\in \SO(3)$ such that 
\begin{equation}
\label{Eq: matrix D=CA}
    \frac{1}{a\sqrt[3]{\det(D)}}D=CA.
\end{equation}
    Moreover, for any $(a,D)\in \mathbb{R}^+\times\Gl^+(3,\mathbb{R})$, there exists $(r_1,...,r_4,v,h)\in \mathbb{R}^4\smallsetminus\{0\}\times\sptc(1)^2$ such that
\begin{equation}
\label{eq: block_decomposition}
    a=\frac{1}{\sqrt[3]{r_1r_2r_3}}>0, \quad D=\Big[\sqrt[3]{\tfrac{r_4^2}{r_1r_2r_3}}\Upsilon(v)\sqrt{Q}\Big]\Upsilon(h).
\end{equation}
\end{proposition}

\begin{proof}
Define the matrix $A=a^{-1} \sqrt[3]{\det(D^{-1})}C^{-1}D$ and combine equations \eqref{Eq: S_matrix} and \eqref{Eq: matrix C}, to obtain
\begin{equation*}
    AA^t=I \qandq \det (A)=1.
\end{equation*}
The second part of the proposition follows by applying the decomposition \eqref{Eq: matrix D=CA} to the morphism \eqref{Eq: double_cover_SO3}.
\end{proof}

In summary, from the expression of $D$ in \eqref{eq: block_decomposition}, the elements inside the square brackets determine the metric. So, up to isometry, we have
\begin{equation}\label{eq: diagonal_metric}
    g_{(r_1,r_2,r_3,r_4)}=\frac{1}{r_1^2}(e^1)^2+\frac{1}{r_2^2}(e^2)^2+\frac{1}{r_3^2}(e^3)^2+\sqrt[3]{\frac{r_1r_2r_3}{r_4^2}}\Big((e^4)^2+(e^5)^2+(e^6)^2+(e^7)^2\Big).
\end{equation}
And for $\sptc(2)$-invariant $\gt$-structures, we have:
\begin{corollary}
Any $\Ad(\sptc(1))$-invariant $\gt$-structure given by \eqref{Eq:ad-invariant_G2-structure} is equivariantly equivalent with
\begin{align}\label{eq: 7-family_g2_structures}\nonumber
    \widetilde{\varphi}
    =\frac{1}{r_1r_2r_3}e^{123}&+\sqrt[3]{\frac{r_2r_3}{r_4^2r_1^2}}e^1\wedge\bigg((h_0^2+h_1^2-h_2^2-h_3^2)\omega_1+2(h_1h_2+h_0h_3)\omega_2+2(h_1h_3-h_0h_2)\omega_3\bigg)\\ &+\sqrt[3]{\frac{r_1r_3}{r_4^2r_2^2}}e^2\wedge\bigg(2(h_1h_2-h_0h_3)\omega_1+(h_0^2-h_1^2+h_2^2-h_3^2)\omega_2+2(h_2h_3+h_0h_1)\omega_3\bigg)\\ \nonumber &+\sqrt[3]{\frac{r_1r_2}{r_4^2r_3^2}}e^3\wedge\bigg(2(h_1h_3+h_0h_2)\omega_1+2(h_2h_3-h_0h_1)\omega_2+(h_0^2-h_1^2-h_2^2+h_3^2) \omega_3\bigg).
\end{align}
\end{corollary}

\begin{proof}
According to \eqref{Eq: Phi_map} and Proposition \ref{prop: block_decomposition}, we know that an $\Ad(\sptc(1))$-invariant $\gt$-structure \eqref{Eq:ad-invariant_G2-structure} is given by $ \varphi=\Phi(a,D)=(D(a)^{-1})^\ast\varphi_0$, with $a$ and $D$ as in \eqref{eq: block_decomposition}. Then,
\begin{equation*}
    \varphi=((dC_v)_o^{-1})^\ast(\widetilde{D}(a)^{-1})^\ast\widetilde{\varphi}_0,
\end{equation*}
where $\widetilde{D}=\sqrt[3]{\tfrac{r_4^2}{r_1r_2r_3}}\sqrt{Q}\Upsilon(h)$ and 
  $\widetilde{\varphi}_0=e^{123}+\sum_{l=1}^3e^i\wedge\widetilde{\omega}_l$ with $\widetilde{\omega}_l=((dC_v)_o)|_{\frakp_4}^\ast\omega_l$.
\end{proof}

\subsection{$\gt$-structures on $\Sym^2_+(\frakp)^{\Ad(\sptc(1))}$}\label{Subsec: G2_structure_of_the_space_of_metric}

We identify the manifold of $\Ad(\sptc(1))$-invariant inner products of $\frakp$ with a $7$-dimensional solvable Lie group, which, in turn, endows $\Sym^2_+(\frakp)^{\Ad(\sptc(1))}$ with a natural $\gt$-structure.

By Proposition \ref{prop: isometric_equivalence}, the map $\varphi\in (\Lambda_+^3(\frakp)^\ast)^{\Ad(\sptc(1))}\rightarrow g_\varphi\in \Sym^2_+(\frakp)^{\Ad(\sptc(1))}$ translates into $\mathbb{R}^+\times\Gl^+(3,\mathbb{R})\rightarrow \mathbb{R}^+\times\Big(\Gl^+(3,\mathbb{R})/\SO(3)\Big)$. On the one hand, the group $\Gl^+(3,\mathbb{R})$ is isomorphic to $\mathbb{R}^+\times \Sl(3,\mathbb{R})$, by the Levi decomposition $\frakgl(3,\mathbb{R})=\mathbb{R}\oplus\fraksl(3,\mathbb{R})$. Notice that the space of orbits  $\Gl^+(3,\mathbb{R})/\SO(3)$ corresponds to $\mathbb{R}^+\times\Big(\Sl(3,\mathbb{R})/\SO(3)\Big)$, since the factor $\mathbb{R}^+$ describes the determinant of the elements in $\Gl^+(3,\mathbb{R})$.
On the other hand, the special linear group  admits the decomposition $\Sl(3,\mathbb{R})=NU\SO(3)$, where $N$ is the group of upper triangular matrices with all diagonal entries equal to 1 and $U$ the group of diagonal matrices with positive entries and unit determinant. 
This gives the following identification:
$$
\Sym^2_+(\frakp)^{\Ad(\sptc(1))}\simeq \mathbb{R}^+\times \mathbb{R}^+\times NU.
$$
The group $NU$ is simply-connected and diffeomorphic to the Euclidean space, and its Lie algebra consists of the traceless upper-triangular matrices, under the commutator. Denoting by $E_{ij}$ the matrix with entries $1$ at its $i$-th row and $j$-th column, and zero elsewhere, a basis of $\Lie(NU)$ is given by
\begin{equation*}
    u_1=E_{11}-E_{33}, \quad u_2=E_{22}-E_{33}, \quad v_3=E_{21}, \quad v_4=E_{32} \qandq v_5=E_{31}.
\end{equation*}
A direct computation of the bracket shows that $\Lie(NU)\simeq\mathbb{R}^2\times_\rho\fh_3$, where $\mathbb{R}^2=\Span(u_1,u_2)$ 
is the Abelian subalgebra, $\mathbb{R}^3=\Span(v_3,v_4,v_5)$ is the $2$-step nilpotent Lie algebra with $[v_3,v_4]=-v_5$, which is isomorphic to the Heisenberg Lie algebra $\fh_3$, the map $\kappa:\mathbb{R}^2\rightarrow \Der(\fh_3)$ given by
\begin{equation*}
    u_1\mapsto \kappa(u_1)=\left(\begin{array}{ccc}
        -1 & & \\
         & -1 & \\
         & & -2
    \end{array}\right), 
    \qquad  
    u_2\mapsto \kappa(u_2)=\left(\begin{array}{ccc}
        1 & & \\
         & -2 & \\
         & & -1
    \end{array}\right)
\end{equation*}
is a representation, and the bracket is the one induced by the semi-direct product
$$
  [[(u,v),(\bar{u},\bar{v})]]=(0,\kappa(u)\bar{v}-\kappa(\bar{u})v+[v,\bar{v}]) \qforq u,\bar{u}\in \mathbb{R}^2 \qandq v,\bar{v}\in \fh_3.
$$
Therefore, $NU\simeq \mathbb{R}^2\times_\mu H(3)$, with $\kappa=d\mu_e$, $\Lie(H(3))=\fh_3$, and the Lie group structure can be extended trivially to $\mathbb{R}^+\times \mathbb{R}^+\times \mathbb{R}^2\times_\mu H(3)$ by
$$
  (s_1,t_1,u_1,v_1)(s_2,t_2,u_2,v_2)=(s_1s_2,t_1t_2,u_1+u_2,v_1\mu(u_1)(v_2)),
  \qforq (s_i,t_i,u_i,v_i)\in \mathbb{R}^+\times\mathbb{R}^+\times \mathbb{R}^2\times_\mu H(3).
$$
The Lie algebra of $ \mathbb{R}^+\times \mathbb{R}^+\times \mathbb{R}^2\times_\mu H(3)$ is $\fh= \mathbb{R}^4\times_\kappa\fh_3$, and it has a natural $\gt$-structure by considering $v^3,v^4,v^5$ a dual basis of  $\fh_3$ and $\nu_1,\nu_2,\nu_3$ a basis of the self-dual $2$-forms $\Lambda_+^2(\mathbb{R}^4)^\ast$. By left-translation, the $\gt$-structure on $\fh$
$$
  \varphi=v^{345}+v^3\wedge\nu_1+v^4\wedge\nu_2+v^5\wedge\nu_3
$$
defines a $\gt$-structure on the Lie group $ \mathbb{R}^+\times \mathbb{R}^+\times \mathbb{R}^2\times_\mu H(3)$.

\section{$\rG_2$-structures in different isometric classes}
\label{Sec:Isometric_classes}

We now propose an approach to explicitly parametrise and study different isometric classes of $\sptc(2)$-invariant $\gt$-structures on $\bbS^7$. By considering $\sptc(2)$-metrics only up to isometries, we describe the reduced space $\Sym^2_+(\frakp)^{\Ad(\sptc(1))}$ in terms of just $3$ parameters, cf. \eqref{Eq: homothetic_metric} and Theorem \ref{Th: Phi_isomorphism_theorem}. This framework will lead to Theorem \ref{Th: isometric_families_theorem}, which situates some distinguished types of $\gt$-structures on $\bbS^7$, within a same family of (isometric classes of) $\gt$-structures. The main inspirations mobilised here come from \cites{bryant2003some,dwivedi2019,friedrich1997nearly}.
 
 Set $\br=(r_1,r_2,r_3)$, with $r_1,r_2,r_3>0$, and consider the $\gt$-structure 
\begin{equation}\label{eq: varphi_r123}
    \varphi_\br=\frac{1}{(r_1r_2r_3)^3}e^{123}+\frac{r_2r_3}{r_1^2}e^1\wedge\omega_1+\frac{r_1r_3}{r_2^2}e^2\wedge\omega_2+\frac{r_1r_2}{r_3^2}e^3\wedge\omega_3,
\end{equation}
with dual $4$-form
\begin{equation*}
    \psi_\br:=\ast\varphi_\br=(r_1r_2r_3)^2\frac{\omega_1^2}{2}+\frac{r_3}{r_1^2r_2^2}e^{12}\wedge\omega_3-\frac{r_2}{r_1^2r_3^2}e^{13}\wedge\omega_2+\frac{r_1}{r_2^2r_3^2}e^{23}\wedge\omega_1,
\end{equation*}
and induced $\gt$-metric 
\begin{equation}\label{Eq:metric_r1r2r3}
    g_\br=\frac{1}{r_1^6}(e^1)^2+\frac{1}{r_2^6}(e^2)^2+\frac{1}{r_3^6}(e^3)^2+r_1r_2r_3\Big((e^4)^2+(e^5)^2+(e^6)^2+(e^7)^2\Big).
\end{equation}
R. Bryant \cite{bryant2003some}*{Remark 4} provides a formula for $\gt$-structures with the same associated metric and orientation, in terms of $(f,X)\in C^\infty(M)\times\mathfrak{X}(M)$ satisfying $f^2+|X|^2=1$:
\begin{equation}\label{eq: Bryant_formula}
         \varphi_{(f,X)}=(f^2-|X|^2)\varphi-2f(X\lrcorner\psi)+2X^\flat\wedge(X\lrcorner\varphi).  \end{equation}
For $\sptc(2)$-invariant $\gt$-structures, we have the following instance.

\begin{lemma}
\label{lemma: (f,X)_invariant}
    If the isometric $\rG_2$-structures $\varphi_{(f,X)}$ and $\varphi$ in \eqref{eq: Bryant_formula} are $\sptc(2)$-invariant, then $f$ must be a constant function and $X$ an $\sptc(2)$-invariant vector field. In particular, the $\gt$-structure \eqref{eq: 7-family_g2_structures} has the form
\begin{equation*}
        \varphi:=\varphi_{(h_0,X)}=(h_0^2-|X|^2)\varphi_\br-2h_0(X\lrcorner\psi_\br)+2X^\flat\wedge(X\lrcorner\varphi_\br) \qforq h_0^2+|X|^2=1,
    \end{equation*}
    where $X=r_1^3h_1e_1+r_2^3h_2e_2+r_3^3h_3e_3$ is an $\ad(\fraksp(1))$-invariant vector of $\frakp$, and $|X|^2=g_\br(X,X)=1-h_0^2$.
\end{lemma}
\begin{proof}
Given $p=[h]\in \sptc(2)/\sptc(1)$ and $k\in \sptc(2)$, the $\sptc(2)$-invariance of $\varphi_{(f,X)}$ precisely means that  
$$
(L_k^\ast\varphi_{(f,X)})_p=(\varphi_{(f,X)})_{k\cdot p}.
$$
Expanding and equating,
\begin{align*}
    (L_k^\ast\varphi_{(f,X)})_p
    &=(f(p)^2-|X|^2(p))\varphi_{k\cdot p}-2f(p)((dL_{k^{-1}})_{k\cdot p}X\lrcorner\psi)_{k\cdot p}+2(L_k^\ast X^\flat)_p\wedge((dL_{k^{-1}})_{k\cdot p}X\lrcorner\varphi)_{k\cdot p},\\
    (\varphi_{(f,X)})_{k\cdot p}
    &=(f(k\cdot p)^2-|X|^2(k\cdot p))\varphi_{k\cdot p}-2f(k\cdot p)(X\lrcorner\psi)_{k\cdot p}+2X^\flat_{k\cdot p}\wedge(X\lrcorner\varphi)_{k\cdot p},
\end{align*}
we find that $f(k\cdot p)=f(p)$ and $(L_{k^{-1}})_\ast X=X$. 

For the second part of the Lemma, write $f_a=r_a^3e_a$ and $\beta_a=r_1r_2r_3\omega_a$, for $a=1,2,3$. Then \eqref{eq: 7-family_g2_structures} becomes
  \begin{align*}
   \begin{split}
        \varphi=f^{123}&+\Big((h_0^2+h_1^2-h_2^2-h_3^2)f^1 +2(h_1h_2-h_0h_3)f^2 +2(h_1h_3+h_0h_2)f^3\Big)\wedge\beta_1\\
        &+\Big(2(h_1h_2+h_0h_3)f^1 +(h_0^2-h_1^2+h_2^2-h_3^2)f^2 + 2(h_2h_3-h_0h_1)f^3\Big)\wedge\beta_2\\
        &+\Big(2(h_1h_3-h_0h_2)f^1 + 2(h_2h_3+h_0h_1)f^2+(h_0^2-h_1^2-h_2^2+h_3^2)f^3\Big)\wedge \beta_3\\
        =f^{123}& +\Big((1+2h_1^2-2(h_1^2+h_2^2+h_3^2))f^1 +2(h_1h_2-h_0h_3)f^2 +2(h_1h_3+h_0h_2)f^3\Big)\wedge\beta_1\\
        &+\Big(2(h_1h_2+h_0h_3)f^1 +(1+2h_2^2-2(h_1^2+h_2^2+h_3^2))f^2 +2(h_2h_3-h_0h_1)f^3\Big)\wedge\beta_2\\
        &+\Big(2(h_1h_3-h_0h_2)f^1+ 2(h_2h_3+h_0h_1)f^2+(1+2h_1^2-2(h_1^2+h_2^2+h_3^2))f^3\Big)\wedge \beta_3\\
        =\varphi_\br&-2(h_1^2+h_2^2+h_3^2)\varphi_\br-2h_0\Big( (h_3f^2-h_2f^3)\wedge\beta_1+(-h_3f^1+h_1f^3)\wedge\beta_2\\
        &-2h_0\Big(h_2f^1- h_1f^2\Big)\wedge \beta_3+2\Big((h_1^2+h_2^2+h_3^2)f^{123}+(h_1^2f^1+h_1h_2f^2+h_1h_3f^3)\wedge\beta_1\Big)\\
        &+(h_1h_2+h_0h_3f^1+h_2^2f^2+ h_2h_3f^3)\wedge\beta_2+(h_1h_3f^1+ h_2h_3f^2+h_3^2f^3)\wedge \beta_3\Big)\\
        =(1&-2|X|^2)\varphi_\br-2h_0 X\lrcorner \psi_\br+2X^\flat\wedge X\lrcorner\varphi_\br
        \qedhere
        \end{split}
   \end{align*}
\end{proof}

\subsection{The torsion of homogeneous $\rG_2$-structures}

The existence of $\gt$-structures on a manifold  $M$ determines a decomposition of the spaces of differential forms on $M$ into irreductible $\gt$-representations. Then, the spaces $\Omega^2$ and $\Omega^3$ of $2$-forms and $3$-forms decompose as
	\begin{align}
	\Omega^2 &= \Omega^2_7\oplus\Omega^2_{14}\nonumber\\
	\Omega^3 &= \Omega^3_1\oplus\Omega_{7}^{3}\oplus\Omega^3_{27}\nonumber
	\end{align}
	where each $\Omega^k_l$ has (pointwise) dimension $l$ and this decomposition is orthogonal with respect to the metric $g$. The spaces $\Omega^2_7$ and $\Omega_7^3$ are both isomorphic to the cotangent bundle $\Omega^1_7=T^{\ast}M$. In \cites{karigiannis2009flows}, Karigiannis gives explicit isomorphisms between the space $\Omega^2_{14}$ and the Lie algebra $\fg_2=\Lie(\gt)$ and between $\Omega^3_{27}$ and the space of traceless symetric $2$-tensors $\Sym_0^2(T^{\ast}M)$ on $M$. The first identification comes from the canonical isomorphism between $\Omega^2$ and $\so(7)$, the second one is given by the maps
	\begin{align}\label{eq: Bryant_maps}
	    \imath: \Sym^2_0(T^\ast M)&\rightarrow\Omega_{27}^3, & \jmath: \Omega_{27}^3&\rightarrow \Sym^2_0(T^\ast M)\\
	    \beta_{ij}&\mapsto \beta_{ij}g^{jl}e^i\wedge (e_l\lrcorner\varphi), & \tau&\mapsto \ast_\varphi\big((e_i\lrcorner\varphi)\wedge(e_j\lrcorner\varphi)\wedge\tau\big)\nonumber
	\end{align}
Applying this decomposition to the natural components of torsion $d\varphi$ and $d\psi$, one is lead to the following definition.
\begin{definition}
If $\varphi$ is a $\gt$-structure on a $7$-manifold, with associated $4$-form $\psi$, then there are unique forms $\tau_0\in\Omega^0_1$, $\tau_1\in\Omega^1_7$, $\tau_2\in\Omega^2_{14}$ and $\tau_3\in\Omega^3_{27}$, called the \emph{torsion forms} of $\varphi$, such that
\begin{align*}
    d\varphi=& \tau_0\psi+3\tau_1\wedge\varphi+\ast_{\varphi}\tau_3,\nonumber\\
    d\psi=& 4\tau_1\wedge\psi+\ast_{\varphi}\tau_2
\end{align*}
\end{definition}
The torsion forms can be explicitly computed from $\varphi$ and $\psi$ by means of the following identities: 
\begin{align}
\label{eq: identities_torsion_forms}
    \tau_0=&\frac{1}{7}\ast_{\varphi}(\varphi\wedge d\varphi), &\tau_1=&\frac{1}{12}\ast_{\varphi}(\varphi\wedge\ast_{\varphi}d\varphi)=\frac{1}{12}\ast_{\varphi}(\psi\wedge\ast_{\varphi}d\psi),\\
    \tau_2=&-\ast_{\varphi}(d\psi)+4\ast_{\varphi}(\tau_1\wedge\psi), &\tau_3=&\ast_{\varphi}(d\varphi)-\tau_0\varphi-3\ast_{\varphi}(\tau_1\wedge\varphi).\nonumber
\end{align}
Moreover, the torsion forms are completely encoded in the \emph{full torsion tensor} $T\in \End(TM)\simeq T^\ast M\otimes TM$,
\begin{equation}
\label{eq: full_torsion_tensor}
    T=\frac{\tau_0}{4}g_{\varphi}-\ast(\tau_1\wedge\psi)-\frac{1}{2}\tau_2-\frac{1}{4}\jmath(\tau_3),
\end{equation}
which is expressed in terms of the irreducible $\gt$-decomposition $\End(TM)=W_0\oplus W_1\oplus W_2\oplus W_3$, where $W_0\simeq \Omega^0_1$, $W_1\simeq\Omega^3_7$, $W_2\simeq \Omega^2_{14}$ and $W_3\simeq \Omega^3_{27}$, see e.g. \cite{fernandez1982riemannian}.

In particular, back on $M=\bbS^7$, we can compute the torsion forms of the $\gt$-structures $\varphi_\br$ given by \eqref{eq: varphi_r123}.  For $X\in \fraksp(2)$, consider the vector field $\widetilde{X}$ on $\bbS^7$ and the left-invariant vector field $X^l$ on $\sptc(2)$ defined by
\begin{equation*}
    \widetilde{X}(u\sptc(1))=\frac{d}{dt}|_{t=0}\exp(tX)u\sptc(1) \qandq X^l(u)=(dL_u)_{1_{\sptc(2)}}X, \qforq u\in \sptc(2),
\end{equation*}
respectively. Using the canonical projection $p:\sptc(2)\rightarrow \sptc(2)/\sptc(1)$, we have
\begin{equation*}
    dp_u(X^l)=(d\lambda_u)_o(\widetilde{X}),
\end{equation*}
where $\lambda_u$ denotes the left action of $u$ on $\bbS^7$. Now, if $\alpha\in \Omega^k(\bbS^7)^{\sptc(2)}$ then $\beta=p^\ast\alpha\in\Omega^k(\sptc(2))^{\sptc(2)}$, thus for $X_1,...,X_k\in \fp\simeq T_o\bbS^7$
\begin{align*}
    d\alpha_o(X_1,...,X_k)=&(d\alpha)_{\lambda_{u^{-1}}(p(u))}((d\lambda_{u^{-1}})_{p(u)}dp_uX^l_1,...,(d\lambda_{u^{-1}})_{p(u)}dp_uX^l_k)\\
    =& d\alpha_{p(u)}(dp_uX^l_1,...,dp_uX^l_k)\\
    =& d\beta_u(X^l_1,...,X^l_k)=d\beta_1(X_1,...,X_k)\\
    =&\sum_{i<j}(-1)^{i+j}\beta_{1_{\sptc(2)}}([X_i,X_j],X_1,...,\hat{X_i},...,\hat{X_j},...,X_k),
\end{align*}
where we used the $\sptc(2)$-invariance of $\alpha$, $d\alpha$ and $\beta$, as well as the property $d\alpha=p^\ast d\beta$. Now, notice that $p\circ L_u=\lambda_u\circ p$, then
\begin{align*}
    \beta_{1_{\sptc(2)}}([X_i,X_j],X_1,...,\hat{X_i},...,\hat{X_j},...,X_k)=&\alpha_u(d(p\circ L_u)_1[X_i,X_j],d(p\circ L_u)_1X_1,...,\hat{X_i},...,\hat{X_j},...,d(p\circ L_u)_1X_k)\\
    =&\alpha_u((d\lambda_u)_0[X_i,X_j]_\frakp,(d\lambda_u)_0X_1,...,\hat{X_i},...,\hat{X_j},...,(d\lambda_u)_0X_k)\\
    =&\alpha_o(([X_i,X_j]_\frakp,X_1,...,\hat{X_i},...,\hat{X_j},...,X_k)
\end{align*}
Hence, using \eqref{bracket_sp2} for the invariant forms $e^1,e^2, e^3\in \Omega^1(\bbS^7)^{\sptc(2)}$ and $\omega_1,\omega_2,\omega_3\in \Omega^2(\bbS^7)^{\sptc(2)}$, we have
\begin{align}
\label{Eq:Differential.1.form} \nonumber
    de^1&=-2e^{23}-\omega_1, &  d\omega_1=-2e^2\wedge\omega_3-2e^3\wedge\omega_2,\nonumber\\
    de^2&=2e^{13}+\omega_2, & d\omega_2=-2e^1\wedge\omega_3+2e^3\wedge\omega_1,\\ \nonumber
    de^3&=-2e^{12}-\omega_3, & d\omega_3=\quad 2e^1\wedge\omega_2+2e^2\wedge\omega_1,\nonumber
\end{align}

\begin{equation}
\label{bracket_sp2}
\newcolumntype{M}{>{$}c<{$}}
\begin{tabular}{M|M|M|M|M|M|M|M|M|M|M}
   [\cdot,\cdot]  & v_1 & v_2 & v_3 & e_1 & e_2 & e_3 & e_4 & e_5 & e_6 & e_7  \\ \hline
  v_1 & 0 & 2v_3 & -2v_2 & 0 & 0 & 0 & -e_7 & e_6 & -e_5 & e_4 \\ \hline
  v_2 & -2v_3 & 0 & 2v_1 & 0 & 0 & 0 & -e_6 & -e_7 & e_4 & e_5 \\ \hline
  v_3 & 2v_2 & 2v_1 & 0 & 0 & 0 & 0 & e_5 & -e_4 & -e_7 & e_6 \\ \hline
  e_1 & 0 & 0 & 0 & 0 & 2e_3 & -2e_2 & e_7 & e_6 & -e_5 & -e_4 \\ \hline
  e_2 & 0 & 0 & 0 & -2e_3 & 0 & 2e_1 & -e_6 & e_7 & e_4 & -e_5 \\ \hline
  e_3 & 0 & 0 & 0 & 2e_2 & -2e_1 & 0 & e_5 & -e_4 & e_7 & -e_6 \\ \hline
  e_4 & e_7 & e_6 & -e_5 & -e_7 & e_6 & -e_5 & 0 & v_3+e_3 & -v_2-e_2 & -v_1+e_1 \\ \hline
  e_5 & -e_6 & e_7 & e_4 & -e_6 & -e_7 & e_4 & -v_3-e_3 & 0 & v_1+e_1 & -v_2+e_2 \\ \hline
  e_6 & e_5 & -e_4 & e_7 & e_5 & -e_4 & -e_7 & v_2+e_2 & -v_1-e_1 & 0 & -v_3+e_3 \\ \hline
  e_7 & -e_4 & -e_5 & -e_6 & e_4 & e_5 & e_6 & v_1-e_1 & v_2-e_2 & v_3-e_3 & 0
\end{tabular}
\end{equation}


From \eqref{Eq:Differential.1.form} it is easy to verify that  $d\psi_\br=0$ and 
\begin{align}\label{eq: dvarphi_r123}
\begin{split}
    d\varphi_\br=&2\frac{(r_1^3r_3^3-r_2^3r_3^3-r_1^3r_2^3)}{(r_1r_2r_3)^2}\frac{\omega_1^2}{2}-\left[\frac{2r_1^4r_2r_3^4-2r_1r_2^4r_3^4+2r_1^4r_2^4r_3+1}{(r_1r_2r_3)^3} \right]e^{12}\wedge\omega_3\\
    &+\left[\frac{2r_1^4r_2r_3^4+2r_1r_2^4r_3^4+2r_1^4r_2^4r_3-1}{(r_1r_2r_3)^3} \right]e^{13}\wedge\omega_2-\left[\frac{2r_1^4r_2r_3^4+2r_1r_2^4r_3^4-2r_1^4r_2^4r_3+1}{(r_1r_2r_3)^3} \right]e^{23}\wedge\omega_1.
    \end{split}
\end{align}
Hence the torsion forms $\tau_1(\br)$ and $\tau_2(\br)$ vanish, and the full torsion tensor $T_\br=\frac{\tau_0(\br)}{4}g_\br-\tau_{27}(\br)$ is diagonal.
\begin{proposition}\label{prop: torsion_forms_r123}
The torsion forms of \eqref{eq: varphi_r123} are
\begin{align*}
    \tau_0(\br)=-\frac{4}{7(r_1r_2r_3)^4}\left[2r_1^4r_2^4r_3^4(r_1^3-r_2^3+r_3^3)+r_1r_2r_3(r_1^6r_3^6+r_2^6r_3^6+r_1^6r_2^6)+r_1^3r_2^3-r_1^3r_3^3+r_2^3r_3^3\right]
\end{align*}
and $\tau_{27}(\br)=\diag(p_1,p_2,p_3,p_4,p_4,p_4,p_4)$ where
\begin{align*}
    p_1=&-\frac{1}{7r_1^{10}r_2^4r_3^4}\left(6r_1^7r_2r_3^7-2r_1^4r_2^4r_3^7-8r_1r_2^7r_3^7+12r_1^7r_2^4r_3^4+2r_1^4r_2^7r_3^4+6r_1^7r_2^7r_3+r_1^3r_3^3-8r_2^3r_3^3-r_1^3r_2^3 \right)\\
    p_2=&-\frac{1}{7r_1^4r_2^{10}r_3^4}\left(6r_1r_2^7r_3^7-2r_1^4r_2^4r_3^7-8r_1^7r_2r_3^7-12r_1^4r_2^7r_3^4-2r_1^7r_2^4r_3^4+6r_1^7r_2^7r_3+8r_1^3r_3^3-r_2^3r_3^3-r_1^3r_2^3 \right)\\
    p_3=&-\frac{1}{7r_1^4r_2^4r_3^{10}}\left(6r_1^7r_2r_3^7+12r_1^4r_2^4r_3^7+6r_1r_2^7r_3^7-2r_1^7r_2^4r_3^4+2r_1^4r_2^7r_3^4-8r_1^7r_2^7r_3+r_1^3r_3^3-r_2^3r_3^3-8r_1^3r_2^3 \right)\\
    p_4=&-\frac{1}{7(r_1r_2r_3)^3}\left(2r_1^4r_2^4r_3^4(r_1^3-r_2^3+r_3^3)+r_1r_2r_3(r_1^6r_3^6+r_2^6r_3^6+r_1^6r_2^6)-\frac{5}{2}(r_1^3r_2^3-r_1^3r_3^3+r_2^3r_3^3) \right).
\end{align*}
\end{proposition}

\begin{proof}
 The torsion forms $\tau_0(\br)$ and $\tau_3(\br)$ follow by applying the identities \eqref{eq: identities_torsion_forms} to the $3$-form \eqref{eq: varphi_r123} and its exterior derivative \eqref{eq: dvarphi_r123}. Using the map $\jmath$ from \eqref{eq: Bryant_maps} on $\tau_3(\br)$, the traceless symmetric $2$-tensor $\tau_{27}(\br)=\frac{1}{4}\jmath(\tau_3(\br))$ follows by a long, but straightforward computation. 
\end{proof}

\subsection{A general Ansatz of inequivalent homogeneous $\rG_2$-metrics} 
\label{Subsec: Ansatz_r_G2-structures}
 
The general formula \eqref{eq: divergence_in_coordinates} offers little insight into the construction of interesting examples. Let us consider, from now on, the Ansatz $r_1=r_2=r_3=r^{-1/3}$,
so the metric prescribed in \eqref{Eq:metric_r1r2r3} is  
\begin{equation}
\label{Eq:metric}
    g_r(u,v)=r^2\langle u, v\rangle, \quad \forall u,v\in \frakp_1\oplus \frakp_2\oplus \frakp_3 \qandq g_r(u,v)=\frac{1}{r}\langle u,v\rangle, \quad \forall u,v\in \frakp_4.
\end{equation}
The corresponding isometric family is
\begin{align}\label{eq: isometric_G2}
   \begin{split}
        \varphi_{r}=r^3e^{123}&+\Big((h_0^2+h_1^2-h_2^2-h_3^2)e^1+2(h_1h_2-h_0h_3)e^2+2(h_1h_3+h_0h_2)e^3\Big)\wedge\omega_1\\
        &+\Big(2(h_1h_2+h_0h_3)e^1+(h_0^2-h_1^2+h_2^2-h_3^2)e^2+ 2(h_2h_3-h_0h_1)e^3\Big)\wedge\omega_2\\
        &+\Big(2(h_1h_3-h_0h_2)e^1+ 2(h_2h_3+h_0h_1)e^2+(h_0^2-h_1^2-h_2^2+h_3^2)e^3\Big)\wedge \omega_3.
        \end{split}
   \end{align}
According to Theorem \ref{Th: Phi_isomorphism_theorem}, $\varphi_r=\Phi(r.\Upsilon(\bar{h}))$ where $\Upsilon$ is the double cover homomorphism of $\SO(3)$, cf. \eqref{Eq: double_cover_SO3}. Moreover, the induced dual $4$-form is
\begin{align}\label{eq: psi_r}
   \begin{split}
        \psi_{r}=\frac{1}{2r^2}\omega_1^2&+r\Big((h_0^2+h_1^2-h_2^2-h_3^2)e^{23}-2(h_1h_2-h_0h_3)e^{13}+2(h_1h_3+h_0h_2)e^{12}\Big)\wedge\omega_1\\
        &+r\Big(2(h_1h_2+h_0h_3)e^{23}-(h_0^2-h_1^2+h_2^2-h_3^2)e^{13}+ 2(h_2h_3-h_0h_1)e^{12}\Big)\wedge\omega_2\\
        &+r\Big(2(h_1h_3-h_0h_2)e^{23}- 2(h_2h_3+h_0h_1)e^{13}+(h_0^2-h_1^2-h_2^2+h_3^2)e^{12}\Big)\wedge \omega_3.
        \end{split}
   \end{align}

For the $\gt$-structure  $\varphi_{r}$ from \eqref{eq: isometric_G2}, with associated $4$-form $\psi_r=\ast\varphi_r$, substituting Equations~\eqref{Eq:Differential.2.form.without.Omega} and \eqref{Eq:Differential.2.form.with.Omega} from Appendix~\ref{Standard computations} into $d\varphi_{r}$ and $d\psi_{r}$ yields:
\begin{eqnarray}
 d \varphi_{r} &=&  -\bigg( 8h_1h_3e^{12}+8h_0h_3e^{13}+(r^3-8h_3^2+2)e^{23}\bigg)\wedge \omega_1\nonumber\\ 
      & &+\bigg(8h_0h_1e^{12}-(r^3-8h_0^2+2)e^{13}-8h_0h_3e^{23}\bigg)\wedge \omega_2\label{Eq:diff.varphi.isometric}\\
       & &-\bigg((r^3-8h_1^2+2)e^{12}-8h_0h_1e^{13}+8h_1h_3e^{23}\bigg)\wedge\omega_3-(1-4h_2^2)\omega_1^2,
       \nonumber\\
 d\psi_{r} &=& -8h_2r\bigg(e^{123}\wedge(h_1\omega_3+h_0\omega_2-h_3\omega_1)+(h_3e^1+h_0e^2-h_1e^3)\wedge\frac{\omega_1^2}{2} \bigg).\nonumber
\end{eqnarray}  
Furthermore, we deduce that
\begin{align}
    \ast d \varphi_{r} =&-2r^5(1-4h_2^2)e^{123}- \frac{1}{r}\Big((r^3-8h_3^2+2)e^{1}-8h_0h_3e^{2}+8h_1h_3e^{3}\Big)\wedge\omega_1\nonumber\\   
    & -\frac{1}{r}\Big( 8h_0h_3e^{1}-(r^3-8h_0^2+2)e^2-8h_0h_1e^{3}\Big)\wedge\omega_2\label{Eq:star.d.varphi}\\
    & -\frac{1}{r}\Big(8h_1h_3e^{1}+8h_0h_1e^{2}+(r^3-8h_1^2+2)e^3\Big)\wedge\omega_3 ,\nonumber\\
    \ast d\psi_ {r}
    =& \; 8h_2\bigg( r^4h_1e^{12}+r^4h_0e^{13}-r^4h_3e^{23}+\frac{1}{r^2}\Big(h_3\omega_1-h_0\omega_2-h_1\omega_3\Big)\bigg).
 \label{Eq:star.d.psi}
\end{align}
    We are now in position to compute the torsion forms of these $\gt$-structures.
\begin{proposition}
\label{prop: h_torsion_forms}
    The torsion forms of the  $\gt$-structure $\varphi_r$ defined by \eqref{eq: isometric_G2} are:
\begin{align*}
\begin{split}
    \tau_0
    &=-\frac{4}{7r}\Big(r^3(1-4h_2^2)+(5-8h_2^2)\Big), \\
    \tau_1
    &= -\frac{2(r^3+2)h_2}{3}\Big(h_3e^1+h_0e^2-h_1e^3\Big),\\
    \tau_2 
    &=  -\frac{8(r^3-1)h_2}{3}\Big( h_1(2re^{12}+\frac{1}{r^2}\omega_3)+h_0(2re^{13}+\frac{1}{r^2}\omega_2)-h_3(2re^{23}+\frac{1}{r^2}\omega_1) \Big),\\
    \tau_3
    &=\imath ((\tau_{27})_{ij})= (\tau_{27})_{ij}(g_r)^{jl}dx^i\wedge (e_l\lrcorner\varphi_r).
\end{split}
 \end{align*}
  where the components $(\tau_{27})_{ij}=\frac{1}{4}\jmath(\tau_3)_{ij}$ define the matrix
  \begin{align*}
  \tau_{27}&=\frac{2}{7r^2}\left(
        \begin{array}{ccc|c}
           r^3p_3(r) & 7r^3(r^3-2)h_0h_3 & -7r^3(r^3-2)h_1h_3  &   \\ 
            7r^3(r^3-2)h_0h_3 & r^3p_0(r) & -7r^3(r^3-2)h_0h_1 &  \\
            -7r^3(r^3-2)h_1h_3 & -7r^3(r^3-2)h_0h_1 & r^3p_1(r)  &   \\ \hline
              &   &   & (\frac{5}{4}(r^3-2)-(5r^3-4)h_2^2)I_{4\times 4}    \\
        \end{array}
    \right)
\end{align*}
with 
$$
p_k(r)=7(r^3-2)h_k^2+(9r^3-10)h_2^2-4(r^3-2),\quad k=0,1,3.
$$ 
 \end{proposition}
\begin{proof}
Taking the induced metric $g_r$ induced by \eqref{eq: isometric_G2},  substituting the equations  \eqref{Eq:diff.varphi.isometric} and \eqref{Eq:star.d.varphi} into the formula \eqref{eq: identities_torsion_forms}, and using the identities   $\omega_1^2=\omega_2^2=\omega_3^2=2e^{4567}$ and $\omega_l\wedge\omega_m=0$ ($l\neq m$), we obtain  
\begin{align}
    \tau_0 =&-\frac{4}{7}(r^3(1-4h_2^2)+(5-8h_2^2))\ast\Big(e^{123}\wedge\frac{\omega_1^2}{2}\Big)=-\frac{4}{7r}\Big(r^3(1-4h_2^2)+(5-8h_2^2)\Big),\nonumber\\
    \tau_1 =&  \frac{1}{12}\ast\Big[\frac{8(r^3+2)}{r}\Big(h_1h_2e^{12}-h_0h_2e^{13}-h_3h_2e^{23}\Big)\wedge\frac{\omega_1^2}{2}\Big] = -\frac{2(r^3+2)h_2}{3}\Big( h_3e^1+h_0e^2-h_1e^3\Big)\nonumber.
\end{align}

Now, having in mind the torsion identity for $\tau_2$ in \eqref{eq: identities_torsion_forms}, we use the above expression for $\tau_1$ to compute: 
\begin{align*}
    \ast( \tau_1\wedge\psi_{r}) &=\frac{2(r^3+2)h_2}{3}\ast\Bigg(\Big(r\big(h_3\omega_1-h_0\omega_2-h_1\omega_3\big)\wedge e^{123}+\frac{1}{r^2}\big(-h_3e^1-h_0e^2+h_1e^3\big)\wedge\frac{\omega^2}{2}\Big)\Bigg)\\
    &= \frac{2(r^3+2)h_2}{3}\Big( -rh_3e^{23}+rh_0e^{13}+rh_1e^{12}+\frac{1}{r^2}h_3\omega_1-\frac{1}{r^2}h_0\omega_2-\frac{1}{r^2}h_1\omega_3\Big).
\end{align*}
On the other hand, we already have an expression for $\ast d\psi_ {r}$ in \eqref{Eq:star.d.psi}, so we obtain
\begin{equation*}
    \tau_2 = -\frac{8(r^3-1)h_2}{3}\Big( h_1(2re^{12}+\frac{1}{r^2}\omega_3)+h_0(2re^{13}+\frac{1}{r^2}\omega_2)-h_3(2re^{23}+\frac{1}{r^2}\omega_1) \Big).
\end{equation*}

Finally, the coefficients of the traceless symmetric $2$-tensor $\tau_{27}$ are
\begin{equation*}
    (\tau_{27})_{ab}:=\frac{1}{4}\jmath(\tau_3)_{ab}=\frac{1}{4}\ast(e_a\lrcorner\varphi_r\wedge e_b\lrcorner\varphi_r\wedge\tau_3),
\end{equation*}
where $\tau_3=\ast d\varphi_r-\tau_0\varphi_r-3\ast(\tau_1\wedge\varphi_r)$ is the torsion $3$-form. In terms of the dual basis to \eqref{Eq:sp(2)_basis}, the torsion $3$-form is  
$$
  \tau_3=\gamma_{123}e^{123}+\sum_{c,d=1}^3\gamma_{cd}e^c\wedge \omega_d,
$$
where the coefficients $\gamma_{123}$ and $\gamma_{cd}$, for $c,d=1,2,3$, are obtained from \eqref{Eq:star.d.varphi} and the previous torsion forms $\tau_0,\tau_1$. 
The explicit computation of $(\tau_{27})_{ab}$ involves  polynomial operations for $h=(h_0,h_1,h_2,h_3)$, subject to the unitary condition $h\bar{h}=1$, which quickly get out of hand. We resorted to the MAPLE computer algebra system to obtain the expression of $\tau_{27}$. However, the computation can be carried out by systematically applying the operator $\jmath$ on the basis elements involved in $\tau_3$ (cf. Lemma \ref{lemma: j_tau_3} in Appendix~\ref{Standard computations}).
\end{proof}

The explicit torsion forms along the family $\Phi(\mathbb{R}^+\times \SO(3))$ lead to our second main result, which describes several special types of $\sptc(2)$-invariant $\gt$-structures for each $\bbS^3$-family with induced metric $g_r$, for some $r\in \mathbb{R}^+$. 

\begin{proof}[Proof of Theorem \ref{Th: isometric_families_theorem}]
    Any $D\in \SO(3)$ is described by $h\in \bbS^3$ using the double cover homomorphism $D=\Upsilon(h)$ from \eqref{Eq: double_cover_SO3}. Also recall that the torsion forms in the present case are explicitly computed in Proposition \ref{prop: h_torsion_forms}. 
    
    Claim \textrm{(i)} follows immediately by solving the equations $\tau_1=0$ and $\tau_2=0$, (i.e. $d\varphi_r=0$), from which  we get
    $$
      h_2=0 \qorq h_0=h_1=h_3=0.
    $$
    For claim \textrm{(ii)}, we further impose $\tau_{27}=0$ (i.e. $d\varphi_r=\tau_0\psi_r$ and $d\psi_r=0$), which implies $r^3=2$, for the case $h_2=0$, and $r^3=2/5$, for the case $h_0=h_1=h_3=0$. Finally, claim \textrm{(iii)} stems from the fact that $d\psi_r=4\tau_1\wedge\psi_r$ implies $d\tau_1\wedge\psi_r=0$. 
    
    To conclude the proof, the torsion $0$-form can only take the following values:
    \begin{align*}
    \tau_0 &= 
    \begin{cases}
        -\frac{4(r^3+5)}{7r}\neq 0,
        &\text{if }h_2=0\\
        \frac{12(r^3+1)}{7r}\neq 0,
        &\text{if }h_0=h_1=h_3=0
    \end{cases}.
    \qedhere
    \end{align*}
\end{proof}

\begin{remark*}\quad
\begin{itemize}
   \item The curvature formula \cite{Besse2007}
    \begin{equation*}
        \langle R(X,Y)Y,X\rangle=-\frac{3}{4}|[X,Y]_{\frakp}|^2-\frac{1}{2}\langle [X,[X,Y]]_{\frakp},Y\rangle -\frac{1}{2}\langle [Y,[Y,X]]_{\frakp},X\rangle +|U(X,Y)|^2-\langle U(X,X),U(Y,Y)\rangle
    \end{equation*}
    implies that the invariant metric induced by the isometric family $(\sqrt[3]{2},h_0,h_1,0,h_3)$ has constant sectional curvature $K(X,Y)=\frac{1}{2}$.
    On the other hand,  Friedrich et al. prove in \cite{friedrich1997nearly}*{\S5} that, on a given $3$-Sasakian manifold $(M^7,g)$, the Berger metric $g^s$, obtained from $g$ by conformal deformation along the $3$-dimensional foliation, is Einstein if, and only if, $s=1$ or $s=1/\sqrt{5}$. Moreover, it is nearly-parallel for $s=1/\sqrt{5}$\,.
    Somewhat similarly, the Einstein metrics $g_{\sqrt[3]{2}}$ and $g_{\sqrt[3]{2/5}}$ induced by the $\gt$-structures $$
    (\sqrt[3]{\frac{2}{5}},0,0,\pm 1, 0)\qandq (\sqrt[3]{2},h_0,h_1,0,h_3)
    $$ 
    are homothetic to the metrics $g^1$ and $g^{\frac{1}{\sqrt{5}}}$ of \cite{friedrich1997nearly}*{Lemma 5.3}, respectively: 
      $$
        g^1=\frac{1}{\sqrt[3]{4}}g_{\sqrt[3]{2}} \qandq g^{\frac{1}{\sqrt{5}}}=\frac{1}{\sqrt[3]{20}}g_{\sqrt[3]{2/5}}\, .
      $$
      \item As a by-product of Proposition \ref{prop: h_torsion_forms},  $\bbS^7$ does not admit any $\sptc(2)$-invariant locally conformally closed $\gt$-structure. However, one of the anonymous referees has pointed out that this actually holds for any compact, connected and simply connected homogeneous space $M$. Namely, if $\varphi$ is an invariant locally conformally closed $\gt$-structure, then $d\varphi=\theta\wedge\varphi$, where $\theta$ is an invariant closed $1$-form, because $\theta=\frac{1}{4}\ast(\varphi\wedge\ast d\varphi)$. Now, since $M$ is simply connected, $\theta$ is actually exact, and hence it vanishes,  therefore, $\varphi$ is closed. However, \cite{podesta2019}*{Corollary 2.2} show that invariant closed $\gt$-structures on compact homogeneous spaces must be parallel.
      \end{itemize}
\end{remark*}

\section{Harmonicity and stability: The Ansatz Case}
\label{Sec: Harmonicity_and_stability}

Building upon the Ansatz established in \S\ref{Subsec: Ansatz_r_G2-structures}, we formulate the natural variational problem associated to the Dirichlet energy on a given isometric class of $\gt$-structures. Our main goal is to examine the behavior of the critical points of the energy functional \eqref{eq: energy_functional}, restricted to an $\sptc(2)$-invariant isometric class, which will culminate in the proof of Theorems \ref{Th: index_nullity_estimates_theorem} and \ref{Th: reduced_energy_stability_theorem}. For questions of stability, we emulate the approach of \cite{Urakawa2013}*{Chapter 5, \S 1}.

The energy of $\varphi_r$ is defined as the $L^2$-norm of its full torsion tensor: 
\begin{align}
\label{eq: energy_functional}
    E(\varphi_r)=\frac{1}{2}\int_{\bbS^7} |T(r)|^2\vol_{\varphi_r} .
\end{align}
Its critical points are \emph{harmonic} $\gt$-structures, characterised by a divergence-free torsion \cite{grigorian2017}*{Corollary 10.3}
$$
  \diver T(r)=0.
$$
The existence of critical points of \eqref{eq: energy_functional}, and specifically of minimisers, has been studied using the associated gradient flow \cites{dwivedi2019,Grigorian2019,loubeau2019}
\begin{equation}
\label{eq: isometric_flow}
    \pdv{\varphi_t}{t}=(\Div T_t)^\sharp\lrcorner \psi_t \qandq \varphi(0)=\varphi_r,
\end{equation}
known as the \emph{isometric flow}, since it preserves isometric classes of $\gt$-structures. 
In particular, Dwivedi et al. \cite{dwivedi2019} proved that \eqref{eq: isometric_flow} is equivalent to 
\begin{equation}
\label{eq: pde_(f,X)}
    \pdv{f}{t}=\frac{1}{2}\langle X, (\Div T_t)^\sharp\rangle \qandq \pdv{X}{t}=-\frac{1}{2}f(\Div T_t)^\sharp+\frac{1}{2}(\Div T_t)^\sharp\times X,
\end{equation}
where $\{\varphi_t=\varphi_{(f,X)}\}$ is the isometric class \eqref{eq: Bryant_formula}.

Specialising to the case of $\sptc(2)$-invariant $\gt$-structures, the pointwise norm $|T(r)|^2$ is computed with respect to the $\sptc(2)$-invariant metric of $\bbS^7$. It is therefore everywhere constant and equal to the norm of the torsion of the $\Ad(\sptc(1))$-invariant $\gt$-structure $\varphi_r\in \Lambda^3(\frakp)^\ast$, hence 
$$
E(\varphi_r)= \frac{1}{2}|T(r)|^2\vol_{\varphi_r}(\bbS^7).
$$
Moreover, the divergence of the full torsion tensor is an $\sptc(2)$-invariant $1$-tensor: 
$$
  (\diver T\lrcorner\psi)(p)=L_x^\ast((dL_{x^{-1}})_x(\diver T_p\lrcorner\psi)(o)=L_x^\ast(\diver T(r)\lrcorner\psi)(o),
$$
where $p=x\cdot\sptc(1)\in \bbS^7$, in particular $o=1_{\sptc(2)}\sptc(1)$ is the orbit of the identity, and $\diver T(r):=\diver T_o\in(\frakp^\ast)^{\Ad(\sptc(1))}$. Hence, when restricted to $\sptc(2)$-invariant solutions, the flow \eqref{eq: isometric_flow} is $\sptc(2)$-invariant, becoming an ODE on $\Lambda_+^3(\frakp^\ast)^{\Ad(\sptc(1))}$ \cite{Lauret2016}*{\S 2}. This allows us to draw general conclusions from the study of critical points of the reduced energy functional $E^{\red} = E|_{\Omega_+^3(\bbS^7)^{\sptc(2)}}$, mediated by the following tailor-made version of Palais' Principle of Symmetric Criticality~\cite{Palais}:

\begin{lemma}
\label{lem: Palais}
    Let $\varphi$ be an $\sptc(2)$-invariant $\gt$-structure. 
    If $\varphi\in\Crit (E^{\red})$, then $\varphi\in\Crit (E)$, i.e. $\diver T_{\varphi} =0$.
\end{lemma}
\begin{proof}
    The flow $\eqref{eq: isometric_flow}$ always exists and is unique for some short time, by \cite{dwivedi2019}*{Theorem 2.12} or \cite{loubeau2019}*{Theorem 1}.
    Since $\sptc(2)$ acts by isometries on $\bbS^7$, and \eqref{eq: isometric_flow} is invariant under this action, the short-time flow, with initial value $\varphi_{0}=\varphi$, will be $\sptc(2)$-invariant.

    Given $\varphi\in\Crit(E^\red)$, for any variation $\varphi_{t}$ of $\varphi$ through $\sptc(2)$-invariant $\gt$-structures with variation vector field $V$, we have
$$
    \left.\frac{d E^\red (\varphi_{t})}{dt}\right\vert_{t=0} = -\int_{M} \langle \diver T_{\varphi}\lrcorner\psi , V\rangle =0.
$$
    If we choose this $\sptc(2)$-invariant variation to be precisely the one given by flow \eqref{eq: isometric_flow} with initial data $\varphi$, we deduce that $\diver T_{\varphi} =0$.
\end{proof}

\begin{proposition}
   The norm of the full torsion tensor \eqref{eq: full_torsion_tensor} of $\varphi_{r}$ is 
   \begin{equation}\label{eq: norm_torsion}
       |T(r)|^2=\frac{1}{r^2}\Big( 4r^6-14r^3+19+8(r^3+2)(r^3-1)h_2^2\Big).
   \end{equation}
\end{proposition}

\begin{proof}
Since each term of the full torsion tensor $T(r)=\frac{\tau_0}{4}g_r-(\tau_1)^\sharp\lrcorner\varphi_r-\frac{1}{2}\tau_2-\tau_{27}$ belongs to an irreducible component of $\frakp\otimes \frakp^\ast=W_1\oplus W_7\oplus W_{14}\oplus W_{27}$, respectively, we have 
\begin{equation*}
    |T(r)|^2=\frac{\tau_0^2}{16}|g_r|^2+|\tau_{27}|^2+|(\tau_1)^\sharp\lrcorner\varphi_r|^2+\frac{1}{4}|\tau_2|^2. 
\end{equation*}
Using the expressions for the torsion forms found in Proposition~\ref{prop: h_torsion_forms}, we have
\begin{equation*}
    \frac{\tau_0^2}{16}|g_r|^2
    = \frac{1}{7r^2}\Big(r^3(1-4h_2^2) +(5-8h_2^2)\Big)^2
    =\frac{1}{7r^2}\Big(r^3+5-4(r^3+2)h_2^2\Big)^2.
\end{equation*}
For the next term, since $\tau_{27}$ is symmetric, we have
\begin{align*}
    |\tau_{27}|^2
    =&\;\frac{4}{49r^2}\Big( p_3(r)^2+49(r^3-2)^2h_0^2h_3^2+49(r^3-2)^2h_1^2h_3^2+49(r^3-2)^2h_0^2h_3^2+p_0(r)^2\\
    &+49(r^3-2)^2h_0^2h_1^2+49(r^3-2)^2h_1^2h_3^2+49(r^3-2)^2h_0^2h_1^2+p_1(r)^2+\frac{25}{16}(r^3-2)^2\\
    &+(5r^3-4)^2h_2^4-\frac{5}{2}(r^3-2)(5r^3-4)h_2^2
    \Big)\\
    =&\;\frac{1}{7r^2}\Big((152r^6-288r^3+160)(1-h_2^2)^2-(200r^6-240r^3+64)(1-h_2^2)+75r^6-60r^3+12
    \Big).
\end{align*}
For the skew-symmetric part of $T(r)$, we use the identity $\varphi_{ajk}\varphi_b^{jk}=6g_{ab}$:
\begin{align*}
    |(\tau_1)^\sharp\lrcorner\varphi_r|^2+\frac{1}{4}|\tau_2|^2
    &= 6|(\tau_1)^\sharp|^2+\frac{16}{9}(r^3-1)^2h_2^2(1-h_2^2)|2re^{12}+\frac{1}{r^2}\omega_3|^2\\
    &=\frac{8}{3r^2}(r^3+2)^2h_2^2(1-h_2^2)+\frac{64}{3r^2}(r^3-1)^2h_2^2(1-h_2^2)\\
    &=\frac{8}{r^2}(3r^6-4r^3+4)h_2^2(1-h_2^2).
\end{align*}
Finally, a simple computation yields the norm of the symmetric part of $T(r)$:
\begin{align*} 
    \frac{\tau_0^2}{16}|g_r|^2+|\tau_{27}|^2
    &=\frac{1}{r^2}\Big((24r^6-32r^3+32)h_2^4-(16r^6-40r^3+48)h_2^2+4r^6-14r^3+19 \Big).
    \qedhere
\end{align*}
\end{proof}

\subsection{Harmonic $\gt$-structures on the $7$-sphere}
\label{sec: divergence}

The aim of this section is to compute the divergence of the full torsion tensor for isometric $\gt$-structures \eqref{eq: isometric_G2} with metric $g_r$, and then find critical points of the energy functional $E$ from \eqref{eq: energy_functional}. We first show that, under our Ansatz, the divergence of the symmetric part of the full torsion tensor will automatically vanish.

The Levi-Civita connection $\nabla$ induced by the corresponding left-invariant metric \cite{Besse2007}*{Proposition 7.28} acts on $\frakp\simeq T_o\bbS^7$ in the following way:
\begin{align}
\label{eq: LC_connection}
    \begin{split}
    \nabla_pY_m
    &=g_\br(\nabla_pY,e_m)=\frac{1}{2}g_\br([e_p,Y]_{\frakp},e_m)+g_\br(U(e_p,Y),e_m)\\
    &=\frac{1}{2}g_\br([e_p,Y]_{\frakp},e_m)+\frac{1}{2}g_\br([e_m,e_p]_{\frakp},Y)+\frac{1}{2}g_\br([e_m,Y]_{\frakp},e_p),
    \end{split}
    \end{align}
where $U:\fp\times\fp\rightarrow\fp$ is defined by 
 \begin{equation}\label{eq: operator_U}
      2g_\br (U(X,Y),Z)=g_{\br}([Z,X]|_{\fp},Y)+g_\br(X,[Z,Y]|_{\fp}),
 \end{equation}
 and $X,Y,Z$ Killing vector fields in $\fp$, note that $U$ is symmetric tensor. We now can introduce the divergence of a $(p,0)$-tensor $\Theta$ as the $(p-1,0)$-tensor
\begin{align*}
    (\diver \Theta)(Y_1,...,Y_{p-1}):= \tr\big(Z\mapsto (\nabla_Z\Theta)(\cdot,Y_1,...,Y_{p-1})\big).
\end{align*}
The full torsion tensor \eqref{eq: full_torsion_tensor}
of $\varphi_{(h_0,X)}$, as well as its divergence, can be expressed in terms of $(h_0,X)$ and $(\varphi_\br,\psi_\br, T_\br)$  \cite{dwivedi2019}*{Lemma 2.9 \& Corollary 2.10}. Hence, using Lemma \ref{lemma: (f,X)_invariant} and Proposition \ref{prop: torsion_forms_r123} we have
\begin{lemma}
    The full torsion tensor of $\varphi_{(h_0,X)}$ is 
    \begin{equation*}
        (T_{(h_0,X)})_{ab}=(1-2|X|^2)(T_\br)_{ab}+2(T_\br)_{am}X_mX_b+2h_0(T_\br)_{am}X_n(\varphi_\br)_{mnb}-2\nabla_aX_mX_n(\varphi_\br)_{mnb}-2h_0\nabla_aX_b,
    \end{equation*}
    and its divergence is
    \begin{align}
    \label{eq: divergence_in_coordinates}
    \begin{split}
        (\diver T_{(h_0,X)})_q
        =&\; 2(T_\br)_{pm}X_m\nabla_pX_q +2h_0(T_{\br})_{pl}\nabla_pX_m(\varphi_\br)_{lmq} -2\nabla_p\nabla_pX_lX_m(\varphi_\br)_{lmq}\\
        &-2\nabla_pX_lX_m(T_\br)_{ps}(\psi_\br)_{slmq} -2h_0\nabla_p\nabla_pX_q,
    \end{split}
\end{align}
\end{lemma}
 Note that, as $\diver T_{(h_0,X)}$ is an $\ad(\fraksp(1))$-invariant $1$-tensor, it must belong to  $\frakp_1\oplus \frakp_2\oplus \frakp_3$.

\begin{lemma}
\label{Lema.symmetric.S} 
    Any $\sptc(2)$-invariant symmetric $2$-tensor is divergence-free, with respect to the metric 
    $g_r$ induced by the $\gt$-structure $\varphi_r$ from \eqref{eq: isometric_G2}. 
\end{lemma}
\begin{proof}
The divergence of a $2$-tensor $S$ is, by definition, the $1$-tensor with coefficients $(\diver S)_j=\nabla_i S_{ij}$, 
where $\nabla$ is the Levi-Civita connection \eqref{eq: LC_connection} of $g_r$. If $S$ is $\sptc(2)$-invariant then so is $\diver S$, and we only need computing it at the point $o\in \bbS^7$. Hence, under the identification $T_o\bbS^7\simeq \frakp$, we have that $\diver S|_{\frakp_4}=0$ since $(\frakp^\ast)^{\Ad(\sptc(1))}=\frakp_1^\ast\oplus\frakp_2^\ast\oplus\frakp_3^\ast$. The basis \eqref{Eq:sp(2)_basis} in $\fp$  is identified with a frame in $T\bbS^7$ generated by the one-parameter subgroups $\exp(te_i)\subset \sptc(2)$. Thus, at $o\in \bbS^7$, we have:
\begin{align}\label{eq: coefficient.divS_j}
    (\Div S)_j = & \nabla_iS_{ij}=(\nabla_iS)(e_i,e_j)=e_i(S_{ij})-S(\nabla_ie_i,e_j)-S(e_i,\nabla_ie_j) \qforq j=1,2,3,
\end{align}
where 
$$\nabla_ie_j=-\frac{1}{2}[e_i,e_j]_\frakp+U(e_i,e_j) \qandq  2g_r(U(e_i,e_j),e_k)=g_r([e_i,e_k]_\frakp,e_j)+g_r([e_j,e_k]_\frakp,e_i).$$
For the first term of \eqref{eq: coefficient.divS_j} we have:
\begin{align*}
    e_i(S_{ij})=&\frac{d}{dt}|_{t=0}S\big((d\lambda_{e^{te_i}})_oe_i,(d\lambda_{e^{te_i}})_o(\exp(-t\ad(e_i)(e_j))_\frakp\big)_{e^{te_i}o}\\
    =&\frac{d}{dt}|_{t=0}S\big(e_i,(\exp(-t \ad(e_i)(e_j))_\frakp\big)_o =-S(e_i,[e_i,e_j]_\frakp).
\end{align*}
The middle term of \eqref{eq: coefficient.divS_j} is:
\begin{align*}
    S(\nabla_ie_i,e_j)=&S(U(e_i,e_i),e_j)=g_r([e_i,e_k]_\frakp,e_i)S(e_k,e_j),
\end{align*}
where $j\in\{1,2,3\}$. As $g_r$ is $\sptc(2)$-invariant, it can be written $g_r(u,v)=\langle Au,v \rangle$,
with $A$ an $\Ad(\sptc(1))$-invariant symmetric (with respect to (minus) the Killing form) positive definite matrix on $\frakp$, so 
$$
\sum_{i} g_r([e_i,e_k]_\frakp,e_i) = \tr (A \, \ad_{e_{k}}) = - \tr (\ad_{e_{k}} \, A) = 0.
$$
As to the third term of \eqref{eq: coefficient.divS_j}, 
since $S$ is symmetric, 
\begin{align*}
    S(e_i,\nabla_ie_j)=&-\frac{1}{2}S(e_i,[e_i,e_j]_\frakp)+\frac{1}{2}\big(g_r([e_i,e_k]_\frakp,e_j)+g_r([e_j,e_k]_\frakp,e_i)\big)S(e_i,e_k)\\
    =&-\frac{1}{2}S(e_i,[e_i,e_j]_\frakp),
\end{align*}
 and writing $S(e_i,e_j)=:g_r(\beta(e_i),e_j)$ for $\beta$ a $g_r$-self-adjoint linear operator on $\frakp$, we have:
\begin{align*}
    S(e_i,[e_i,e_j]_\frakp)=&g_r(\beta(e_i),[e_i,e_j])=-\tr_{g_r}(\ad(e_j)\beta)=0,
\end{align*}
therefore \eqref{eq: coefficient.divS_j} is zero.

\end{proof}

Together with the Palais Lemma \ref{lem: Palais}, Proposition \ref{prop: divergence} below immediately implies Theorem \ref{Th: critical_points_theorem}: 

\begin{proposition}
\label{prop: divergence}
    The full torsion tensor $T(r)$ of the $\gt$-structure \eqref{eq: isometric_G2} has divergence 
    \begin{equation*}
       \diver T(r) =\frac{4(r^3+2)(r^3-1)h_2}{r}\Big(h_3e^1+h_0e^2-h_1e^3\Big).
    \end{equation*}
    In particular, in each of the isometric classes $\cB_r$ (see Theorem \ref{Th: isometric_families_theorem}), the families $\{1\}\times \bbRP^3$, $\{r\}\times \bbRP^2$ and $\{r\}\times \rN\rS$ are exactly the critical points of the energy functional \eqref{eq: energy_functional}.
\end{proposition}
\begin{proof} 
    Since the full torsion tensor \eqref{eq: full_torsion_tensor} of $\varphi_r$ is $\Ad(\sptc(1))$-invariant, we know from Lemma \ref{Lema.symmetric.S} that its symmetric part  (i.e. $\tau_0$ and $\tau_3$) is divergence-free. Hence
\begin{align*}
    \diver T&= -\delta\ast(\tau_1\wedge\psi_{r})-\frac{1}{2}\delta\tau_2=-\ast d(\tau_1\wedge\psi_{r})+\frac{1}{2}\ast d(\tau_2\wedge \varphi_{r})\\
    &= \frac{4(r^3+2)(r^3-1)}{3r}\Big(h_2h_3e^1+h_0h_2e^2-h_1h_2e^3\Big)+\frac{8(r^3+2)(r^3-1)}{3r}\Big(h_2h_3e^1+h_0h_2e^2-h_1h_2e^3\Big)\\
    &= \frac{4(r^3+2)(r^3-1)h_2}{r}\Big(h_3e^1+h_0e^2-h_1e^3\Big).
\end{align*}
    The critical points of the energy functional \eqref{eq: energy_functional} are parametrised by solutions of
\begin{align*}
    &|\diver T(r)|^2
    =\frac{16(r^3+2)^2(r^3-1)^3}{r^4}h_2^2(h_0^2+h_1^2+h_3^2)=0,\\
    \Leftrightarrow\quad
    & r=1 \quad\text{or}\quad 
    h_2=0 \quad\text{or}\quad h_0=h_1=h_3=0.
\end{align*}
    These cases correspond to the three families in the second part of the statement.
\end{proof}

An alternative method to characterise critical points of the energy functional \eqref{eq: energy_functional} relies on considering  isometric variations of $\varphi_r$ along paths of $\Ad(\sptc(1))$-invariant $\gt$-structures. This will be helpful when computations of quantities such as the divergence of $T$ get too complicated, as in \S\ref{Sec: Harmonicity_and_stability: The General Case}.

\begin{remark}\label{rem: alternative_critical_points_method}
    Consider the smooth curve defined by
    $$
    k(t)=(k_0(t),k_1(t),k_2(t),k_3(t))\in \bbS^3,
    \quad {t\in (-\delta,\delta)}.
    $$
    By Proposition \ref{prop: isometric_equivalence}, it describes a path inside the  $r$-family of $\Ad(\sptc(1))$-invariant $\gt$-structures $\varphi_r(t)=\Phi(r,\Upsilon(\overline{k(t)h}))$, with initial condition $\varphi_r(0)=\varphi_r=\Phi(r,\Upsilon(\bar{h}))$ under the isomorphism \eqref{Eq: Phi_map}. In particular, 
\begin{equation*}
    \varphi_r(t)=r^3e^{123}+\Upsilon(k(t)h)^\ast e^1\wedge\omega_1+\Upsilon(k(t)h)^\ast e^2\wedge\omega_2+\Upsilon(k(t)h)^\ast e^3\wedge\omega_3, 
    \qwithq k(0)=(1,0,0,0),
\end{equation*}
    hence, by \eqref{eq: norm_torsion} the norm of the full torsion tensor of $\varphi_r(t)$ is 
\begin{equation*}
    |T_t(r)|^2=\frac{1}{r^2}\Big( 4r^6-14r^3+19+8(r^3+2)(r^3-1)(k_0(t)h_2-k_1(t)h_3+k_2(t)h_0+k_3(t)h_1)^2\Big).
\end{equation*}
    Since $k(0)=(1,0,0,0)$, we necessarily have $k'_0(0)=0$, and the derivative of $|T_t(r)|^2$ at $t=0$ is
\begin{equation*}
    \left.\frac{d}{dt}|T_t(r)|^2\right\vert_{t=0}
    =\frac{16(r^3+2)(r^3-1)}{r^2}h_2\Big(-k'_1(0)h_3+k'_2(0)h_0+k'_3(0)h_1\Big).
 \end{equation*}
    Choosing the path $k(t)=(\cos(t),\sin(t),0,0)$, the condition $\frac{d}{dt}|_{t=0}|T_t(r)|^2=0$ implies that $(r,h_0,h_1,h_2,h_3)$ is one of the quintuples $$(1,h_0,h_1,h_2,h_3), \quad (r,h_0,h_1,0,h_3) \qorq (r,h_0,h_1,h_2, 0).$$
    Similarly, for the variation $(\cos(t),0,\sin(t),0)$, we have
\begin{equation*}
    (1,h_0,h_1,h_2,h_3), \quad (r,h_0,h_1,0,h_3) \qorq (r,0,h_1,h_2,h_3) 
\end{equation*}
    and for the variation $(\cos(t),0,0,\sin(t))$, we obtain 
\begin{equation*}
    (1,h_0,h_1,h_2,h_3), \quad (r,h_0,h_1,0,h_3) \qorq (r,h_0,0,h_2,h_3).
\end{equation*}
    In summary, for any isometric variation $\varphi_r(t)$, critical points of $E^\red(\varphi_r)$ are of the forms $(1,h_0,h_1,h_2,h_3)$, $(r,h_0,h_1,0,h_3)$ and $(r,0,0,\pm 1, 0)$. By Lemma \ref{lem: Palais} those are actual critical points of $E$. 
 \end{remark}

In particular, Remark \ref{rem: alternative_critical_points_method}  provides an alternative proof of Proposition \ref{prop: divergence}. 
 
  Recall that the isometric family $\varphi_r$ is parametrised by $\Phi(r,\Upsilon(\bar{h}))$. On the other hand, the critical points in the isometric families parametrised by $\Phi(r,\Upsilon(\overline{ih}))$, $\Phi(r,\Upsilon(\overline{jh}))$ and $\Phi(r,\Upsilon(\overline{kh}))$  are given by the parameters 
  \begin{align*}
      (1,h_0,h_1,h_2,h_3), \quad (r,h_0,h_1,h_2,0) \qandq (r,0,0,0,\pm 1),\\
      (1,h_0,h_1,h_2,h_3), \quad (r,0,h_1,h_2,h_3) \qandq (r,\pm 1,0,0,0),
  \end{align*}
  and 
      $$
      (1,h_0,h_1,h_2,h_3), \quad (r,h_0,0,h_2,h_3) \qandq (r,0,\pm 1,0,0).
      $$
 By Proposition \ref{prop: isometric_equivalence}  the families $\varphi_r=\Phi(r,\Upsilon(\bar{h}))$, $\Phi(r,\Upsilon(\overline{ih}))$, $\Phi(r,\Upsilon(\overline{jh}))$ and $\Phi(r,\Upsilon(\overline{kh}))$ are isometric, so the critical points of Proposition \ref{prop: divergence} are unique up to isometries induced by the $\sptc(1)$-action, cf. \eqref{eq: block_decomposition} in Proposition \ref{prop: block_decomposition}.

\subsection{The $\Ad(\sptc(1))$-invariant gradient flow of $E$}
\label{subsec: flow_ansatz_case}

Let $\{\varphi_t\}_{t\in (-\varepsilon,\varepsilon)}$ be a solution of the gradient flow \eqref{eq: isometric_flow} of the Dirichlet energy functional \eqref{eq: energy_functional}.
By Proposition \ref{prop: isometric_equivalence}, a solution with initial condition $\varphi_r$, as in the Ansatz \eqref{eq: isometric_G2}, is parametrised by  $$
\varphi_t=\Phi(r,\overline{k(t)h}),
\qwithq h,k(t)\in \sptc(1)
\qandq k(0)=(1,0,0,0).
$$ 
Letting 
$$
m(t):=k(t)h=(m_0,m_1,m_2,m_3)\in \sptc(1),
$$ 
the divergence of the full torsion tensor of $\varphi_t$ is 
\begin{equation*}
    \Div T_t:=\Div T_{m(t)}=\frac{4(r^3+2)(r^3-1)m_2}{r}\Big(m_3e^1+m_0e^2-m_1e^3\Big).
\end{equation*}
Applying Lemma \ref{lemma: (f,X)_invariant} to the equations in \eqref{eq: pde_(f,X)}, we find the explicit evolution of each component of $m(t)$:
\begin{align}
\label{eq: system_of_ode}
    \begin{split}
    \odv{m_0}{t}=&\frac{2(r^3+2)(r^3-1)m_2}{r^2}(m_0m_2),\\
    \odv{m_1}{t}=&\frac{2(r^3+2)(r^3-1)m_2}{r^2}(m_1m_2),\\
    \odv{m_2}{t}=&\frac{2(r^3+2)(r^3-1)m_2}{r^2}(-m_0^2-m_1^2-m_3^2),\\
    \odv{m_3}{t}=&\frac{2(r^3+2)(r^3-1)m_2}{r^2}(m_3m_2).
    \end{split}
\end{align}
\begin{remark}
\label{rem: short_time_existence_remark}
    In  matrix form, the system \eqref{eq: system_of_ode} is a first-order non-linear ODE,
\begin{equation}
\label{eq: matrix_ode}
    \odv{m(t)}{t}
    =\frac{2(r^3+2)(r^3-1)m_2}{r^2}
    \left(\begin{array}{cccc}
        0 & 0 & m_0 & 0 \\
    0 & 0 & m_1 & 0\\
    -m_0 & -m_1 & 0 & -m_3\\
    0 & 0 & m_3 & 0
    \end{array}\right)
    m(t)=\Psi(m(t)).
\end{equation}
    The short-time existence and uniqueness of solutions of \eqref{eq: matrix_ode}, with any initial value $m(0)=h\in \sptc(1)$, follows from the Picard-Lindelöf Theorem, since $\Psi: \H\rightarrow \H$ is locally Lipschitz (polynomial, de facto). Namely, 
\begin{equation*}
    m(t)=h+\int_0^t\Psi(m(s))ds, \quad -\varepsilon< t<\varepsilon.
\end{equation*}
    This agrees with the short-time existence results in \cite{dwivedi2019}*{Theorem 2.12} and \cite{loubeau2019}*{Theorem 1}.
\end{remark}

\begin{proposition}
\label{prop: ode_explicit_solution}
 Given any initial value $m(0)=h=(h_0,h_1,h_2,h_3)\in \sptc(1)$,  the ODE \eqref{eq: matrix_ode} admits the following unique solution, defined for all $t\in\R$:
\begin{subequations}
\begin{align}\label{eq: equation m_2}
    m_2(t)=&\frac{h_2}{\sqrt{(1-h_2^2)e^{\frac{4(r^3+2)(r^3-1)t}{r^2}}+h_2^2}},\\ \label{eq: equation m_k}
    m_k(t)=&\frac{h_ke^{\frac{2(r^3+2)(r^3-1)t}{r^2}}}{\sqrt{(1-h_2^2)e^{\frac{4(r^3+2)(r^3-1)t}{r^2}}+h_2^2}}, \qforq k=0,1,3.
\end{align}
\end{subequations}
\end{proposition}

\begin{proof}
Since the matrix 
$$
  A:=\frac{2(r^3+2)(r^3-1)m_2}{r^2}\left(\begin{array}{cccc}
    0 & 0 & m_0 & 0 \\
    0 & 0 & m_1 & 0\\
    -m_0 & -m_1 & 0 & -m_3\\
    0 & 0 & m_3 & 0
\end{array}\right)
$$
is skew-symmetric, the norm $|m(t)|$ is constant: 
\begin{align*}
    \tfrac12 \odv{|m(t)|^2}{t}
    &= \left\langle \odv{m(t)}{t} , m(t) \right\rangle 
    = \langle Am(t) , m(t) \rangle 
    =0.
\end{align*}
Therefore any short-time solution of \eqref{eq: matrix_ode} can be extended for all $t\in \mathbb{R}$.

In particular, if the initial condition $m(0)=h$ has $h_2=0$, by uniqueness, the solution of \eqref{eq: matrix_ode} must be the constant map $m(t)=(h_0,h_1,0,h_3)$. Similarly, if $h_2=\pm1$, then $h\in \sptc(1)$ and the solution is the constant map $m(t)=(0,0,1,0)$.

Now, suppose that  $m(0)=h$ satisfies $0<h_2<1$. We claim that 
\begin{equation}
\label{eq: estimates_of_m2}
    0<m_2(t)<1, 
    \quad\forall t\in \R.
\end{equation}
Indeed, if there were an instant  $t_0\in \R$ such that $m_2(t_0)=0$, then  $\widetilde{m}(t)=m(t+t_0)$ would be a solution of \eqref{eq: matrix_ode}, with initial condition $\widetilde{m}(0)=(m_0(t_0),m_1(t_0),0,m_3(t_0))$. Therefore, by uniqueness, $\widetilde{m}(t)=(m_0(t_0),m_1(t_0),0,m_3(t_0))$, so $m(t)$ would be the constant solution. On the other hand, if there exists $t_0\in \R$ such that $m_2(t_0)=1$, then $\widetilde{m}(t)=m(t+t_0)$ would be a solution of \eqref{eq: matrix_ode}, with initial condition $\widetilde{m}(0)=(0,0,1,0)$. Again, by uniqueness, $\widetilde{m}(t)=(0,0,1,0)$, so $m(t)$ would be constant.    

Let $m(t)$ be henceforth a non-constant solution of \eqref{eq: matrix_ode}, such that the map $m_2(t)$ satisfies \eqref{eq: estimates_of_m2}. Then, we can write the ODE system \eqref{eq: system_of_ode} as
\begin{equation}\label{eq: separables_ode's}
    \frac{1}{m_2(1-m_2^2)}\odv{m_2}{t}=\frac{2(r^3+2)(r^3-1)}{r^2} \qandq \frac{1}{m_k}\odv{m_k}{t}=\frac{2(r^3+2)(r^3-1)}{r^2}m_2^2 \qforq k=0,1,3.
\end{equation}
A simple computation for the first ODE of \eqref{eq: separables_ode's}, gives us \eqref{eq: equation m_2}. 
Replacing $m_2(t)$ into the second ODE of \eqref{eq: separables_ode's}, we obtain \eqref{eq: equation m_k}.

Finally, if the initial condition $m(0)=h$ satisfies $-1<h_2<0$, the previous arguments apply to the solution $\widetilde{m}(t)=-m(t)$.
\end{proof}

The next result describes the landscape of limiting harmonic homogeneous $\rG_2$-structures for the flow \eqref{eq: system_of_ode} at infinity. This  can be seen as a concrete instance of the general theory \cite{loubeau2019}*{Theorem 3 \& Remark 22}, which predicts subsequential convergence to a harmonic limit, under uniformly bounded torsion, hence long-time existence for homogeneous structures. However, in the present context, we get to be much more precise. Since, ultimately, we are studying a gradient flow, solutions are expected to flow, forwards and backwards, between critical regions, in this case, the equatorial  $\bbRP^2$ and $\rN\rS$. This theorem tells us exactly how such gradient flow lines behave: 

\begin{theorem}
\label{th: convergent_subseq}
    Given an initial condition $m(0)\in \sptc(1)$, let $m(t)$ be the all-time solution of \eqref{eq: matrix_ode} from Proposition \ref{prop: ode_explicit_solution}, and let $\varphi(t)=\Phi(r,\Upsilon(\overline{m(t)})$ be the corresponding solution of the isometric flow \eqref{eq: isometric_flow}. Then there exist subsequences  $t_{n_k^\pm}\rightarrow \pm\infty$ such that 
    $m(t_{n_k^\pm})\to m_{\pm\infty}\in\sptc(1)$, and
$$
    \varphi_{\pm\infty}
    =\Phi(r,\Upsilon(\overline{m}_{\pm\infty}))
$$ 
are harmonic $\Ad(\sptc(1))$-invariant coclosed $\gt$-structures. According to the parameter $r>0$ 
, the flow on $\bbRP^3$ behaves asymptotically as follows:

\noindent\underline{$r < 1$:} \quad $m(-\infty)\in \{r\}\times \bbRP^2$ and $m(+\infty)=[(0,0,\pm1,0)]\in \{r\}\times \rN\rS$.

\noindent\underline{$r > 1$:} \quad
$m(+\infty)\in \{r\}\times\bbRP^2$ and  $m(-\infty)=[(0,0,\pm1,0)]\in \{r\}\times \rN\rS$. 
\end{theorem}

\begin{proof}
    Let $m(t)$ be  solution of \eqref{eq: matrix_ode}. Then, we have
\begin{align*}
    \Bigg|\odv{m(t)}{t}\Bigg|^2&=\Bigg(\odv{m_0(t)}{t}\Bigg)^2+\Bigg(\odv{m_1(t)}{t}\Bigg)^2+\Bigg(\odv{m_2(t)}{t}\Bigg)^2+\Bigg(\odv{m_3(t)}{t}\Bigg)^2\\
    &=\frac{4(r^3+2)^2(r^3-1)^2}{r^4}m_2(t)^2(1-m_2(t)^2)=|\diver T(t)|^2.
\end{align*}

Consequently, we obtain the inequalities
\begin{subequations}
\begin{align}\label{eq: inequalities_norm_derivative_m2}
    \Bigg|\odv{m(t)}{t}\Bigg|^2\leq& \frac{4(r^3+2)^2(r^3-1)^2}{r^4}m_2(t)^2\\ \label{eq: inequalities_norm_derivative_1-m2} \Bigg|\odv{m(t)}{t}\Bigg|^2\leq& \frac{4(r^3+2)^2(r^3-1)^2}{r^4}(1-m_2(t)^2).
\end{align}
\end{subequations}

    Using the expression \eqref{eq: equation m_2} of $m_2$, we have the limits
    $$
      \begin{array}{c|c|c}
           & t\rightarrow -\infty & t\rightarrow\infty \\ \hline
          r<1 & m_2^2\rightarrow 0 & m_2^2\rightarrow 1 \\ \hline
          r>1 &  m_2^2\rightarrow 1 & m_2^2\rightarrow 0.
      \end{array}
    $$
    Thus, from the inequalities \eqref{eq: inequalities_norm_derivative_m2} and \eqref{eq: inequalities_norm_derivative_1-m2}, we obtain that $m'(t)$ converges to zero, it means that $\varphi(t)=\Phi(r,\Upsilon(\overline{m(t)}))$
     converges to harmonic $\gt$-structure at the infinity. Finally, the conclusion for the limit structures follows from Theorem \ref{Th: isometric_families_theorem}.
\end{proof}

\subsection{Second variation}
\label{Sec:2nd_variation}

Let $(\varphi_{(t,s)})_{(t,s)\in (-\varepsilon,\varepsilon)\times (-\varepsilon,\varepsilon)}$ be an $\Ad(\sptc(1))$-invariant variation within the isometric class of $\varphi_r$ with $\varphi(0,0)=\varphi_r=\Phi(r,\Upsilon(\bar{h}))$ and 
\begin{equation}\label{Eq: field_V_W}
    \frac{\partial \varphi_{(t,s)}}{\partial t}|_{t=s=0}=V\lrcorner\psi_r \qandq \frac{\partial \varphi_{(t,s)}}{\partial s}|_{t=s=0}=W\lrcorner\psi_r.
\end{equation}
By Theorem \ref{Th: Phi_isomorphism_theorem}, the  $3$-form $\varphi(t,s)$ corresponds to the point $k(t,s)\in \bbS^3$. Explicitly,
\begin{equation*}
     \varphi_r(t,s)=\Phi(r,\Upsilon(\overline{k(t,s)h})=r^3e^{123}+\Upsilon(k(t,s)h)^\ast e^1\wedge\omega_1+\Upsilon(k(t,s)h)^\ast e^2\wedge\omega_2+\Upsilon(k(t,s)h)^\ast e^3\wedge\omega_3,
 \end{equation*}
setting $k(0,0):=(1,0,0,0)$
, i.e.,
\begin{equation*}
    k_0(0,0)=1 \qandq k_1(0,0)=k_2(0,0)=k_3(0,0)=0.
\end{equation*}
The initial condition for $k(t,s)$ implies that the variation vector fields given in \eqref{Eq: field_V_W} correspond to the contraction of the $\Ad(\sptc(1))$-invariant vectors 
    \begin{align}\label{eq: expression_V_W}
    \begin{split}
        V=&\frac{2}{r}\Big( \frac{\partial k_1}{\partial t}(0,0)e_1+\frac{\partial k_2}{\partial t}(0,0)e_2+\frac{\partial k_3}{\partial t}(0,0)e_3 \Big), \\
        W=&\frac{2}{r}\Big( \frac{\partial k_1}{\partial s}(0,0)e_1+\frac{\partial k_2}{\partial s}(0,0)e_2+\frac{\partial k_3}{\partial s}(0,0)e_3 \Big),
        \end{split}
    \end{align}
    into the $4$-form $\psi_r$ from \eqref{eq: psi_r}. The expression for the vectors \eqref{eq: expression_V_W} follows from Lemma \ref{lemma: (f,X)_invariant}, since the variation $\varphi(t,s)$ is given by
    $$
      \varphi(t,s)=(k_0(t,s)^2-|X(t,s)|^2)\varphi_r-2k_0(t,s)X(t,s)\lrcorner\psi_r+2X(t,s)^\flat\wedge X(t,s)\lrcorner\varphi_r,
    $$
    where $X(t,s)=\frac{1}{r}(k_1(t,s)e_1+k_2(t,s)e_2+k_3(t,s)e_3)$.
    
    Moreover, since the variation $k(t,s)$ satisfies $\frac{\partial k_0}{\partial t}(0,0)=\frac{\partial k_0}{\partial s}(0,0)=0$, the second derivative of $k_0$ at $(0,0)$ satisfies 
\begin{align}\label{eq: second_derivative_k_0}
\begin{split}
    \frac{\partial^2 k_0}{\partial t\partial s}(0,0)=&-\bigg(\frac{\partial k_1}{\partial t}(0,0)\frac{\partial k_1}{\partial s}(0,0)+\frac{\partial k_2}{\partial t}(0,0)\frac{\partial k_2}{\partial s}(0,0)+\frac{\partial k_3}{\partial t}(0,0)\frac{\partial k_3}{\partial s}(0,0) \bigg)\\
   =&-\frac{1}{4}g_r(V,W),
   \end{split}
\end{align}
where $g_r$ is the induced $\gt$-inner product \eqref{Eq:metric}. Then, the second derivative of $|T_{(t,s)}(r)|^2$ at $t=s=0$ defines a symmetric bilinear form 
$$
\hess(E)_{r,h}(V,W):= \pdv{}{s,t}|T_{(t,s)}(r)|^2.
$$
Notice that $\hess(E)_{r,h}$ is a bilinear for on $\frakp_1\oplus\frakp_2\oplus\frakp_3$. Following standard harmonic map theory, we define the reduced nullity and the reduced index of $\hess(E)_{r,h}$:

\begin{definition}
\label{def: index_nullity}
    Let $\varphi_r=\Phi(r,h)$ be a harmonic $\Ad(\sptc(1))$-invariant $\gt$-structure. 
\begin{itemize}
    \item The \emph{reduced index} of $\varphi_r$, denoted by $\ind(r,h)$, is the dimension of the largest subspace of $\frakp_1\oplus\frakp_2\oplus\frakp_3$ on which $\hess(E)_{r,h}$ is negative-definite. 
    \item The \emph{reduced nullity} of $\varphi_r$, denoted by $\nulli(r,h)$, is the dimension of the subspace of $\frakp_1\oplus\frakp_2\oplus\frakp_3$,  where $\hess(E)_{r,h}$ is degenerate. 
\end{itemize}
    We say that  $\Phi(r,h)$ is \emph{stable} along $\Ad(\sptc(1))$-invariant variations if  $\ind(r,h)=0$, and \emph{unstable} otherwise.
\end{definition}

We are particularly interested in those points such that $(r,h)\in \R^+\times\sptc(1)$ defines a harmonic $\gt$-structure.

\begin{lemma}
\label{lemma: second variation critical points}
    Let $\varphi_r=\Phi(r,\Upsilon(\bar{h}))$ be a harmonic $\Ad(\sptc(1))$-invariant $\gt$-structure, and denote by $T_{(t,s)}(1)$, $T_{(t,s)}^{\bbS^2}(r)$ and $T_{(t,s)}^\pm(r)$ the full torsion tensors associated to variations $(\varphi_{(t,s)})_{(t,s)}$ of the critical points $(1,h_0,h_1,h_2,h_3)$, $(r,h_0,h_1,0,h_3)$ and $(r,0,0,\pm 1, 0)$. Then, the second variation of the reduced energy at $\varphi_r$ is given by 
\begin{subequations}
\begin{align}\label{eq: hess r=1}
     \hess(E)_{1,h}=&\pdv{}{s,t}|T_{(0,0)}(1)|^2=0,\\ \label{eq: hess h_2=1} \hess(E)_{r,\pm}=&\pdv{}{s,t}|T_{(0,0)}^\pm(r)|^2=-\frac{4(r^3+2)(r^3-1)}{r^2}g_r(V,W),\\ \label{eq: hess h_2=0} \hess(E)_{r,\bbS^2}=&\pdv{}{s,t}|T_{(0,0)}^{\bbS^2}(r)|^2=\frac{4(r^3+2)(r^3-1)}{r^2}g_r(J_h(V),W),
\end{align}
\end{subequations}
where
\begin{equation*}
    J_h=\left(\begin{array}{ccc}
        h_3^2 & -h_0h_3 & -h_1h_3 \\
        -h_0h_3 & h_0^2 & h_0h_1 \\
        -h_1h_3 & h_0h_1 & h_1^2
    \end{array}\right).
\end{equation*}
\end{lemma}

\begin{proof}
Consider the $\Ad(\sptc(1))$-invariant variation $(\varphi_{(t,s)})_{(t,s)\in (-\varepsilon,\varepsilon)\times (-\varepsilon,\varepsilon)}$ within the isometric class of $\varphi_r$, with variation vector fields \eqref{Eq: field_V_W}
\begin{align}
\label{eq: 2nd_variation_ts}
\begin{split}
    \hess(E)_{r,h}(V,W) :=\quad &\left.\pdv{}{s,t}\right\vert_{t=s=0}|T_{(t,s)}(r)|^2\\
    :=&\frac{16(r^3+2)(r^3-1)}{r^2}\Bigg[ \Big( -\pdv{k_1}{s}(0,0)h_3+\pdv{k_2}{s}(0,0)h_0+\pdv{k_3}{s}(0,0)h_1\Big)\\
    &\Big(-\pdv{k_1}{t}(0,0)h_3+\pdv{k_2}{t}(0,0)h_0+\pdv{k_3}{t}(0,0)h_1\Big)+h_2\Big( \pdv{k_0}{s,t}(0,0)h_2\\
    &-\pdv{k_1}{s,t}(0,0)h_3+\pdv{k_2}{s,t}(0,0)h_0+\pdv{k_3}{s,t}(0,0)h_1\Big) \Bigg].
    \end{split}
\end{align}
    Replacing the values $(1,h_0,h_1,h_2,h_3)$, $(r,0,0,\pm1,0)$ and $(r,h_0,h_1,0,h_3)$, into \eqref{eq: 2nd_variation_ts} we obtain the bilinear forms \eqref{eq: hess r=1}, \eqref{eq: hess h_2=1} and \eqref{eq: hess h_2=0}, respectively. 
\end{proof}

Theorems \ref{Th: reduced_energy_stability_theorem} and \ref{Th: index_nullity_estimates_theorem} establish precisely the stability profile of critical points, for variations through $\sptc(2)$-invariant $\gt$-structures. We first check the critical points detected in Proposition~\ref{prop: divergence}, in accordance with Definition \ref{def: index_nullity}:

\begin{proof}[Proof of Theorem \ref{Th: reduced_energy_stability_theorem}]
From Proposition \ref{prop: divergence}, we know that $\{1\}\times \bbRP^3$, $\{r\}\times \rN\rS$ and $\{r\}\times \bbRP^2$ are sets of critical points of the energy functional \eqref{eq: energy_functional} among $\sptc(2)$-invariant $\gt$-structures. We examine each one of them, in light of the Hessian computations of Lemma \ref{lemma: second variation critical points}.

For the family $\{1\}\times \bbRP^3$, we see from \eqref{eq: hess r=1} that $\hess(E)_{1,h}=0$, therefore $\nulli(1,h)=\dim\bbRP^3=3$ and $\ind(1,h)=0$, for any $h\in \bbRP^3$.

For the family $\{r\}\times \rN\rS$, Equation \eqref{eq: hess h_2=1} implies that 
\begin{align}\label{eq: null_index of poles}
    \nulli(r,\rN\rS)=0 
    \quad \forall r\in \R^+\smallsetminus\{1\} 
    &\qandq \ind(r,\rN\rS)=
    \begin{cases}
        0 \qforq r<1 ,\\
        3 \qforq r>1 .
    \end{cases}
\end{align}

Finally, for the family $\{r\}\times \bbRP^2$, the bilinear form \eqref{eq: hess h_2=0} has eigenvalues $\lambda_1=\lambda_2=0$ and $\lambda_3=1$, so
\begin{align}\label{eq: null_index of equator}
    \nulli(r,\bbRP^2)=2, 
    \quad \forall r\in \R^+\smallsetminus\{1\} 
    &\qandq \ind(r,\rN\rS)=
    \begin{cases}
        1 \qforq r<1 ,\\
        0 \qforq r>1 .
    \end{cases}
\end{align}

In terms of Definition \ref{def: index_nullity}, the family $\{1\}\times \bbRP^3$ is $E^\red$-stable, since $\ind(1,h)=0$ and the other two, from \eqref{eq: null_index of poles} and \eqref{eq: null_index of equator}, we have
$$
\begin{array}{|c|c|c|c|c|c|c|}\hline
     & \ind(r,\rN\rS) & \ind(r,\bbRP^2) & \nulli(r,\rN\rS) & \null(r,\bbRP^2) & (r,\rN\rS) & (r,\bbRP^2)  \\ \hline
    r<1 & 0 & 1 & 0 & 2 & \text{$E^\red$-stable} & \text{$E^\red$-unstable} \\ \hline
    r>1 & 3 & 0 & 0 & 2 &\text{$E^\red$-unstable} & \text{$E^\red$-stable} \\ \hline
\end{array}
$$
\end{proof}

Mind that, since \eqref{eq: 2nd_variation_ts} is only valid for variations through $\Ad(\sptc(1))$-invariant structures, these stability properties hold only for the class of $\sptc(2)$-invariant $\gt$-structures. Nevertheless, instability results for some critical points, as in Theorem \ref{Th: index_nullity_estimates_theorem}, can be drawn from the reduced energy functional: 

\begin{proof}[Proof of Theorem \ref{Th: index_nullity_estimates_theorem}]
   According to the correspondence between  $\frakp^{\Ad(\sptc(1))}=\frakp_1\oplus\frakp_2\oplus\frakp_3$ and $\sX(\bbS^7)^{\sptc(2)}$, the definition  of $\ind$ can be rewritten as 
\begin{multline*}
  \ind(r,h):=\sup\{\dim(F): \, F\subset \sX(\bbS^7)^{\sptc(2)} \, \text{subspace on which} \hess_{(r,h)}(E)  \text{ is negative definite}\} .
\end{multline*}
In contrast, the index of a harmonic $\sptc(2)$-invariant $\gt$-structure is given by the dimension of the largest subspace of $\sX(\bbS^7)$ on which $\hess(E)_{(r,h)}$ is negative definite. Thus, the inequality $\ind(r,h)\leq \textrm{index}(r,h)$ holds and Theorem \ref{Th: index_nullity_estimates_theorem} follows by \eqref{eq: null_index of poles} and \eqref{eq: null_index of equator}.
\end{proof}

Moreover, for $r\neq 1$ the critical points in Theorem \ref{Th: critical_points_theorem} are limits of the harmonic flow \eqref{eq: isometric_flow}, cf. Theorem \ref{th: convergent_subseq}, and their corresponding energy levels are
\begin{equation*}
    |T^\pm(r)|:=\frac{3}{r^2}(4r^6-2r^3+1) \qandq |T^{\bbS^2}(r)|=\frac{1}{r^2}(4r^6-14r^3+19).
\end{equation*}
Hence, 
$$
\begin{array}{c|c|c}
     &  |T(r)\geq \quad & \text{Minimiser of } E^\red   \\ \hline
    r<1 & |T^\pm(r)| & (r,\rN\rS)  \\ \hline
    r>1 & |T^{\bbS^2}(r)| & (r,\bbRP^2) \\
\end{array}
$$
\begin{remark}
Since our approach only deals with $\sptc(2)$-invariant $\gt$-structures, no information can be directly inferred from the stability properties of the other critical points in Theorem \ref{Th: critical_points_theorem},
for the more general variational problem of the functional $E$ defined on all isometric $\gt$-structures. 
\end{remark}

\section{Harmonicity and stability: The General Case}
\label{Sec: Harmonicity_and_stability: The General Case}

For the case of a generic $\sptc(2)$-invariant Riemannian metric on the homogeneous $7$-sphere $\sptc(2)/\sptc(1)$, take  $r_1,r_2,r_3 \in\R\smallsetminus\{0\}$, such that $r_1r_2r_3>0$, and consider the metric given by Ziller's parametrisation \eqref{Eq: Ziller_metric_form}. The computations of Lemma \ref{lemma: (f,X)_invariant} become long but it remains fairly easy to determine the associated $\bbRP^3$-family of isometric $\Ad(\sptc(1))$-invariant $\gt$-structures. For example, the coefficients of the torsion forms \eqref{eq: identities_torsion_forms} are expressed in terms of polynomials in $(h_0,h_1,h_2,h_3)\in \bbS^3$ and the parameters $r_1,r_2,r_3$.

More precisely, for such a triplet $(r_1,r_2,r_3)$, the family of isometric $\Ad(\sptc(1))$-invariant $\gt$-structures is described by
\begin{align}
\label{Eq: isometric:G2-structure_appendix}
\begin{split}
        \varphi_{(\br,h)}
        =&\;(r_1r_2r_3)^3e^{123}\\ &+\Big(\frac{r_1^2}{r_2r_3}(h_0^2+h_1^2-h_2^2-h_3^2)e^1+\frac{2r_2^2}{r_1r_3}(h_1h_2-h_0h_3)e^2+\frac{2r_3^2}{r_1r_2}(h_1h_3+h_0h_2)e^3\Big)\wedge\omega_1\\
        &+\Big(\frac{2r_1^2}{r_2r_3}(h_1h_2+h_0h_3)e^1+\frac{r_2^2}{r_1r_3}(h_0^2-h_1^2+h_2^2-h_3^2)e^2+\frac{2r_3^2}{r_1r_2}(h_2h_3-h_0h_1)e^3\Big)\wedge\omega_2\\
        &+\Big(\frac{2r_1^2}{r_2r_3}(h_1h_3-h_0h_2)e^1+\frac{2r_2^2}{r_1r_3}(h_2h_3+h_0h_1)e^2+\frac{r_3^2}{r_1r_2}(h_0^2-h_1^2-h_2^2+h_3^2)e^3\Big)\wedge \omega_3.
        \end{split}
\end{align}
           
The computation of the torsion forms and the norm of the full torsion tensor is too long to appear in print, so we implemented the expressions of Proposition \ref{prop: h_torsion_forms} on MAPLE (See Appendix \ref{Ap. Maple}).

For the torsion $0$-form, we have
\begin{align*}
    \begin{split}
        \tau_0^{(\br,h)}=&-\frac{4}{7}\Bigg[2\bigg(r_1r_2r_3^4+r_1^4r_2r_3+\frac{2}{r_1^3}+\frac{2}{r_3^3}\bigg)h_0^2+2\bigg(r_1^4r_2r_3+r_1r_2^4r_3+\frac{2}{r_2^3}+\frac{2}{r_1^3}\bigg)h_1^2\\
        &+2\bigg(r_1r_2r_3^4+r_1r_2^4r_3+\frac{2}{r_2^3}+\frac{2}{r_3^3}\bigg)h_3^2-(r_1^4r_2r_3+r_1r_2^4r_3+r_1r_2r_3^4)\\
        &-2\bigg(\frac{1}{r_1^3}+\frac{1}{r_2^3}+\frac{1}{r_3^3}\bigg)-\bigg(\frac{r_1^3}{r_2^3r_3^3}+\frac{r_2^3}{r_1^3r_3^3}+\frac{r_3^3}{r_1^3r_2^3}\bigg)\Bigg].
    \end{split}
\end{align*}

The torsion $1$-form is
\begin{align*}
    \begin{split}
        \tau_1^{(\br,h)}=&-\frac{r_1^3}{3r_2^3r_3^3}\bigg( (r_1r_2^4r_3^7+r_1r_2^7r_3^4+2r_2^3+2r_3^3)h_2h_3+(r_1r_2^4r_3^7-r_1r_2^7r_3^4+2r_2^3-2r_3^3)h_0h_1 \bigg) e^1\\
        &-\frac{r_2^3}{3r_1^3r_3^3}\bigg( (r_1^4r_2r_3^7+r_1^7r_2r_3^4+2r_1^3+2r_3^3)h_0h_2-(r_1^4r_2r_3^7-r_1^7r_2r_3^4+2r_1^3-2r_3^3)h_1h_3 \bigg) e^2\\
        &+\frac{r_3^3}{3r_1^3r_2^3}\bigg( (r_1^7r_2^4r_3+r_1^4r_2^7r_3+2r_1^3+2r_2^3)h_1h_2-(r_1^7r_2^4r_3-r_1^4r_2^7r_3-2r_1^3+2r_2^3)h_0h_3 \bigg) e^3.
    \end{split}
\end{align*}

The torsion $2$-form is
\begin{align*}
    \tau_2^{(\br,h)}= -\ast d &\psi+4\ast\tau_1\wedge\psi\nonumber\\
    =-\frac{4}{3}\Big[&(r_1^7r_2^4r_3+r_1^4r_2^7r_3-r_1^3-r_2^3)h_1h_2\Big(2e^{12}+\frac{1}{r_1^4r_2^4r_3}\omega_3\Big)\\
        -&(r_1^7r_2^4r_3-r_1^4r_2^7r_3+r_1^3-r_2^3)h_0h_3\Big(2e^{12}-\frac{1}{r_1^4r_2^4r_3}\omega_3\Big)\\
        +&(r_1^4r_2r_3^7+r_1^7r_2r_3^4-r_1^3-r_3^3)h_0h_2\Big(2e^{13}+\frac{1}{r_1^4r_2r_3^4}\omega_2\Big)\\
        -&(r_1^4r_2r_3^7-r_1^7r_2r_3^4-r_1^3+r_3^3)h_1h_3\Big(2e^{13}-\frac{1}{r_1^4r_2r_3^4}\omega_2\Big)\\
        -&(r_1r_2^7r_3^4+r_1r_2^4r_3^7-r_2^3-r_3^3)h_2h_3\Big(2e^{23}+\frac{1}{r_1r_2^4r_3^4}\omega_1\Big)\\
        +&(r_1r_2^7r_3^4-r_1r_2^4r_3^7+r_2^3-r_3^3)h_0h_1\Big(2e^{23}-\frac{1}{r_1r_2^4r_3^4}\omega_1\Big)
            \Big] .
\end{align*}
Finally, we compute the traceless $2$-symmetric tensor $\tau_{27}^{(\br,h)}=\frac{1}{4}\jmath(\tau_3^{(\br,h)})$, induced by the torsion $3$-form. The off-diagonal components are
\begin{align*}
    (\tau_{27}^{(\br,h)})_{12}
    =&\quad(r_1^4r_2^7r_3+r_1^7r_2^4r_3-2r_1^3-2r_2^3)h_0h_3+                   (r_1^4r_2^7r_3-r_1^7r_2^4r_3-2r_1^3+2r_2^3)h_1h_2 \\
    (\tau_{27}^{(\br,h)})_{13}
    =&-(r_1^4r_2r_3^7+r_1^7r_2r_3^4-2r_1^3-2r_3^3)h_1h_3+(r_1                   ^4r_2r_3^7-r_1^7r_2r_3^4-2r_1^3+2r_3^3)h_0h_2 \\
    (\tau_{27}^{(\br,h)})_{23}
    =&-(r_1r_2^7r_3^4+r_1r_2^4r_3^7-2r_2^3-2r_2^3)h_0h_1+(r_1                    r_2^7r_3^4-r_1r_2^4r_3^7+2r_1^3-2r_3^3)h_2h_3 
\end{align*}
and the diagonal components are

\begin{align*}
    (\tau_{27}^{(\br,h)})_{11}
    =&-\frac{2r_1^3}{7r_2^3r_3^3}\Big[r_2^3(r_1^4r_2r_3^7+8r_                  1^7r_2r_3^4+2r_1^3-12r_3^3)h_0^2+r_3^3(r_1^4r_2^7r_3+8r_                  1^7r_2^4r_3+2r_1^3-12r_2^3)h_1^2\\
        &+r_1^3(r_1r_2^4r_3^7+r_1r_2^7r_3^4+2r_2^3+2r_3^3)h_3^2-\frac{r_1r_2^4r_3^4}{2}\bigg(r_1^3r_2^3+r_1^3r_3^3+8r_1^6\bigg)\\
        &+4r_1^6-3r_2^6-3r_3^6-r_1^3r_2^3-r_1^3r_3^3+6r_2^3r_3^3\Big] \\
    (\tau_{27}^{(\br,h)})_{22}
    =&-\frac{2r_2^3}{7r_1^3r_3^3}\Big[r_2^3(r_1^4r_2r_3^7+r_1                  ^7r_2r_3^4+2r_1^3+2r_3^3)h_0^2+r_3^3(r_1^7r_2^4r_3+8r_1^                  4r_2^7r_3-12r_1^3+2r_2^3)h_1^2\\
        &+r_1^3(r_1r_2^4r_3^7+8r_1r_2^7r_3^4+2r_2^3-12r_3^3)h_3^2-\frac{r_1^4r_2r_3^4}{2}\bigg(r_1^3r_2^3+r_2^3r_3^3+8r_2^6\bigg)\\
        &+4r_2^6-3r_1^6-3r_3^6-r_1^3r_2^3-r_2^3r_3^3+6r_1^3r_3^3\Big] \\
    (\tau_{27}^{(\br,h)})_{33}
    =&-\frac{2r_3^3}{7r_1^3r_2^3}\Big[r_2^3(r_1^7r_2r_3^4+8r_                  1^4r_2r_3^7-12r_1^3+2r_3^3)h_0^2+r_3^3(r_1^7r_2^4r_3+r_                  1^4r_2^7r_3+2r_1^3+2r_2^3)h_1^2\\
        &+r_1^3(8r_1r_2^4r_3^7+r_1r_2^7r_3^4-12r_2^3+2r_3^3)h_3^2-\frac{r_1^4r_2^4r_3}{2}\bigg(r_1^3r_3^3+r_2^3r_3^3+8r_3^6\bigg)\\
        &+4r_3^6-3r_1^6-3r_2^6-r_1^3r_3^3-r_2^3r_3^3+6r_1^3r_2^3\Big] ;\\
\end{align*}
and, for $k=4,5,6,7$,
\begin{align*}
    (\tau_{27}^{(\br,h)})_{kk}
    =&\quad\frac{1}{7r_1^4r_2^4r_3^4}\Big[r_2^3(5r_1^4r_2r_3^                  7+5r_1^7r_2r_3^4-4r_1^3-4r_3^3)h_0^2+r_3^3(5r_1^4r_2^7r_                  3+5r_1^7r_2^4r_3-2r_1^3-2r_2^3)h_1^2\\
    &+r_1^3(5r_1r_4^4r_3^7+5r_1r_2^7r_3^4-2r_2^3-2r_3^3)h_3^2-\frac{5r_1^4r_2^4r_3^4}{2}\bigg(r_1^3+r_2^3+r_2^3\bigg)\\
        &+2(r_1^3r_2^3+r_1^3r_3^3+r_2^3r_3^3)-(r_1^6+r_2^6+r_3^6)\Big].
\end{align*}
NB.: In particular, choosing $r_1=r_2=r_3=r^{1/3}$, the above formulas coincide with Proposition \ref{prop: h_torsion_forms}. 

In the general case, the norm of the full torsion tensor in \eqref{eq: norm_torsion} becomes
\begin{align}\label{eq: general_norm_T}
    \begin{split}
        |T|^2=\rho_0h_0^2+\rho_1h_1^2+\rho_3h_3^2+\varrho ,
    \end{split}
\end{align}

where
$\rho_0 = f(r_1,r_2,r_3)$, $\rho_1= f(r_2,r_3,r_1)$ and $\rho_3= f(r_3,r_1,r_2)$ for
$$
    f(x,y,z) = 4\frac{(x+z)}{x^6y^3z^6}(x^2-xz+z^2)(x^4yz^4+2)
    (x^6+y^6+z^6-2x^3z^3-x^4y^7z^4).
$$

\subsection{Harmonic $\gt$-structures on the $7$-sphere}

We study critical points of the energy functional \eqref{eq: energy_functional} restricted to the isometric family $\varphi_{(\br,h)}$, as defined by \eqref{Eq: isometric:G2-structure_appendix}. 
Given the intricacy of formulas in the general case, it becomes even more convenient to apply Palais' Principle of Symmetric Criticality, via the alternative method described in Remark \ref{rem: alternative_critical_points_method}.  For sake of clarity, we introduce some notation for the following subsets of $\bbRP^3$:
$$  \bbS^2_{k,l,m}\simeq \{[h] \in \bbRP^3: h=(h_0,h_1,h_2,h_3)\in \sptc(1); \quad h^2_{k} + h^2_{l} + h^2_{m}= 1\}, $$
$$  \bbS^1_{k,l}\simeq \{[h] \in \bbRP^3: h=(h_0,h_1,h_2,h_3)\in \sptc(1); \quad h^2_{k} + h^2_{l} = 1\} 
\qandq $$
$$ \rN\rS_k \simeq \{[h] \in \bbRP^3: h=(h_0,h_1,h_2,h_3)\in \sptc(1); \quad h_k=\pm 1\}.
$$

As for Remark~\ref{rem: alternative_critical_points_method}, we consider a curve $k(t)$ in the space of parameters with $k(0) = 1$, so that along this curve the torsion has square norm
\begin{align*}
    \begin{split}
        |T_t|^2=&\rho_0 (h_0k_0(t) -h_1k_1(t) - h_2k_2(t) - h_3k_3(t))^2 \\
        + &\rho_1(h_0k_1(t) +h_1k_0(t) + h_3k_2(t) - h_2k_3(t))^2 \\
        +& \rho_3(h_3k_0(t) + h_2k_1(t) - h_1k_2(t) + h_0k_3(t))^2 \\
        + &\varrho ,
    \end{split}
\end{align*}
and critical points will satisfy
\begin{align*}
    \begin{split}
        0= &\rho_0h_0 ( -h_1k'_1(0) - h_2k_2'(0) - h_3k_3'(0))\\
        + &\rho_1 h_1(h_0k_1'(0) + h_3k_2'(0) - h_2k_3'(0)) \\
        +& \rho_3h_3 (h_2k_1'(0) - h_1k_2'(0) + h_0k_3'(0)),
    \end{split}
\end{align*}
or equivalently
\begin{align*}
    \begin{cases}
        \rho_0 h_0 h_2 &= h_1 h_3 (\rho_1 - \rho_3) \\
        \rho_1 h_1 h_2 &= h_3 h_0 (\rho_3 - \rho_0)\\
        \rho_3 h_3 h_2 &= h_1 h_0 (\rho_0 - \rho_1).
    \end{cases}
\end{align*}
Analysis of this condition will yield new harmonic $\Ad(\sptc(1))$-invariant $\gt$-structures.

\begin{proposition}
\label{thm critical points - gen case}
    Let $\varphi_{(\br,h)}$ denote the family of isometric $\Ad(\sptc(1))$-invariant $\gt$-structures described by \eqref{Eq: isometric:G2-structure_appendix}. According to the values of $(r_1,r_2,r_3)$, with $r_1,r_2,r_3>0$, and up to cyclic permutation on the parameters, the $\Ad(\sptc(1))$-invariant $\gt$-structures parametrized by the following sets are exactly the critical points of the energy functional $E$, i.e. harmonic $\gt$-structures:
\begin{enumerate}[(i)]
    \item
    If $r_1=r_2=r_3=r^{1/3}$, the critical sets are those described in Theorem~\ref{Th: critical_points_theorem}. 
    \item
    If $\rho_0=0$,
        $$
      \{(r_1,r_2,r_3)\}\times \bbS^1_{0,2}.
      $$
    \item
    If $\rho_1=\rho_3\neq 0$, 
    $$
      \{(r_1,r_2,r_3)\}\times \bbS^1_{1,3}.
    $$
     \item
    If $\rho_1=\rho_3= 0$, 
    $$
      \{(r_1,r_2,r_3)\}\times \bbS^2_{1,2,3}.
    $$
    \item 
    If $\rho_0=\rho_1=\rho_3\neq0$,
    $$
      \{(r_1,r_2,r_3)\}\times \bbS^2_{0,1,3}.
    $$
    \item 
    If $\rho_0=\rho_1=\rho_3=0$, then $|T|^2$ is constant on the isometric class and so the critical points are 
    $$
    \{(r_1,r_2,r_3)\}\times \bbRP^3 .
    $$
    \item For any $(r_1, r_2, r_3)$ the poles
    \begin{align*}
        \{(r_1,r_2,r_3)\}\times \rN\rS_k, \qforq k=0,1,2,3.
    \end{align*}
    are critical points.
\end{enumerate}   
\end{proposition}

Some classes of Proposition~\ref{thm critical points - gen case} are illustrated by the following examples.
\begin{example}
\label{Example 4.1}
Under the hypotheses of Proposition \ref{thm critical points - gen case}, the following hold.
\begin{enumerate}
    \item 
    Suppose $r_1=r_2$, hence $r_3>0$. Thus, $\rho_0=0$ and $\rho_1=0$ if, and only if, 
\begin{align}
\label{eq: values_of_r1}
    &r_1^8r_3=1 \qandq (r_1^9+1)(r_1^{18}-r_1^9+1)(1+2r_1^{27})(1-r_1^{27})=0,\text{ i.e,}\\ \nonumber
    & r_1\in
    \left\{\pm 1,-\frac{1}{\sqrt[27]{2}}\right\}, \quad r_1\approx -0.832039 \qorq r_1\approx 0.833359.
\end{align}
    Therefore, for the isometric classes $(r_1,r_1,r_3)$, with $r_1$ and $r_3$ solving  \eqref{eq: values_of_r1}, the reduced energy functional \eqref{eq: general_norm_T} is constant, so any $h\in \sptc(1)$ will provide a critical point of $E^{\red}$, hence a critical point of the energy functional $E$, i.e. a harmonic $\gt$-structure.
  \item 
The triple $(r_1 , r_2, r_3)=(2^{-1/27},-2^{8/27},-2^{1/27})$ illustrates a particular case of Proposition~\eqref{thm critical points - gen case} $\mathrm{(vi)}$ where the reduced energy functional \eqref{eq: general_norm_T} does not depend of $h=(h_0,h_1,h_2,h_3)$ because $\rho_0=\rho_1=\rho_3=0$. Therefore,  the isometric class  $$\{(2^{-1/27},-2^{8/27},-2^{1/27})\}\times \bbRP^3
$$ consists entirely of harmonic  $\gt$-structures.
\end{enumerate} 
\end{example}

\subsection{The $\Ad(\sptc(1))$-invariant gradient flow of $E$} 

In the search for critical points, a general expression of $\Div T$ has been made redundant by the Palais approach of the previous section. It remains nonetheless essential to obtain the equation of the harmonic flow and establish all-time existence of solutions.

The $\Ad(\sptc(1))$-invariant inner product induced by the $\gt$-structure \eqref{Eq: isometric:G2-structure_appendix} is 
\begin{equation}
\label{eq: r_123-metric}
    g_\br=r_1^6(e^1)^2+r_2^6(e^2)^2+r_3^6(e^3)^2+\frac{1}{r_1r_2r_3}((e^4)^2+(e^5)^2+(e^6)^2+(e^7)^2).
\end{equation}
The operator $U$ defined by \eqref{eq: operator_U} and associated to the Levi-Civita connection of $(\sptc(2)/\sptc(1),g_\br)$ satisfies
\begin{align}\label{eq: general_operator_U}
    \nonumber U(e_1,e_2)=&\quad \frac{(r_2^6-r_1^6)}{r_3^6}e_3, & U(e_1,e_3)=&\quad\frac{(r_1^3-r_3^6)}{r_2^6}e_2, & U(e_1,e_3)=&\quad\frac{(r_3^6-r_2^6)}{r_1^6}e_1, \\ \nonumber
    U(e_1,e_4)=&\quad\frac{(1-r_1^7r_2r_3)}{2}e_7, & U(e_2,e_4)=&-\frac{(1-r_1r_2^7r_3)}{2}e_6, & U(e_3,e_4)=&\quad\frac{(1-r_1r_2r_3^7)}{2}e_5,\\ \nonumber
    U(e_1,e_5)=&\quad\frac{(1-r_1^7r_2r_3)}{2}e_6, & U(e_2,e_5)=&\quad\frac{(1-r_1r_2^7r_3)}{2}e_7, & U(e_3,e_5)=&-\frac{(1-r_1r_2r_3^7)}{2}e_4,\\
    U(e_1,e_6)=&-\frac{(1-r_1^7r_2r_3)}{2}e_5, & U(e_2,e_6)=&\quad\frac{(1-r_1r_2^7r_3)}{2}e_4, & U(e_3,e_6)=&\quad\frac{(1-r_1r_2r_3^7)}{2}e_7,\\ \nonumber
    U(e_1,e_7)=&-\frac{(1-r_1^7r_2r_3)}{2}e_4, & U(e_2,e_7)=&-\frac{(1-r_1r_2^7r_3)}{2}e_5, & U(e_3,e_7)=&-\frac{(1-r_1r_2r_3^7)}{2}e_6. \nonumber
\end{align}

From the proof of Lemma \ref{Lema.symmetric.S}, we know that an $\Ad(\sptc(1))$-invariant symmetric $2$-tensor on $\frakp$ must be of the form $\beta=g_r([\beta]\cdot,\cdot)$, where
\begin{align*}
    [\beta]=\left(
    \begin{array}{ccc|c}
        \beta_{11} & \beta_{12} & \beta_{13} &  \\
        \beta_{12} & \beta_{22} & \beta_{23} & \\
        \beta_{13} & \beta_{23} & \beta_{33} & \\ \hline 
         & & & \beta_{44}I_{4\times 4}
    \end{array}
    \right).
\end{align*}
Using the values of $U$ on basis elements \eqref{eq: general_operator_U}, we obtain a general version of Lemma \ref{Lema.symmetric.S}. 
\begin{lemma}
    The divergence of the $\Ad(\sptc(1))$-invariant symmetric $2$-tensor $\beta$ with respect to the $\gt$-metric $g_\br$ given by Equation~\eqref{eq: r_123-metric} is
\begin{equation}
    \Div\beta
    =-\frac{r_1^6}{r_2^6r_3^6}(r_2^{12} -r_1^6r_2^6+r_1^6r_3^6-r_3^{12})\beta_{23}e_1 -\frac{r_2^6}{r_1^6r_3^6}(r_3^{12}-r_2^6r_3^6+r_1^6r_2^6-r_1^{12})\beta_{13}e_2-\frac{r_3^6}{r_1^6r_2^6}(r_1^{12} -r_1^6r_3^6+r_2^6r_3^6-r_2^{12})\beta_{12}e_3 .
\end{equation}
\end{lemma}

Using the torsion forms $\tau_1^{(\br,h)}$, $\tau_2^{(\br,h)}$ and $\tau_{27}^{(\br,h)}$, we obtain an expression for the divergence of the full torsion tensor
\begin{equation}\label{eq: general_divergence_torsion}
  \Div T(\br,h)=\underbrace{\frac{1}{2}\ast d(\tau_2^{(\br,h)}\wedge \varphi_{(\br,h)})-\ast d(\tau_1^{(\br,h)}\wedge \psi_{(\br,h)})}_{(\star)}-\Div \tau_{27}^{(\br,h)}
\end{equation}
where
\begin{align*}
  (\star) =&\quad\frac{2r_1^6}{r_2^6r_3^6}\Big[(r_1r_2^4r_3^4+2)(r_1r_2^4r_3^4-1)(r_2^3+r_3^3)h_2h_3+(r_1r_2^4r_3^4-2)(r_1r_2^4r_3^4+1)(r_2^3-r_3^3)h_0h_1 \Big]e^1\\
    &+\frac{2r_2^6}{r_1^6r_3^6}\Big[(r_1^4r_2r_3^4+2)(r_1^4r_2r_3^4-1)(r_1^3+r_3^3)h_0h_2-(r_1^4r_2r_3^4-2)(r_1^4r_2r_3^4+1)(r_1^3-r_3^3)h_1h_3 \Big]e^2\\
    &-\frac{2r_3^6}{r_1^6r_2^6}\Big[(r_1^4r_2^4r_3+2)(r_1^4r_2^4r_3-1)(r_1^3+r_2^3)h_1h_2+(r_1^4r_2^4r_3-2)(r_1^4r_2^4r_3+1)(r_1^3-r_2^3)h_0h_3 \Big]e^3.
\end{align*}

For the pair $(f,X)=(h_0, \frac{h_1}{r_1^3}e_1+\frac{h_2}{r_2^3}e_2+\frac{h_3}{r_3^3}e_3)$, Equation \eqref{eq: general_divergence_torsion} is an $\Ad(\sptc(1))$-invariant global expression of the divergence of $T$, whilst Equation~\eqref{eq: divergence_in_coordinates} is in local coordinates.
As in \S\ref{subsec: flow_ansatz_case}, take $\varphi_t=\Phi(r_1r_2r_3,\Upsilon(\overline{k(t)h}))$ to be a solution of the gradient flow \eqref{eq: isometric_flow} with initial condition $\varphi(0)=\varphi_{(\br,h)}$. Denote by  $m(t))=k(t)h=(m_0,m_1,m_2,m_3)\in\sptc(1)$ a curve in the space of parameters, and use the octonionic inner and cross products \cite{harvey1982calibrated}*{Section IV}: 
\begin{equation}\label{eq: inner_cross_quater_product}
    \langle u,v\rangle=\real(u\overline{v}) \qandq u\times v=\imag(\overline{v}u), \qforq u,v\in \mathbb{O}.
\end{equation}
Since $m(t)\in \sptc(1)\subset \H$ and $\diver T(\br,m(t))^\sharp\in \fp_1\oplus\fp_2\oplus\fp_3\simeq \imag \H$, using the identifications for the inner and cross product \eqref{eq: inner_cross_quater_product}, the system of ODE's \eqref{eq: pde_(f,X)} can be re-written as
\begin{align*}
    \odv{m_0(t)}{t}=-\frac{1}{2}\real(m(t)\diver T(\br,m(t))^\sharp) \qandq \odv{}{t}\imag(m(t))=-\frac{1}{2}\imag(m(t)\diver T(\br,m(t))^\sharp).
\end{align*}
This allows to put together the above equations into
\begin{equation}\label{eq: ode_system_general_case}
    \odv{m(t)}{t}=-\frac{1}{2}m(t)\diver T(\br,m(t))^\sharp.
\end{equation}
Remark~\ref{rem: short_time_existence_remark} on the short-time existence of solutions of the flow still applies to Equation~\eqref{eq: ode_system_general_case}. Moreover
\begin{align*}
  \odv{|m(t)|^2}{t}=&\langle m(t)\diver T(\br,m(t))^\sharp,m(t)\rangle=\real\bigg(m(t)\diver T(\br,m(t))^\sharp\overline{m(t)}\bigg)\\
    =&-\real\bigg( m(t)\overline{m(t)\diver T(\br,m(t))^\sharp}\bigg)=-\langle m(t), m(t)\diver T(\br,m(t))^\sharp\rangle=0.
\end{align*}
Thus, when the initial data is $\sptc(2)$-invariant, the isometric flow \eqref{eq: ode_system_general_case} for $\gt$-structures on the homogeneous $7$-sphere $\sptc(2)/\sptc(1)$ exists for all time (cf. the comments introducing Theorem~\ref{th: convergent_subseq}).

\subsection{Second variation}

We re-employ the method of \S\ref{Sec:2nd_variation} for the general case and analyse the second variation for some of the critical points of Proposition~\ref{thm critical points - gen case}. Let $\varphi(t,s)$ be an isometric $\Ad(\sptc(1))$-invariant variation of $\varphi_{(\br,h)}$ (cf. Equation~\eqref{Eq: isometric:G2-structure_appendix}) such that $\varphi(0,0)=\varphi_{(\br,h)}$ and
\begin{equation}\label{eq: variation_vectors_V_W_r1r2r3}
    \frac{\partial \varphi_{(t,s)}}{\partial t}|_{t=s=0}=V\lrcorner\psi_{(\br,h)} \qandq \frac{\partial \varphi_{(t,s)}}{\partial s}|_{t=s=0}=W\lrcorner\psi_{(\br,h)}.
\end{equation}
According to Theorem~\ref{Th: Phi_isomorphism_theorem}, $\varphi(t,s)=\Phi(r_1r_2r_3, \Upsilon(\overline{k(t,s)h})$ and the variation vector fields are
\begin{align}\label{eq: expression_V_W_r1r2r3}
    \begin{split}
        V=&\frac{2}{r_1^3}\frac{\partial k_1}{\partial t}(0,0)e_1+\frac{2}{r_2^3}\frac{\partial k_2}{\partial t}(0,0)e_2+\frac{2}{r_3^3}\frac{\partial k_3}{\partial t}(0,0)e_3, \\
        \text{and}\qquad\qquad\quad&\\
        W=&\frac{2}{r_1^3} \frac{\partial k_1}{\partial s}(0,0)e_1+\frac{2}{r_2^3}\frac{\partial k_2}{\partial s}(0,0)e_2+\frac{2}{r_3^3}\frac{\partial k_3}{\partial s}(0,0)e_3.
        \end{split}
    \end{align}
Taking the second derivative of $|T(\varphi(t,s))|^2$ at $t=s=0$, we have
\begin{align}\label{eq: second_variation_r1r2r3}
\begin{split}
    \hess(E)_{\br,h}(V,W) = &\frac{\partial^2}{\partial s\partial t}|_{t=s=0}|T(\varphi(t,s))|^2\\
    =& 2\rho_0\Big(\pdv{m_0}{t}(0,0)\pdv{m_0}{s}(0,0)+h_0\pdv{m_0}{t,s}(0,0) \Big)\\
    +& 2\rho_1\Big(\pdv{m_1}{t}(0,0)\pdv{m_1}{s}(0,0)+h_1\pdv{m_1}{t,s}(0,0) \Big)\\
    +& 2\rho_3\Big(\pdv{m_3}{t}(0,0)\pdv{m_3}{s}(0,0)+h_3\pdv{m_3}{t,s}(0,0) \Big),
\end{split}
\end{align}
where $m(t,s)=(m_0,m_1,m_2,m_3)=k(t,s)h$. For the critical points of Proposition~\ref{thm critical points - gen case} (iii), Equation~\eqref{eq: second_variation_r1r2r3} becomes
\begin{align}\label{eq: Hessian of crit points (iii)}
\begin{split}
    \hess(E)_{\{\br\}\times\rN\rS_0}(V,W)=&\frac{1}{2}g_\br(\diag(\rho_1-\rho_0, -\rho_0,\rho_3-\rho_0)V,W),\\
    \hess(E)_{\{\br\}\times\rN\rS_1}(V,W)=&\frac{1}{2}g_\br(\diag(\rho_0-\rho_1, \rho_3-\rho_1,-\rho_1)V,W),\\
    \hess(E)_{\{\br\}\times\rN\rS_2}(V,W)=&\frac{1}{2}g_\br(\diag(\rho_3, \rho_0,\rho_1)V,W),\\
    \hess(E)_{\{\br\}\times\rN\rS_3}(V,W)=&\frac{1}{2}g_\br(\diag(-\rho_3, \rho_1-\rho_3,\rho_0-\rho_3)V,W),
    \end{split}
\end{align}
and for the critical points of Theorem \ref{thm critical points - gen case} (ii), it becomes 
\begin{align}\label{eq: Hessian of crit points (ii)}
\begin{split}
    \hess(E)_{\{(r_1,r_3)\}\times\rN\rS_1}(V,W)=&\frac{1}{2}g_\br(\diag(\rho_0-\rho_1, \rho_0-\rho_1,-\rho_1)V,W),\\
    \hess(E)_{\{(r_1,r_3)\}\times\rN\rS_2}(V,W)=&\frac{1}{2}g_\br(\diag(\rho_0, \rho_0,\rho_1)V,W),\\
    \hess(E)_{\{(r_1,r_3)\}\times\bbS^1}(V,W)=&\frac{1}{2}g_\br(J_{(\br,h)}V,W),
    \end{split}
\end{align}
where 
\begin{equation*}
    J_{(\br,h)}=\left(\begin{array}{ccc}
       -\rho_0+h_0^2\rho_1  & \rho_1h_0h_3 & 0 \\
         \rho_1h_0h_3 & -\rho_0+h_3^2\rho_1 & 0\\
         0 & 0 & 0
    \end{array}
    \right).
\end{equation*}
The characteristic polynomial of the symmetric operator $J_{(\br,h)}$ is  $p(\lambda)=\lambda^3-(\rho_1-2\rho_0)\lambda^2-(\rho_0\rho_1-\rho_0^2)\lambda$ and its eigenvalues are $\lambda=0$, $\lambda=-\rho_0$ and $\lambda=\rho_1-\rho_0$. \\
The stability properties of Proposition \ref{thm critical points - gen case} (i) have already been included in Theorem~\ref{Th: index_nullity_estimates_theorem}.

To illustrate the stability behaviour of the energy functional, we select some specific examples of the cases (ii) and (iii) of Proposition \ref{thm critical points - gen case}. To simplify computations, we restrict ourselves to $\Ad(\sptc(1))$-invariant inner-products of $\frakp$ with unitary volume, i.e. $r_1r_2r_3=1$. The coefficients $\rho_0,\rho_1,\rho_3$, given by Equation~\eqref{eq: general_norm_T}, then simplify into
  \begin{align}\label{eq: reduced rho}
  \begin{split}
      \rho_0=&2r_3^3r_2+2r_3r_2^3+2r_1^3r_2+2r_1r_2^3+r_3^3-r_3^2r_1-r_3r_1^2-r_3r_2^2+r_1^3-r_1r_2^2-r_3r_2-r_1r_2-2r_3-2r_1,\\
      \rho_1=&2r_3^3r_1+2r_3^3r_2+2r_3r_1^3+2r_3r_2^3-r_3^2r_1-r_3^2r_2+r_1^3-r_1^2r_2-r_1r_2^2+r_2^3-r_3r_1-r_3r_2-2r_1-2r_2,\\
      \rho_3=&2r_3^3r_1+2r_3r_1^3+2r_1^3r_2+2r_1r_2^3+r_3^3-r_3^2r_2-r_3r_1^2-r_3r_2^2-r_1^2r_2+r_3^3-r_3r_1-r_1r_2-2r_3-2r_2.
      \end{split}
  \end{align}

\begin{example}\label{Th: seventh_main_theorem}
Take $(r_1,r_2,r_3)=(2,2,1/4)$, the harmonic $\gt$-structures given by the parameters $$\{(2,2,1/4)\}\times \{[h] \in \bbS^1\}, \quad  \qandq\{(2,2,1/4)\}\times
\rN\rS_2,$$
 from Equations \eqref{eq: reduced rho} we have $\rho_0=\frac{3645}{64}$ and $\rho_1=-\frac{9}{8}$. 
  Then, by the expression \eqref{eq: Hessian of crit points (ii)} of the second variation, the reduced index and nullity are
$$
      \begin{array}{c|c|c|c}
           & \ind & \nulli & E^\red- \\ \hline \{(2,2,1/4)\}\times \bbS^1 & 2 & 1 &\text{unstable} \\ \hline
           \{(2,2,1/4)\}\times\rN\rS_2 & 1 & 0 & \text{unstable} \\ \hline 
           \{(2,2,1/4)\}\times\rN\rS_1 & 0 & 0 & \text{stable} 
      \end{array}
    $$

When $(r_1 , r_2,r_3)=(2,-1/2,-1)$, the harmonic $\gt$-structures parametrized by 
$$
\{(2,-1/2,-1)\}\times \rN\rS_k \quad (i.e.\quad h_k = \pm 1) \qforq k=0,1,2,
$$ 
Equation~\eqref{eq: reduced rho} yields 
  $$
    \rho_0=0, \quad \rho_1=-\frac{99}{8} \qandq \rho_3=-\frac{135}{8},
  $$
  and the expressions for the Hessian~\eqref{eq: Hessian of crit points (iii)} forces the reduced index and nullity to be
  $$
  \begin{array}{c|c|c|c}
           & \ind & \nulli & E^\red- \\ \hline
           \{(2,-1/2,-1)\}\times \rN\rS_0 & 2 & 1 & \text{unstable} \\ \hline
           \{(2,-1/2,-1)\}\times \rN\rS_1 & 1 & 0 & \text{unstable}\\ \hline
           \{(2,-1/2,-1)\}\times \rN\rS_2 & 2 & 1 & \text{unstable} \\ \hline
           \{(2,-1/2,-1)\}\times \rN\rS_3 & 0 & 0 & \text{stable} 
      \end{array}
      $$
In all cases, $\ind$ gives lower bounds for the index of the energy functional $E$.
\end{example}

\appendix
\section{MAPLE code}
\label{Ap. Maple}
{\scriptsize

\begin{verbatim}
restart;
with(DifferentialGeometry); with(LieAlgebras); with(LinearAlgebra); with(Tensor); interface(rtablesize = 20);
dim := 7;
StructureEquations := [[e1, e2] = 2*e3, [e1, e3] = -2*e2, [e1, e4] = e7, [e1, e5] = e6, [e1, e6] = -e5,
[e1, e7] = -e4, [e2, e3] = 2*e1, [e2, e4] = -e6, [e2, e5] = e7, [e2, e6] = e4, [e2, e7] = -e5,
[e3, e4] = e5, [e3, e5] = -e4, [e3, e6] = e7, [e3, e7] = -e6, [e4, e5] = e3, [e4, e6] = -e2, 
[e4, e7] = e1, [e5, e6] = e1, [e5, e7] = e2, [e6, e7] = e3];
L := LieAlgebraData(StructureEquations, [e1, e2, e3, e4, e5, e6, e7], Alg)
DGsetup(L)
omega1 := evalDG(`&w`(theta4, theta7)+`&w`(theta5, theta6))
omega2 := evalDG(`&w`(theta4, theta6)-`&w`(theta5, theta7))
omega3 := evalDG(`&w`(theta4, theta5)+`&w`(theta6, theta7))
rels := {h0^2+h1^2+h2^2+h3^2 = 1}
`&varphi;` := evalDG(r1*r2*r3*`&w`(`&w`(theta1, theta2), theta3)
+r1^2*(h0^2+h1^2-h2^2-h3^2)*`&w`(theta1, omega1)/(r2*r3)
+2*r2^2*(-h0*h3+h1*h2)*`&w`(theta2, omega1)/(r1*r3)+2*r3^2*(h0*h2+h1*h3)*`&w`(theta3, omega1)/(r1*r2)
+2*r1^2*(h0*h3+h1*h2)*`&w`(theta1, omega2)/(r2*r3)
+2*r2^2*(h0^2-h1^2+h2^2-h3^2)*`&w`(theta2, omega2)/(r1*r3)
+2*r3^2*(-h0*h1+h2*h3)*`&w`(theta3, omega2)/(r1*r2)
+2*r1^2*(-h0*h2+h1*h3)*`&w`(theta1, omega3)/(r2*r3)
+2*r2^2*(h0*h1+h2*h3)*`&w`(theta2, omega3)/(r1*r3)
+(r3^2*(1/(r2*r3)))*(h0^2-h1^2-h2^2+h3^2)*`&w`(theta3, omega3))
g := evalDG(r1^6*`&t`(theta1, theta1)+r2^6*`&t`(theta2, theta2)+r3^6*`&t`(theta3, theta3)
+(`&t`(theta4, theta4)+`&t`(theta5, theta5)+`&t`(theta6, theta6)+`&t`(theta7, theta7))/(r1*r2*r3))    
psi := HodgeStar(g, `&varphi;`)
`d&varphi;` := evalDG(ExteriorDerivative(`&varphi;`))
`d&psi;` := evalDG(ExteriorDerivative(psi))
T_0 := evalDG((1/7)*HodgeStar(g, `&w`(`&varphi;`, `d&varphi;`)))
T_1 := evalDG((1/12)*HodgeStar(g, `&w`(`&varphi;`, `stard&varphi;`)))
T_3 := evalDG(`stard&varphi;`-3*HodgeStar(g, `&w`(T_1, `&varphi;`))-T_0*`&varphi;`)
`stard&psi;` := evalDG(HodgeStar(g, `d&psi;`))
T_2 := evalDG(-`stard&psi;`+4*HodgeStar(g, `&w`(T_1, psi)))
divT_1 := HodgeStar(g, ExteriorDerivative(`&w`(T_1, psi)))
divT_2 := -HodgeStar(g, ExteriorDerivative(`&w`(T_2, `&varphi;`)))
norm_tau2 := TensorInnerProduct(g, T_2, T_2)
norm_tau0 := TensorInnerProduct(g, (1/4)*T_0*g, (1/4)*T_0*g)
`&varphi;1` := evalDG(Hook(e1, `&varphi;`))
`&varphi;2` := evalDG(Hook(e2, `&varphi;`))
`&varphi;3` := evalDG(Hook(e3, `&varphi;`))
`&varphi;4` := evalDG(Hook(e4, `&varphi;`))
`&varphi;5` := evalDG(Hook(e5, `&varphi;`))
`&varphi;6` := evalDG(Hook(e6, `&varphi;`))
`&varphi;7` := evalDG(Hook(e7, `&varphi;`))
G11 := HodgeStar(g, `&wedge`(`&wedge`(`&varphi;1`, `&varphi;1`), T_3))
G12 := HodgeStar(g, `&wedge`(`&wedge`(`&varphi;1`, `&varphi;2`), T_3))
G13 := HodgeStar(g, `&wedge`(`&wedge`(`&varphi;1`, `&varphi;3`), T_3))
G22 := HodgeStar(g, `&wedge`(`&wedge`(`&varphi;2`, `&varphi;2`), T_3))
G23 := HodgeStar(g, `&wedge`(`&wedge`(`&varphi;2`, `&varphi;3`), T_3))
G33 := HodgeStar(g, `&wedge`(`&wedge`(`&varphi;3`, `&varphi;3`), T_3))
G44 := HodgeStar(g, `&wedge`(`&wedge`(`&varphi;4`, `&varphi;4`), T_3))
G55 := HodgeStar(g, `&wedge`(`&wedge`(`&varphi;5`, `&varphi;5`), T_3))
G66 := HodgeStar(g, `&wedge`(`&wedge`(`&varphi;6`, `&varphi;6`), T_3))
G77 := HodgeStar(g, `&wedge`(`&wedge`(`&varphi;7`, `&varphi;7`), T_3))
tau27 := evalDG((1/4)*G11*`&t`(theta1, theta1)+(1/4)*G12*(`&t`(theta1, theta2)+`
&t`(theta2, theta1))+(1/4)*G13*(`&t`(theta1, theta3)+
`&t`(theta3, theta1))+(1/4)*G22*`&t`(theta2, theta2)+
(1/4)*G23*(`&t`(theta2, theta3)+`&t`(theta3, theta2))+
(1/4)*G33*`&t`(theta3, theta3)
+(1/4)*G44*(`&t`(theta4, theta4)+`&t`(theta5, theta5)+`&t`(theta6, theta6)+`&t`(theta7, theta7)))
trtau27 := TensorInnerProduct(g, tau27, g)
simplify(trtau27, rels)
norm_tau27 := TensorInnerProduct(g, tau27, tau27)
simplify(norm_tau27, rels)
ST := evalDG((1/4)*T_0*g-tau27)
CT := evalDG(-(1/2)*T_2-HodgeStar(g, `&w`(T_1, psi)))
normT := TensorInnerProduct(g, ST, ST)+TensorInnerProduct(g, CT, CT)
simplify(normT, rels)
norm_CT := TensorInnerProduct(g, CT, CT)
tensorT_1 := evalDG(`&w`(T_1, psi))
normT_1 := TensorInnerProduct(g, tensorT_1, tensorT_1)
simplify(normT_1+(1/4)*norm_tau2-norm_CT, rels)
dualtensorT_1 := HodgeStar(g, tensorT_1)
TensorInnerProduct(g, dualtensorT_1, T_2)
simplify(TensorInnerProduct(g, dualtensorT_1, dualtensorT_1)+(1/4)*norm_tau2-norm_CT, rels)
TensorInnerProduct(g, dualtensorT_1, dualtensorT_1)+(1/4)*norm_tau2
\end{verbatim}

}

\section{Standard computations}\label{Standard computations}

To compute the terms $d\varphi_{r}$ and $d\psi_{r}$, we use the exterior derivatives \eqref{Eq:Differential.1.form}:
\begin{align}\label{Eq:Differential.2.form.without.Omega}
    d(e^{123})=&-e^{23}\wedge\omega_1-e^{13}\wedge\omega_2-e^{12}\wedge\omega_1, & d(e^2\wedge\omega_1)=&\quad2e^{13}\wedge\omega_1+2e^{23}\wedge\omega_2,\\
d(e^1\wedge\omega_1)=&-2e^{23}\wedge\omega_1+2e^{13}\wedge\omega_2+2e^{12}\wedge\omega_3-\omega_1^2, &
d(e^3\wedge\omega_1)=&-2e^{12}\wedge\omega_1-2e^{23}\wedge\omega_3, \nonumber\\
d(e^2\wedge\omega_2)=&-2e^{23}\wedge\omega_1+2e^{13}\wedge\omega_2-2e^{12}\wedge\omega_3+\omega_2^2, & d(e^1\wedge\omega_2)=&-2e^{23}\wedge\omega_2-2e^{13}\wedge\omega_1,\nonumber\\
d(e^3\wedge\omega_3)=&\quad2e^{23}\wedge\omega_1+2e^{13}\wedge\omega_2-2e^{12}\wedge\omega_3-\omega_3^2, & d(e^3\wedge\omega_2)=&-2e^{12}\wedge\omega_2-2e^{13}\wedge\omega_3,\nonumber\\
d(e^1\wedge\omega_3)=&-2e^{23}\wedge\omega_3-2e^{12}\wedge\omega_1, &
d(e^2\wedge\omega_3)=&\quad2e^{13}\wedge\omega_3+2e^{12}\wedge\omega_2.\nonumber
\end{align}
Similarly, 
\begin{align}
\label{Eq:Differential.2.form.with.Omega}
    d(e^{23}\wedge\omega_1) =&0, & d(e^{13}\wedge\omega_1) =&-e^3\wedge\omega_1^2+2e^{123}\wedge\omega_3,\nonumber\\
    d(e^{12}\wedge\omega_1) =& -e^2\wedge\omega_1^2-2e^{123}\wedge\omega_2, &
d(e^{23}\wedge\omega_2) =& e^3\wedge\omega_1^2-2e^{123}\wedge\omega_3,\nonumber\\
    d(e^{13}\wedge\omega_2) =& 0, &
d(e^{12}\wedge\omega_2) =& -e^1\wedge\omega_2^2+2e^{123}\wedge\omega_1, \\
    d(e^{23}\wedge\omega_3) =& e^2\wedge\omega_3^2+2e^{123}\wedge\omega_2, &
d(e^{13}\wedge\omega_3) =& e^1\wedge\omega_3^2-2e^{123}\wedge\omega_1,\nonumber\\
      d(e^{12}\wedge\omega_3) =& 0 .\nonumber
\end{align}

\begin{lemma}
\label{lemma: j_tau_3}
    Given $h\in \sptc(1)$, consider the linear map defined by 
$$
A(x)=hx\bar{h},
\qwithq 
x\in \frakp_1\oplus\frakp_2\oplus\frakp_3,
$$
    and denote by $(A_{ab})$ its induced matrix, in the basis $\{e_1,e_2,e_3\}$ from \eqref{Eq:sp(2)_basis}.  
    Thus, the image of $e^{123}\in \Lambda^3(\frakp_1\oplus\frakp_2\oplus\frakp_3)^\ast$ and $e^m\wedge\omega_n\in \Lambda^3\frakp^\ast$ for $m,n=1,2,3$, under the operator $\jmath$  are
\begin{eqnarray*}
    \jmath(e^{123})_{ab} &=&\frac{1}{r}g_r(e_a,e_b),
    \qforq a,b=1,2,3,\\ \jmath(e^m\wedge\omega_n)_{ab}&=& 2r(A_{bn}\delta_{ma}+A_{an}\delta_{nb}),
    \qforq a,b=1,2,3,\\
  \jmath(e^m\wedge\omega_n)_{ab}&=&-r^2A_{pq}A_{uv}\varepsilon_{mpu}\varepsilon_{nqv}g_r(e_a,e_b), \qforq a,b=4,5,6,7,\\
    \jmath(e^{123})_{ab}=\jmath(e^m\wedge\omega_n)_{ab}&=&0
    \quad\text{otherwise},
  \end{eqnarray*}
  where $\varepsilon_{pqn}=\pm 1$ is the sign of the permutation $(p,q,n)\sim(1,2,3)$.
\end{lemma}

\begin{proof}
For the $3$-form $e^{123}$, we have
 $$
  \jmath(e^{123})_{ab}=\frac{2}{r}g_r(A(e_a),A(e_b)), \qforq a,b=1,2,3.
$$
Notice also that
$$
  g_r(A(e_a),A(e_b)) = r^2\langle h e_a\bar{h},h e_b\bar{h}\rangle=r^2\langle e_a,e_b\rangle=g_r(e_a,e_b).
$$
For $b=4,5,6,7$, the  $7$-form
$$
  e_a\lrcorner\varphi_r\wedge e_b\lrcorner\varphi_r\wedge e^{123}=\sum_{m=1}^3e_a\lrcorner\varphi_r\wedge (e_b\lrcorner\omega_m)\wedge A^\ast e^m\wedge e^{123},
$$
is degenerate, since $A^\ast e^m\wedge e^{123}$ is a $4$-form in $\frakp_1\oplus\frakp_2\oplus\frakp_3$. Therefore, 
\begin{align*}
    \jmath(e^{123})_{ab}
    =\begin{cases}
        \frac{1}{r^3}g_r(e_a,e_b), &a,b \in \{1,2,3\}
        \\
        \quad 0, &\text{otherwise}
    \end{cases}.
\end{align*}
On the other hand, if $\alpha\in \Lambda^k(\frakp_1\oplus\frakp_2\oplus\frakp_3)^\ast$, then $\ast(\alpha\wedge\omega_n^2)=2\ast_3\alpha$, with respect to  the Hodge star operator $\ast_3$ induced by $g_r|_{\frakp_1\oplus\frakp_2\oplus\frakp_3}$. For each $a,b=1,2,3$, we have
\begin{align*}
    \jmath(e^m\wedge\omega_n)_{ab}
    &= \ast(e_a\lrcorner\varphi_r\wedge e_b\lrcorner\varphi_r\wedge e^m\wedge\omega_n)\\
    &= r^4A_{bq}\ast(e_a\lrcorner e^{123}\wedge e^m\wedge\omega_q\wedge\omega_n)+r^4A_{ap}\wedge\ast(e_b\lrcorner e^{123}\wedge e^m\wedge\omega_p\wedge\omega_n)\\
    &= 2r^4(A_{bn}e^m(e_a)+A_{an}e^m(e_b))\ast_3(e^{123})\\
    &= 2r(A_{bn}\delta_{ma}+A_{an}\delta_{mb}).
\end{align*}
Now, for $a,b=4,5,6,7,$  the $\SU(2)$-structure $(\omega_1,\omega_2,\omega_3)$ satisfies
\begin{align*}
    e_a\lrcorner\omega_p\wedge e_b\lrcorner\omega_q\wedge\omega_n=\begin{cases}
        \varepsilon_{pqn}g_r(e_a,e_b)e^{4567}, 
        &\text{if $p\neq q\neq n\neq p$}\\
        \qquad 0, &\text{otherwise}
    \end{cases},
\end{align*}
Then we obtain
\begin{align*}
    \jmath(e^m\wedge\omega_n)_{ab}
    &= \ast(e_a\lrcorner\varphi_r\wedge e_b\lrcorner\varphi_r\wedge e^m\wedge\omega_n)\\
    &= r^2A_{pq}A_{uv}\ast(e^p\wedge(e_a\lrcorner \omega_q)\wedge e^u\wedge(e_b\lrcorner \omega_v)\wedge e^m\wedge\omega_n)\\
    &= -r^2A_{pq}A_{uv}\varepsilon_{qvn}g_r(e_a,e_b)\ast_3(e^{mpu}) \qforq m,p,q,u,v\in\{1,2,3\}.
 \end{align*}
 Finally, if $a=1,2,3$ and $b=4,5,6,7$, we have
 $$
   e_a\lrcorner\varphi_r=r^3e^{pq}+\frac{1}{r^2}\sum_{c=1}^3A_{ca}\omega_c \qandq e_b\lrcorner\varphi_r=-\sum_{d=1}^3A^\ast e^d\wedge(e_b\lrcorner\omega_d),
 $$
 where $p,q\in \{1,2,3\}$ are such that $(a,p,q)\sim(1,2,3)$ is an even permutation. Notice that the wedge product $(e_a\lrcorner\varphi_r)\wedge (e_b\lrcorner\varphi_r)\wedge e^m\wedge\omega_n$ is degenerate:
 \begin{align*}
   (e_a\lrcorner\varphi_r)\wedge (e_b\lrcorner\varphi_r)\wedge e^m\wedge\omega_n
   =&\;r \sum_{d=1}^3 e^{mpq}\wedge A^\ast e^d\wedge(e_b\lrcorner\omega_d)\wedge\omega_n\\
   &-\frac{1}{r^2}\sum_{c=1}^3A_{ca}e^{m}\wedge A^\ast e^d\wedge(e_b\lrcorner\omega_d)\wedge\omega_c\wedge\omega_n\\
   =&\;0,
 \end{align*}
 since $e^{mpq}\wedge A^\ast e^d$ is a $4$-form on $\frakp_1\oplus\frakp_2\oplus\frakp_3$ and $(e_b\lrcorner\omega_d)\wedge\omega_c\wedge\omega_n$ is a $5$-form on $\frakp_4$. We conclude that $\jmath(e^m\wedge\omega_n)_{ab}=0$, for $a=1,2,3$ and $b=4,5,6,7$.
\end{proof}

\bibliography{biblio.bib}


\end{document}